\documentclass[11pt]{amsart}
\usepackage[T1]{fontenc}
\usepackage[latin1]{inputenc}
\usepackage{a4}
\usepackage{setspace}
\usepackage{xypic}
\usepackage{amssymb}
\usepackage{amsthm}
\usepackage[english]{babel}
\usepackage{latexsym}
\date{}

\usepackage{srcltx}
\newtheorem{thm}{Theorem}[section]
\newtheorem*{thm*}{Theorem}

\newtheorem{defn}[thm]{Definition}

\newtheorem{Notation}[thm]{Notation}
\newtheorem{choices}[thm]{Data}

 \newtheorem{rem}[thm]{Remark}
 \newtheorem{prop}[thm]{Proposition}
 \newtheorem{lem}[thm]{Lemma}
 \newtheorem{cor}[thm]{Corollary}

\newcommand{\fonc}[5]{
 \begin{array}{cccc}
 #1: & #2 & \longrightarrow & #3\\
     & #4 & \longmapsto & #5
 \end{array}
}
\newcommand{\appl}[4]{
 \begin{array}{cccc}
   #1 & \longrightarrow & #2\\
   #3 & \longmapsto & #4
 \end{array}
}
\def\wdar[#1]{\ar@<4pt>[#1]\ar@<-4pt>[#1]}
\def\tar[#1]{\ar@<0pt>[#1]\ar@<4pt>[#1]\ar@<-4pt>[#1]}
\entrymodifiers={!!<0pt,0.7ex>+}

\begin{document}

\title[The Weil-étale fundamental group II]{The Weil-étale fundamental group of a number field II}
\author{Baptiste Morin}
\maketitle

\begin{abstract}
We define the fundamental group underlying to Lichtenbaum's Weil-\'etale cohomology for number rings. To this aim, we define the Weil-\'etale topos as a refinement of the Weil-\'etale sites introduced in \cite{Lichtenbaum}. We show that the (small) Weil-\'etale topos of a smooth projective curve defined in this paper is equivalent to the natural definition given in \cite{Lichtenbaum-finite-field}. Then we compute the Weil-\'etale fundamental group of an open subscheme of the spectrum of a number ring. Our fundamental group is a projective system of locally compact topological groups, which represents first degree cohomology with coefficients in locally compact abelian groups. We apply this result to compute the Weil-étale cohomology in low degrees and to prove that the Weil-étale topos of a number ring satisfies the expected properties of the conjectural Lichtenbaum topos.
\end{abstract}

\footnotetext{ 
cohomology, Weil-étale cohomology, topos, fundamental group, Dedekind zeta function.}

\section{Introduction}
Lichtenbaum has defined in \cite{Lichtenbaum} the Weil-étale cohomology of a number ring $X=Spec(\mathcal{O}_F)$. He has shown that the resulting cohomology groups with compact support $H_{Wc}^i(X,\mathbb{Z})$ for $i\leq3$ are related to the special value of the Dedekind zeta function $\zeta_F(s)$ at $s=0$. In this paper, we refine Lichtenbaum's construction in order to define and compute the Weil-étale fundamental group.

As observed in \cite{On the WE}, the Weil-étale cohomology introduced in \cite{Lichtenbaum} is not defined as the cohomology of a Grothendieck site (i.e. of a topos). More precisely, Lichtenbaum defined in \cite{Lichtenbaum} a family of sites $T_{F/K,S}$ for any finite Galois extension $K/F$ and any suitable finite set $S$ of primes of $F$. Then he defined the Weil-\'etale cohomology as the direct limit $\underrightarrow{lim}\,H^*(T_{L/K,S},-)$. In this paper, we define a single Weil-étale topos $\bar{X}_W$  which recovers Lichtenbaum's computations. Here $\bar{X}$ denotes the Arakelov compactification of $X=Spec(\mathcal{O}_F)$. The topos $\bar{X}_W$ is endowed with a morphism to the Artin-Verdier \'etale topos $\bar{X}_{et}$. This point of view has some technical advantages. For example, the same definition is used in \cite{Flach-moi} to define the Weil-\'etale topos of an arithmetic scheme as a fiber product.

Motivated by a question asked by Lichtenbaum (see the introduction of \cite{Lichtenbaum}), we show in Section \ref{subsect-WET-fction-field-case} that our definition of the (small) Weil-\'etale topos of a function field is equivalent to the natural definition given in \cite{Lichtenbaum-finite-field}. The same result is actually false with the original definition of \cite{Lichtenbaum}. More precisely, let $Y$ be an open subscheme of a smooth projective curve over a finite field $k$, and denote by $\mathcal{S}_{et}(W_k,\overline{Y})$ the topos of $W_k$-equivariant \'etale sheaves on the geometric curve $\overline{Y}=Y\otimes_k\overline{k}$.
\begin{thm}
There is an equivalence
$$Y^{sm}_W\simeq\mathcal{S}_{et}(W_k,\overline{Y})$$
where $Y^{sm}_W$ is the (small) Weil-\'etale topos defined in this paper.
\end{thm}

Section \ref{sect-fund-group} is devoted to the computation of the Weil-étale fundamental group. Let $\bar{U}$ be a connected \'etale $\bar{X}$-scheme. We define the Weil-\'etale topos of $\bar{U}$ as the slice topos $\bar{U}_W:=\bar{X}_W/\gamma^*\bar{U}$. Let $K$ be the number field corresponding to the generic point of $\bar{U}$, and let $q_{\bar{U}}:Spec(\overline{K})\rightarrow\bar{U}$ be a geometric point. The \'etale fundamental group $\pi_1(\bar{U}_{et},q_{\bar{U}})$ is a (strict) projective system of finite quotients of the Galois group $G_K$. Replacing Galois groups with Weil-groups, we define the analogous (strict) projective system $\underline{W}(\bar{U},q_{\bar{U}})$ of locally compact quotients of the Weil group $W_K$. The following theorem gives a computation of the fundamental group of $\bar{U}_W$.
\begin{thm}
The Weil-\'etale topos  $\bar{U}_W$ is connected and locally connected over the topos $\mathcal{T}$ of locally compact spaces. The geometric point $q_{\bar{U}}$ defines a $\mathcal{T}$-valued point $p_{\bar{U}}$ of the topos $\bar{U}_W$, and we have an isomorphism
$$\pi_1(\bar{U}_W,p_{\bar{U}})\simeq\underline{W}(\bar{U},q_{\bar{U}})$$
of topological pro-groups.
\end{thm}
The consequences of this result are given in Section \ref{sect-consequences}. We denote by $C_{\bar{U}}:=C_{K,S}$ the $S$-id\`ele class group associated to $\bar{U}$ (here $S$ is the set of places of $K$ not corresponding to a point of $\bar{U}$).
\begin{cor}
For any connected \'etale $\bar{X}$-scheme $\bar{U}$, we have an isomorphism of topological groups $\pi_1(\bar{U}_W,p_{\bar{U}})^{ab}\simeq C_{\bar{U}}$. In particular, for any locally compact abelian group $A$, we have
$$H^1(\bar{U}_W,A)=Hom_{cont}(C_{\bar{U}},A)$$
\end{cor}
In particular $\pi_1(\bar{X}_{W})^{ab}$ is topologically isomorphic to the Arakelov Picard group $Pic(\bar{X})$ of the number field $F$, and the canonical class is the canonical continuous morphism
$$\theta\in H^1(\bar{X}_W,\widetilde{\mathbb{R}})=Hom_{cont}(Pic(\bar{X}),\mathbb{R}).$$
The previous corollary allows one to compute the cohomology of the Weil-étale topos in low degrees and to recover Lichtenbaum's computations.

Matthias Flach has shown in \cite{MatFlach} that the current definition of the Weil-étale cohomology is not yet the right one. More precisely, the groups $H_{W}^i(\bar{X},\mathbb{Z})$ are infinitely generated for any $i\geq4$ even. But the conjectural picture still stands. Indeed, Lichtenbaum conjectures in \cite{Lichtenbaum} the existence of a Grothendieck topology for an arithmetic scheme $X$ such that the Euler characteristic of the cohomology groups of the constant sheaf $\mathbb{Z}$ with compact support at infinity gives, up to sign, the leading term of the zeta-function $\zeta_X(s)$ at $s=0$. In \cite{Fund-group-I} we gave a list of axioms that should be satisfied by the category of sheaves on this conjectural Grothendieck topology for $X=Spec(\mathcal{O}_F)$. We denote by $\bar{X}_L$ this conjectural category of sheaves, and we refer to the list of axioms that must be satisfied by $\bar{X}_L$ as Axioms $(1)-(9)$. We also showed in \cite{Fund-group-I} that any topos satisfying these axioms gives rise to complexes of \'etale sheaves computing the expected Lichtenbaum cohomology. The author's main motivation for the present work was to provide an example of a topos satisfying Axioms $(1)-(9)$.
\begin{cor}
The Weil-\'etale topos $\bar{X}_W$ satisfies Axioms $(1)-(9)$.
\end{cor}
This result shows that Axioms $(1)-(9)$ are consistent. Moreover, it gives a natural computation of the base change from the Weil-étale cohomology to the étale cohomology (see Corollary \ref{thm-Lichtenbaum-conj}). More precisely, let $\gamma:\bar{X}_W\rightarrow\bar{X}_{et}$ be the canonical map, and let $\varphi:X_W\rightarrow\bar{X}_W$ be the open embedding. For any abelian sheaf $\mathcal{A}$, we denote by $\tau_{\leq2}R\gamma_*\mathcal{A}$ the truncated complex.
\begin{cor}
Assume that $F$ is totally imaginary. Then the Euler characteristic of the hypercohomology groups of the complex of étale sheaves $\tau_{\leq2}R\gamma_*(\varphi_!\mathbb{Z})$ gives, up to sign,  the leading term of the Dedekind zeta-function $\zeta_F(s)$ at $s=0$.
\end{cor}

\footnotetext{ \emph{2000 Mathematics subject classification} :
14F20 (primary) 14G10 (secondary). \emph{Keywords} : étale
cohomology, Weil-étale cohomology, topos, fundamental group, Dedekind zeta function.}

$\mathbf{Acknowledgments}.$ I am very grateful to Matthias
Flach for his comments and for many stimulating discussions related to Weil-étale cohomology.

\tableofcontents

\section{Preliminaries}

\subsection{Left exact sites}
The category of sheaves of sets on a Grothendieck site $(\mathcal{C},\mathcal{J})$ is denoted by $\widetilde{(\mathcal{C},\mathcal{J})}$ while the category of presheaves on $\mathcal{C}$ is denoted by $\widehat{\mathcal{C}}$. A Grothendieck topology $\mathcal{J}$ on a category $\mathcal{C}$ is said to be \emph{sub-canonical} if $\mathcal{J}$ is
coarser than the canonical topology. This is the case precisely when
any representable presheaf on $\mathcal{C}$ is a sheaf for the
topology $\mathcal{J}$. A family of morphisms $\{X_i\rightarrow X\}$ in $\mathcal{C}$ is said to be a \emph{covering family} for the topology $\mathcal{J}$ when the sieve generated by this family of morphisms belongs to $\mathcal{J}(X)$. A category $\mathcal{C}$ is said to be \emph{left exact} when finite projective limits exist in
$\mathcal{C}$, i.e. when $\mathcal{C}$ has a final object and fiber
products. A functor between left exact categories is said to be left exact
if it commutes with finite projective limits.
\begin{defn}
A Grothendieck site $(\mathcal{C},\mathcal{J})$ is said to be \emph{left
exact} if $\mathcal{C}$ is a left exact category endowed with a
subcanonical topology $\mathcal{J}$. A \emph{morphism of left exact sites}
$(\mathcal{C}',\mathcal{J}')\rightarrow (\mathcal{C},\mathcal{J})$
is a continuous left exact functor
$\mathcal{C}'\rightarrow\mathcal{C}$.
\end{defn}
Note that any Grothendieck topos, i.e. any category satisfying Giraud's axioms, is equivalent to the category of sheaves of sets on a left exact site. Note also that a Grothendieck site $(\mathcal{C},\mathcal{J})$ is left exact if and only if the canonical functor (given in general by Yoneda and sheafification) $y:\mathcal{C}\rightarrow\widetilde{(\mathcal{C},\mathcal{J})}$
identifies $\mathcal{C}$ with a left exact full subcategory of $\widetilde{(\mathcal{C},\mathcal{J})}$. A \emph{morphism of left exact sites}
$f^*:(\mathcal{C}',\mathcal{J}')\rightarrow
(\mathcal{C},\mathcal{J})$ induces a morphism of topoi
$f:(\widetilde{\mathcal{C},\mathcal{J}})\rightarrow
(\widetilde{\mathcal{C}',\mathcal{J}'})$, such that the following diagram is
commutative
\[ \xymatrix{
(\widetilde{\mathcal{C},\mathcal{J}})& (\widetilde{\mathcal{C}',\mathcal{J}'})\ar[l]_{f^*}\\
\mathcal{C}\ar[u]&\mathcal{C}'\ar[l]_{f^*}\ar[u] }\]
where the vertical arrows are the fully faithful Yoneda functors.

Finally, recall that for any object $X$ of $\mathcal{C}$, one has a canonical equivalence
$$\widetilde{(\mathcal{C},\mathcal{J})}/yX\simeq\widetilde{(\mathcal{C}/X,\mathcal{J}_{ind})}$$
where $\mathcal{J}_{ind}$ is the topology on $\mathcal{C}/X$ induced by  $\mathcal{J}$ via the forgetful functor $\mathcal{C}/X\rightarrow\mathcal{C}$ (forget the map to $X$).

\subsection{Basic properties of geometric morphisms}

Let $\mathcal{S}$ and $\mathcal{S}'$ be two Grothendieck topoi. A \emph{(geometric) morphism of topoi}
$$f:=(f^*,f_*):\mathcal{S}'\longrightarrow\mathcal{S}$$
is defined as a pair of functors $(f^*,f_*)$, where $f^*:\mathcal{S}\rightarrow \mathcal{S}'$ is left adjoint to $f_*:\mathcal{S}'\rightarrow \mathcal{S}$ and $f^*$ is left exact (i.e. $f^*$ commutes with finite projective limits). One can also define such a morphism as a left exact functor $f^*:\mathcal{S}\rightarrow \mathcal{S}'$ commuting with arbitrary inductive limits. Indeed, in this case $f^*$ has a uniquely determined right adjoint $f_*$.

If $X$ is an object of $\mathcal{S}$, then the slice category $\mathcal{S}/X$, of objects of $\mathcal{S}$ over $X$, is a topos as well. The base change functor
$$\appl{\mathcal{S}}{\mathcal{S}/X}{Y}{Y\times X}$$
is left exact and commutes with arbitrary inductive limits, since inductive limits are universal in a topos. We obtain a morphism
$$\mathcal{S}/X\longrightarrow \mathcal{S}.$$
Such a morphism is said to be a \emph{localization morphism} or a \emph{local homeomorphism} (the term local homeomorphism is inspired by the case when $\mathcal{S}$ is the topos of sheaves on some topological space). For any morphism $f:\mathcal{S}'\rightarrow\mathcal{S}$ and any object $X$ of $\mathcal{S}$, there is a natural morphism
$$f_{/X}:\mathcal{S}'/f^*X\longrightarrow\mathcal{S}/X.$$
The functor $f_{/X}^*$ is defined in the obvious way: $f^*_{/X}(Y\rightarrow X)=(f^*Y\rightarrow f^*X)$. The direct image functor $f_{/X,*}$ sends $Z\rightarrow f^*X$ to $f_*Z\times_{f_*f^*X}X\rightarrow X$, where $X\rightarrow f_*f^*X$ is the adjunction map. The morphism $f_{/X}$ is a pull-back of $f$, in the sense that the square
\[\xymatrix{
\mathcal{S}'/f^*X\ar[r]^{f_{/X}}\ar[d]&\mathcal{S}/X\ar[d]\\
\mathcal{S}'\ar[r]^{f}&\mathcal{S}
}\]
is commutative and 2-cartesian. In other words, the 2-fiber product $\mathcal{S}'\times_{\mathcal{S}}\mathcal{S}/X$ can be defined as the slice topos $\mathcal{S}'/f^*X$.

A morphism $f:\mathcal{S}'\rightarrow\mathcal{S}$ is said to be \emph{connected} if $f^*$ is fully faithful. It is \emph{locally connected} if $f^*$ has an $\mathcal{S}$-indexed left adjoint $f_!$ (see \cite{elephant}
C3.3). These definitions generalize the usual ones for topological spaces: if $T$ is a topological space, consider the unique morphism $Sh(T)\rightarrow\underline{Sets}$ where $Sh(T)$ is the category of \'etal\'e spaces over $T$. For example a localization morphism $\mathcal{S}/X\rightarrow \mathcal{S}$ is always locally connected (here $f_!(Y\rightarrow X)=Y$), but is connected if and only if $X$ is the final object of $\mathcal{S}$.

A morphism $f:\mathcal{S}'\rightarrow\mathcal{S}$ is said to be an \emph{embedding} when $f_*$ is fully faithful. It is an \emph{open embedding} if $f$ factors through $f:\mathcal{S}'\simeq\mathcal{S}/X\rightarrow\mathcal{S}$, where $X$ is a subobject of the final object of $\mathcal{S}$. Then the essential image $\mathcal{U}$ of the functor $f_*$ is said to be an \emph{open subtopos} of $\mathcal{S}$. The \emph{closed complement} $\mathcal{F}$ of $\mathcal{U}$ is the strictly full subcategory of $\mathcal{S}$ consisting in objects $Y$ such that $Y\times X$ is the final object of $\mathcal{U}$ (i.e. $f^*Y$ is the final object of $\mathcal{S}'$). A \emph{closed subtopos} $\mathcal{F}$ of $\mathcal{S}$ is a strictly full subcategory which is the closed complement of an open subtopos. A morphism of topoi $i:\mathcal{E}\rightarrow\mathcal{S}$ is said to be a \emph{closed embedding} if $i$ factors through $i:\mathcal{E}\simeq\mathcal{F}\rightarrow\mathcal{S}$ where $\mathcal{F}$ is a closed subtopos of $\mathcal{S}$.

A \emph{subtopos} of $\mathcal{S}$ is a strictly full subcategory $\mathcal{S}'$ of $\mathcal{S}$ such that the inclusion functor $i:\mathcal{S}'\hookrightarrow\mathcal{S}$ is the direct image of a morphism of topoi (i.e. $i$ has a left exact left adjoint). A morphism $f:\mathcal{S}'\rightarrow\mathcal{S}$ is said to be \emph{surjective} if $f^*$ is faithful. Any morphism $f:\mathcal{E}\rightarrow\mathcal{S}$ can be decomposed as a surjection $\mathcal{E}\rightarrow Im(f)$ followed by an embedding $Im(f)\rightarrow\mathcal{S}$, where $Im(f)$ is a subtopos of $\mathcal{S}$, which is called the \emph{image of $f$} (see \cite{SGA4} IV. 9.1.7.2).
\subsection{The topos $\mathcal{T}$ of locally compact topological spaces}

In this paper, we denote by $Top$ the category of locally compact topological spaces and continuous maps. Locally compact spaces are assumed to be Hausdorff. This category is endowed with the open cover topology $\mathcal{J}_{op}$, which is generated by the following pretopology: a family of morphisms $(X_{\alpha}\rightarrow X)_{\alpha\in A}$ is in $Cov(X)$ if and only if $(X_{\alpha}\rightarrow X)_{\alpha\in A}$ is an open cover in the usual sense. We denote by $\mathcal{T}$ the topos of sheaves of sets on this left exact site:
$$\mathcal{T}:=\widetilde{(Top,\mathcal{J}_{op})}.$$
The family of compact spaces is a topologically generating family for the site $(Top,\mathcal{J}_{op})$. Indeed, if $X$ is a locally compact space, then any $x\in X$ has a compact neighborhood $K_x\subseteq X$, so $(K_x\hookrightarrow X)_{x\in X}$ is a local section cover, hence a covering family for $\mathcal{J}_{op}$. In particular, if we denote by $Top^c$ the category of compact spaces, then the canonical morphism $\mathcal{T}\rightarrow\widetilde{(Top^c,\mathcal{J}_{op})}$ is an equivalence.

The Yoneda functor
$$\fonc{y}{Top}{\mathcal{T}}{X}{y(X)=Hom_{Top}(-,X)}$$
sending a locally compact space to the sheaf represented by this space
is fully faithful (since $\mathcal{J}_{op}$ is subcanonical) and
commutes with arbitrary projective limits. Since the Yoneda functor is left exact, any locally compact topological group $G$ represents a group object of $\mathcal{T}$. In what follows we consider
$Top$ as a (left exact) full subcategory of $\mathcal{T}$. For
example the sheaf of $\mathcal{T}$ represented by a locally compact space $Z$
is sometimes also denoted by $Z$.

\begin{rem}
In this paper, we consider topoi defined over the topos $\mathcal{T}$ of locally compact spaces since all sheaves, cohomology groups and fundamental groups that we use are defined by locally compact spaces. In order to use non-locally compact coefficients, one can consider the topos $$\mathcal{T}':=\widetilde{(Top^h,\mathcal{J}_{op})}$$ where $Top^h$ is the category of Hausdorff spaces. Then for any topos $\mathcal{E}$ (connected and locally connected) over $\mathcal{T}$, we consider the base change $\mathcal{E}\times_{\mathcal{T}}\mathcal{T}'$ to obtain a (connected and locally connected) topos over $\mathcal{T}'$.
\end{rem}
\subsection{Topological pro-groups}
In this paper, a \emph{filtered category} $I$ is a non-empty small category such that the following holds. For any objects $i$ and $j$ of $I$, there exists a third object $k$ and maps $i\leftarrow k\rightarrow j$. For any pair of maps $i\rightrightarrows j$ there exists a map $k\rightarrow i$ such that the diagram $k\rightarrow i\rightrightarrows j$ is commutative. Let $C$ be any category. A \emph{pro-object} of $C$ is a functor $\underline{X}:I\rightarrow C$, where $I$ is a filtered category. One can see a pro-object in $C$ as a diagram in $C$.
One can define the \emph{category $Pro(C)$ of pro-objects in $C$} (see \cite{SGA4} I. 8.10). The morphisms in this category can be made explicit as follows. Let $\underline{X}:I\rightarrow C$ and $\underline{Y}:J\rightarrow C$ be two pro-objects in $C$. Then one has
$$Hom_{Pro(C)}(\underline{X},\underline{Y}):=\underleftarrow{lim}_{_{j\in J}}\,\underrightarrow{lim}_{_{i\in I}}\, Hom(X_i,Y_j).$$
A pro-object $\underline{X}:I\rightarrow C$ is \emph{constant} if it is a constant functor, and $\underline{X}:I\rightarrow C$ is \emph{essentially constant} if $\underline{X}$ is isomorphic (in the category $Pro(C)$) to a constant pro-object.

\begin{defn}\label{defn-progroup}
A \emph{locally compact topological pro-group} $\underline{G}$ is a pro-object in the category of locally compact topological groups. A locally compact topological pro-group is said to be \emph{strict} if any transition map $G_j\rightarrow G_i$ has local sections.
\end{defn}
If the category $C$ is a topos, then a pro-object $\underline{X}:I\rightarrow C$ in $C$ is said to be \emph{strict} when the transition map $X_i\rightarrow X_j$ is an epimorphism in $C$, for any $i\rightarrow j\in Fl(I)$. In particular, a locally compact topological pro-group $\underline{G}:I\rightarrow Gr(Top)$ pro-represents a strict pro-group-object in $\mathcal{T}$:
$$y\circ\underline{G}:I\rightarrow Gr(Top)\rightarrow Gr(\mathcal{T})$$
where $Gr(Top)$ and  $Gr(\mathcal{T})$ are the categories of group-objects in $Top$ and $\mathcal{T}$ respectively. Indeed, a continuous map of locally compact spaces $X_i\rightarrow X_j$ has local sections if and only if it induces an epimorphism $y(X_i)\rightarrow y(X_j)$ in $\mathcal{T}$.

\subsection{The classifying topos of a group-object}

\subsubsection{General case.}
For any topos $\mathcal{S}$ and any group object $\mathcal{G}$ in $\mathcal{S}$, we denote by $B_{\mathcal{G}}$ the category of (left) $\mathcal{G}$-object in $\mathcal{S}$. Then $B_{\mathcal{G}}$ is a topos, as it follows from Giraud's axioms, and $B_{\mathcal{G}}$ is endowed with a canonical morphism $B_{\mathcal{G}}\rightarrow\mathcal{S}$, whose inverse image functor sends an object $\mathcal{F}$ of $\mathcal{S}$ to $\mathcal{F}$ with trivial $\mathcal{G}$-action. If there is a risk of ambiguity, we denote the topos $B_{\mathcal{G}}$ by $B_{\mathcal{S}}(\mathcal{G})$. The topos $B_{\mathcal{G}}$ is said to be the classifying topos of $\mathcal{G}$ since it classifies $\mathcal{G}$-torsors. More precisely, for any topos $f:\mathcal{E}\rightarrow\mathcal{S}$ over $\mathcal{S}$, the category $\underline{Homtop}_{\mathcal{S}}\,(\mathcal{E},B_{\mathcal{G}})$ is anti-equivalent to the category of $f^*\mathcal{G}$-torsors in $\mathcal{E}$ (see \cite{SGA4} IV Exercise 5.9). It follows that the induced morphism
$$B_{\mathcal{E}}(f^*\mathcal{G})\longrightarrow \mathcal{E}\times_\mathcal{S}B_{\mathcal{S}}(\mathcal{G})$$
is an equivalence (see \cite{these} Corollary 10.14).

\subsubsection{Examples.}

Let $G$ be a discrete group, i.e. a group object of the final topos $\underline{Sets}$. We denote the category of $G$-sets by $B_G^{sm}:=B_{\underline{Sets}}(G)$. The topos $B_G^{sm}$ is called the small classifying topos of the discrete group $G$. If $G$ is a profinite group, then the small classifying topos $B_G^{sm}$ is defined as the category of sets on which $G$ acts continuously.

Let $G$ be a locally compact topological group. Then $G$ represents a group object of $\mathcal{T}$, where $\mathcal{T}:=\widetilde{(Top,\mathcal{J}_{op})}$ is defined above. Then $$B_G:=B_{\mathcal{T}}(G)$$ is the classifying topos of the locally compact topological group $G$. One can define the classifying topos of an arbitrary topological group by enlarging the topos $\mathcal{T}$.

\subsubsection{The local section site.}\label{subsect-loc-section-site-BG}

For any locally compact topological group $G$, we denote by $B_{Top}(G)$ the category
of $G$-equivariant locally compact topological spaces (elements of a given universe). The category $B_{Top}(G)$ is endowed
with the local section topology $\mathcal{J}_{ls}$, which can be described as follows. A family of morphisms $\{X_i\rightarrow X,\,i\in I\}$ in $B_{Top}(G)$ is a covering family for $\mathcal{J}_{ls}$ if and only if the continuous map $\coprod_{i\in I}X_i\rightarrow X$ has local sections. Equivalently, $\mathcal{J}_{ls}$ is the topology induced by the open cover topology on $Top$ via the forgetful functor
$B_{Top}(G)\rightarrow Top$. The Yoneda functor yields a continuous fully faithful functor
$$B_{Top}(G)\longrightarrow B_G,$$
and the induced morphism
$$B_G\longrightarrow\widetilde{(B_{Top}(G),\mathcal{J}_{ls})}$$
is an equivalence (see \cite{MatFlach}).

\subsection{The classifying topos of a strict topological pro-group}

Topos theory provides a natural way to define the limit of a strict topological pro-group without any loss of information.
\begin{defn}\label{def-class-strict-top-progrp}
The classifying topos of a strict locally compact topological pro-group $\underline{G}:I\rightarrow Gr(Top)$ is defined as
$$B_{\underline{G}}:=\underleftarrow{lim}_{I}\, B_{G_i},$$
where the the projective limit is computed in the 2-category of topoi.
\end{defn}

By (\cite{SGA4} VI.8.2.3), a site for the projective limit topos $B_{\underline{G}}$ is given by $(\underrightarrow{lim}_I\,B_{Top}G_i,\mathcal{J})$, where $\underrightarrow{lim}_I\,B_{Top}G_i$ is the direct limit category and $\mathcal{J}$ is the coarsest topology such that all the functors
$$(B_{Top}G_i,\mathcal{J}_{ls})\longrightarrow (\underrightarrow{lim}_I\,B_{Top}G_i,\mathcal{J})$$
are continuous. The direct limit category
$$B_{Top}\underline{G}:=\underrightarrow{lim}_{I}\,B_{Top}G_i$$ can be made explicit as follows. An object of this category is given by a locally compact topological space on which $G_i$ acts continuously for some $i\in I$. Let $Z_i$ and $Z_j$ be two objects of $B_{Top}\underline{G}$, such that $Z_i$ and $Z_j$ are given with an action of $G_i$ and $G_j$ respectively. Then there exists an index $k\in I$ and maps $G_k\rightarrow G_i$ and $G_k\rightarrow G_i$ admitting local sections. Then a morphism $Z_j\rightarrow Z_i$ is a $G_k$-equivariant continuous map
$Z_j\rightarrow Z_i$. Consider the forgetful functor
$$B_{Top}\underline{G}\longrightarrow Top.$$
One can prove that the topology $\mathcal{J}$ on $B_{Top}\underline{G}$ is induced by the local section topology on $Top$ via this forgetful functor, so that the topology $\mathcal{J}$ can be denoted by $\mathcal{J}_{ls}$. We have obtained the following result.
\begin{prop}
The site $(B_{Top}\underline{G},\mathcal{J}_{ls})$ is a site for the classifying topos of the strict topological pro-group $\underline{G}$. In other words, the natural morphism
$$B_{\underline{G}}\longrightarrow\widetilde{(B_{Top}\underline{G},\mathcal{J}_{ls})}$$
is an equivalence.
\end{prop}

\section{The Weil-étale topos}
In this section we define a topos satisfying the expected
properties of the conjectural Lichtenbaum topos (see \cite{Fund-group-I}). This yields a new computation of the Weil-étale cohomology. Our construction is a suitable refinement of the family of  Weil-\'etale sites introduced by Lichtenbaum in
\cite{Lichtenbaum}.
We denote by $\bar{X}=(Spec\,\mathcal{O}_F,X_{\infty})$ the Arakelov
compactification of the ring of integers in a number field $F$ (i.e. $X_{\infty}$ is the set of archimedean places of $F$).

\subsection{The Weil-\'etale topos}\label{section-Lichtenbaum-topos}
As an illustration of the artificiality of the following
construction, we start by making several non-canonical choices.
\begin{choices}\label{choices-X}
~
\begin{enumerate}
\item We choose an algebraic closure $\bar{F}/F$.
\item We choose a Weil group
$W_F$.
\item For any place $v$ of $F$, we choose an algebraic
closure $\bar{F_v}/F_v$.
\item For any place $v$ of $F$, we choose a local Weil group $W_{F_v}$.
\item For any place $v$ of $F$, we choose an embedding $\bar{F}\rightarrow \bar{F_v}$ over $F$.
\item For any place $v$ of $F$, we choose a Weil map $\theta_v:W_{F_v}\longrightarrow
W_F$.
\end{enumerate}
\end{choices}
These choices are required to be compatible in the sense that the diagram
\[ \xymatrix{
W_{F_v}\ar[d]_{\theta_v}\ar[r]&G_{F_v}\ar[d]\\
W_F\ar[r]&G_F }\]
is commutative for any place $v$.

Recall that if $\bar{F}/F$ is an algebraic closure and $\bar{F}/K/F$ a finite Galois extension then the relative Weil group $W_{K/F}$ is defined by the extension of topological groups
$$1\rightarrow C_K \rightarrow W_{K/F}\rightarrow G_{K/F}\rightarrow 1$$
corresponding to the fundamental class in $H^2(G_{K/F},C_K)$ given by class field theory, where $C_K$ is the id\`ele class group of $K$. A Weil group of $F$ is then defined as the projective limit $W_F:=\underleftarrow{lim}\,W_{K/F}$, computed in the category of topological groups. A Weil group for the local field $F_v$ is defined as above, replacing the id\`ele class group $C_K$ with the mutiplicative group $K_w^{\times}$ where $K_w/F_v$ is finite and Galois.

\begin{defn}
Let $W^1_{F_v}$ be the maximal compact subgroup of $W_{F_v}$. The \emph{Weil group of the "residue field" at $v\in\bar{X}$}
is defined as $W_{k(v)}:=W_{F_v}/W^1_{F_v}$. We denote by
$$q_v:W_{F_v}\longrightarrow W_{F_v}/W^1_{F_v}=:W_{k(v)}$$ the map from
the local Weil group $W_{F_v}$ to the Weil group of the residue
field at $v$.
\end{defn}

Lichtenbaum defined in \cite{Lichtenbaum} a family of sites $T_{K/F,S}$ for $K/F$ Galois and $S$ a suitable finite set of primes of $F$. Then he defined the Weil-\'etale cohomology as the direct limit of the cohomologies of the sites $T_{K/F,S}$. Here we define a single site $T_{\bar{X}}$ inspired by a closer look at the \'etale site. This allows us to define a Weil-\'etale topos, over the Artin-Verdier \'etale topos, giving rise to the Weil-\'etale cohomology without the direct limit process.

\begin{defn}\label{defn-TX}
Let $T_{\bar{X}}$ be the category of objects $(Z_0,Z_v,f_v)$ defined
as follows. The topological space $Z_0$ is endowed with a continuous
$W_F$-action. For any place $v$ of $F$, $Z_v$ is a topological space
endowed with a continuous $W_{k(v)}$-action. The continuous map
$f_v:Z_v\rightarrow Z_0$ is $W_{{F_v}}$-equivariant, when $Z_v$ and
$Z_0$ are seen as $W_{{F_v}}$-spaces via the maps
$\theta_v:W_{F_v}\rightarrow W_{F}$ and $q_v:W_{F_v}\rightarrow
W_{k(v)}$. Moreover, we require the following facts.
\begin{itemize}
\item The map $f_v$ is an \emph{homeomorphism for almost all
places $v$ of $F$ and a continuous injective map for all places}.
\item  For any valuation $v$, the space $Z_v$ is locally compact.
\item The action of $W_F$ on $Z_0$ factors through $W_{K/F}$, for
some finite Galois subextension $\bar{F}/K/F$.
\end{itemize}
A \emph{morphism} $$\phi:(Z_0,Z_v,f_v)\longrightarrow(Z'_0,Z'_v,f'_v)$$
in the category $T_{\bar{X}}$ is a continuous $W_F$-equivariant map
$\phi:Z_0\rightarrow Z'_0$ \emph{inducing} a continuous map
$\phi_v:Z_v\rightarrow Z_v$ for any place $v$. Then $\phi_v$ is
necessarily $W_{k(v)}$-equivariant.

The category $T_{\bar{X}}$ is endowed with the local section
topology $\mathcal{J}_{ls}$, i.e. the topology generated by the
pretopology for which a family
$$\{\varphi_i:(Z_{i,0},Z_{i,v},f_{i,v})\rightarrow
(Z_0,Z_v,f_v),\,i\in I\}$$ is a covering family if the continuous map
$\coprod_{i\in I}Z_{i,v}\rightarrow Z_v$ has local sections, for any
place $v$.
\end{defn}

\begin{defn}\label{defn-wettoposX}
We define the \emph{Weil-\'etale topos} $\bar{X}_{W}$ as the topos of
sheaves of sets on the site defined above:
$$\bar{X}_{W}:=\widetilde{(T_{\bar{X}},\mathcal{J}_{ls})}.$$
\end{defn}

\begin{lem}
The site $(T_{\bar{X}},\mathcal{J}_{ls})$ is a left exact site.
\end{lem}
\begin{proof}
The category $T_{\bar{X}}$ has fiber products and a final object, hence finite projective limits are representable in $T_{\bar{X}}$. It
remains to show that $\mathcal{J}_{ls}$ is subcanonical. But for any topological group $G$, the local section topology
$\mathcal{J}_{ls}=\mathcal{J}_{op}$ on $B_{Top}{G}$ is nothing else than the open cover topology (see \cite{MatFlach} Corollary 2), which
is easily seen to be subcanonical. The result follows easily from this fact.
\end{proof}

\begin{prop}
We have a morphism of topoi
$$j:B_{W_{F}}\longrightarrow\bar{X}_W.$$
\end{prop}

\begin{proof}
The classifying topos $B_{W_{F}}$ is defined as the topos of
$y(W_{F})$-objects of $\mathcal{T}$ and the site
$(B_{Top}{W_{F}},\mathcal{J}_{ls})$ is a defining site for $B_{W_{F}}$ (see section \ref{subsect-loc-section-site-BG}). We have a morphism of left exact sites
$$\fonc{j^*}{(T_{\bar{X}},\mathcal{J}_{ls})}{(B_{Top}{W_{F}},\mathcal{J}_{ls})}{(Z_0,Z_v,f_v)}{Z_0}.$$
inducing the morphism of topoi $j$.
\end{proof}
We have a morphism of left exact sites
\begin{equation}\label{morphism-sites-toT}
\fonc{t^*}{(Top,\mathcal{J}_{op})}{(T_{\bar{X}},\mathcal{J}_{ls})}{Z}{(Z,Z,Id_Z)}
\end{equation}
\begin{defn} The canonical morphism from $\bar{X}_{W}$ to $\mathcal{T}$ is the morphism of topoi
$$t:\bar{X}_{W}\longrightarrow\mathcal{T}$$
induced by the morphism of left exact sites
(\ref{morphism-sites-toT}).
\end{defn}

Consider the functor $u^*:\mathcal{T}\rightarrow B_{W_F}$ sending an object $\mathcal{L}$
of $\mathcal{T}$ to $\mathcal{L}$ with trivial $y(W_F)$-action. This
functor commutes with arbitrary inductive and arbitrary projective
limits. Therefore we have a sequence of three adjoint functors
$u_!\,,\,u^*\,,\,u_*$. More explicitly, one has
$u_!\mathcal{\mathcal{F}}=\mathcal{F}/{y(W_F)}$ and
$u_*(\mathcal{F})=\underline{Hom}_{W_F}(\{*\},\mathcal{F})$, where
$\{*\}$ is the final object of $B_{W_F}$. We have a
connected ($u^*$ is clearly fully faithful) and locally connected
morphism
$$u:B_{W_F}\longrightarrow \mathcal{T}.$$
The topos $B_{W_F}$ has a canonical point $q$ over $\mathcal{T}$. In other words there exists a section $q:\mathcal{T}\rightarrow B_{W_F}$ of the structure map $u:B_{W_F}\rightarrow \mathcal{T}$. Indeed, the inverse image of the morphism $q$ is the functor $q^*:B_{W_F}\rightarrow\mathcal{T}$ sending
a $y(W_F)$-object $\mathcal{F}$ to $\mathcal{F}$ with no action. Therefore we have a canonical isomorphism of functor $Id\simeq q^*\circ u^*$, hence an isomorphism of morphisms of topoi :
\begin{equation}\label{point-B_W}
Id\simeq u\circ q: \mathcal{T}\rightarrow B_{W_F}\rightarrow
\mathcal{T}.
\end{equation}
Of course everything above is valid for any topological
group $G$ (and more generally for any group object $\mathcal{G}$ in
any topos $\mathcal{E}$).
\begin{prop}\label{prop-point-XL-and-j}
One has a canonical isomorphism
$$u\simeq t\circ j:B_{W_F}\longrightarrow\bar{X}_W\longrightarrow\mathcal{T}$$
In particular the morphism $j\circ q$ is a point of $\bar{X}_W$ over
$\mathcal{T}$, i.e. the following diagram is commutative.
\[ \xymatrix{
B_{W_F}\ar[r]^j&\bar{X}_W\ar[d]^t\\
\mathcal{T}\ar[u]^q\ar[ru]^p\ar[r]^{Id} &\mathcal{T} }\] If there is a risk of ambiguity, the point $p$ of $\bar{X}_W$ over $\mathcal{T}$ will be denoted by $p_{\bar{X}}$.
\end{prop}
\begin{proof}
The first claim of the proposition follows immediately from the
description of these morphisms of topoi in terms of morphisms of
left exact sites. The second claim then follows from
(\ref{point-B_W}).
\end{proof}

\begin{prop}\label{prop-nice}
The morphism $t:\bar{X}_{W}\rightarrow\mathcal{T}$ is connected
and locally connected.
\end{prop}
\begin{proof}
Let us first make the inverse image functor $t^*$ explicit. Consider
the full subcategory $\mathbb{C}_{\bar{X}}$ of  $T_{\bar{X}}$
consisting in objects $(Z_0,Z_v,f_v)$ such that the quotient space $Z_0/W_F$ is locally compact and such that the canonical morphism in
$\mathcal{T}$ $$y(Z_0)/y(W_F)\longrightarrow y(Z_0/W_F)$$ is an isomorphism.
By Corollary \ref{cor-sympa}, $\mathbb{C}_{\bar{X}}$ is a
\emph{topologically generating} family of $T_{\bar{X}}$ (see
\cite{SGA4} II 3.0.1). Hence the sheaf $t^*\mathcal{L}$ is
completely determined by its restriction to $\mathbb{C}_{\bar{X}}$,
for any $\mathcal{L}$ of $\mathcal{T}$.

One can prove that that one has
\begin{equation}\label{description-t^*-representable}
t^*\mathcal{L}(Z_0,Z_v,f_v)=Hom_{B_{W_F}}(y(Z_0),u^*\mathcal{L})=u^*\mathcal{L}(Z_0).
\end{equation}
for any object $\mathcal{L}$ of $\mathcal{T}$ and any
$(Z_0,Z_v,f_v)\in\mathbb{C}_{\bar{X}}$. Indeed, we check that one
has a bifunctorial isomorphism, in $\mathcal{L}\in\mathcal{L}$ and
$\mathcal{F}\in\bar{X}_W$ :
$$Hom_{\bar{X}_W}(t^*\mathcal{L},\mathcal{F})\simeq Hom_{\mathcal{T}}(\mathcal{L},t_*\mathcal{F}),$$
where $t^*\mathcal{L}$ is defined as above and
$t_*\mathcal{F}(Z):=\mathcal{F}(Z,Z,Id_Z)$ for any $Z\in Ob(Top)$.
The proof of this fact is tedious but straightforward, using the
fact that $(Z_0/W_F,Z_0/W_F,Id)\in\mathbb{C}_{\bar{X}}$ and the
identification $y(Z_0/W_F)=y(Z_0)/y(W_F)$.

More generally, we have
\begin{equation}\label{description-t^*}
Hom_{\bar{X}_W}(\mathcal{F},t^*\mathcal{L})=Hom_{B_{W_F}}(j^*\mathcal{F},u^*\mathcal{L})
\end{equation}
for any object $\mathcal{F}$ of $\bar{X}_W$ and any object
$\mathcal{L}$ of $\mathcal{T}$. Indeed the family of representable
objects $y(Z_0,Z_v,f_v)$ is a generating family of $\bar{X}_W$ (see
\cite{SGA4} II Proposition 4.10) hence any object $\mathcal{F}$ of
$\bar{X}_W$ can be written as an (arbitrary) inductive limit of such
representable objects (see \cite{SGA4} I Proposition 7.2). Therefore
(\ref{description-t^*}) follows from
(\ref{description-t^*-representable}) and from the fact that $j^*$
commutes with inductive limits and with the Yoneda embedding.

If $\mathcal{L}$ and $\mathcal{L}'$ are two objects of
$\mathcal{T}$, then one has
$$Hom_{\bar{X}_W}(t^*\mathcal{L}',t^*\mathcal{L})=Hom_{B_{W_F}}(j^*t^*\mathcal{L}',u^*\mathcal{L})
=Hom_{B_{W_F}}(u^*\mathcal{L}',u^*\mathcal{L})=Hom_{\mathcal{T}}(\mathcal{L}',\mathcal{L})$$
since $t\circ j\simeq u$ and $u^*$ is fully faithful. Hence $t^*$ is
fully faithful, i.e. $t$ is connected.

Let us define the left adjoint of $t^*$. We consider the functor
defined by
$$t_!\mathcal{F}:=u_!j^*\mathcal{F}=j^*\mathcal{F}/y(W_F)$$
where the quotient $j^*\mathcal{F}/y(W_F)$ is defined in
$\mathcal{T}$, for any object $\mathcal{F}$ of $\bar{X}_W$. The
following identifications show that $t_!$ is left adjoint to $t^*$.
$$Hom_{\mathcal{T}}(t_!\mathcal{F},\mathcal{L})=Hom_{\mathcal{T}}(u_!j^*\mathcal{F},\mathcal{L})
=Hom_{B_{W_F}}(j^*\mathcal{F},u^*\mathcal{L})=Hom_{\bar{X}_W}(\mathcal{F},t^*\mathcal{L}).$$
It remains to show that the functor $t_!$ is a $\mathcal{T}$-indexed left adjoint to $t^*$. This means that for any morphism $x:I\rightarrow J$
in $\mathcal{T}$, the natural transformation
\begin{equation}\label{nat-transform}
t_!^I\circ(t^*x)^*\rightarrow x^*\circ t_!^J
\end{equation}
defined by the square
\[ \xymatrix{
\bar{X}_W/t^*I\ar[r]^{t_!^I}&\mathcal{T}/I\\
\bar{X}_W/t^*J\ar[u]^{(t^*x)^*}\ar[r]^{t_!^J}&\mathcal{T}/J\ar[u]^{x^*}
}\] should be an isomorphism (see \cite{elephant} B.3.1.1). Here the functor
$$
\fonc{x^*}{\mathcal{T}/J}{\mathcal{T}/I}{(\mathcal{L}\rightarrow
J)}{(\mathcal{L}\times_JI\rightarrow I)}
$$
is the usual base change and one has
$$
\fonc{t_!^J}{\bar{X}_W/t^*J}{\mathcal{T}/J}{(\mathcal{F}\rightarrow
t^*J)}{(t_!\mathcal{F}\rightarrow t_!t^*J\rightarrow J)}
$$
where the map $t_!t^*J\rightarrow J$ is given by adjunction. Let
$\mathcal{F}\rightarrow t^*J$ be an object of $\bar{X}_W/t^*J$, and
denote it by $\mathcal{F}$. On the one hand one has
$$t_!^I\circ(t^*x)^*\mathcal{F}=t_!(\mathcal{F}\times_{t^*J}t^*I)$$
and
$$x^*\circ t_!^J(\mathcal{F}\rightarrow t^*J)=t_!(\mathcal{F})\times_{J}I$$
one the other. Hence the natural transformation
(\ref{nat-transform}) is given by the canonical morphism from
$$t_!(\mathcal{F}\times_{t^*J}t^*I)=u_!j^*(\mathcal{F}\times_{t^*J}t^*I)
=u_!(j^*\mathcal{F}\times_{j^*t^*J}j^*t^*I)=u_!(j^*\mathcal{F}\times_{u^*J}u^*I)$$
to
$$u_!j^*\mathcal{F}\times_{u_!u^*J}u_!u^*I\simeq u_!j^*\mathcal{F}\times_{J}I.$$
This morphism is an isomorphism because
$u:B_{W_F}\rightarrow\mathcal{T}$ is connected and locally
connected. Indeed, the adjunction map $u_!u^*I\rightarrow I$ is an
isomorphism since $u^*$ is fully faithful. Then
$$u_!(j^*\mathcal{F}\times_{u^*J}u^*I)=(j^*\mathcal{F}\times_{J}I)/yW_F$$
is canonically isomorphic to
$$(j^*\mathcal{F}/yW_F)\times_{J}I$$
since inductive limits (in particular quotients of group actions)
are universal in $\mathcal{T}$. For a down to earth argument proving
the very last claim of this proof, one can use the fact that
$\mathcal{T}$ has enough $\underline{Sets}$-valued points, and check
that $(j^*\mathcal{F}\times_{J}I)/yW_F\rightarrow
(j^*\mathcal{F}/yW_F)\times_{J}I$ induces isomorphisms on stalks.

\end{proof}

\begin{defn}
An object $\mathcal{F}$ of $\bar{X}_W$ is said to be \emph{constant
over $\mathcal{T}$} if there is an isomorphism $\mathcal{F}\simeq
t^*\mathcal{L}$, where $\mathcal{L}$ is an object of $\mathcal{T}$.
\end{defn}
\begin{cor}\label{cor-constant-generic}
If $\mathcal{F}$ is a constant object over $\mathcal{T}$ then the
adjunction map
$$\mathcal{F}\longrightarrow j_*j^*\mathcal{F}$$
is an isomorphism.
\end{cor}
\begin{proof}
This follows immediately from (\ref{description-t^*-representable}).
Indeed, if $\mathcal{F}=t^*\mathcal{L}$ then
$$\mathcal{F}(Z_0,Z_v,f_v)=t^*\mathcal{L}(Z_0,Z_v,f_v)=u^*\mathcal{L}(Z_0)=j^*\mathcal{F}(Z_0)=j_*j^*\mathcal{F}(Z_0,Z_v,f_v).
$$
for any object $(Z_0,Z_v,f_v)$ of $T_{\bar{X}}$.
\end{proof}

\begin{defn}
Let $\mathcal{F}$ be an object of $\bar{X}_W$. The object of
$\mathcal{T}$
$$t_!\mathcal{F}:=(j^*\mathcal{F})/y(W_F)$$
is called the \emph{space of connected components} of $\mathcal{F}$.
\end{defn}

\begin{defn}
An object  $\mathcal{F}$ of $\bar{X}_W$ is said to be
\emph{connected over $\mathcal{T}$} if its space of connected
components $t_!\mathcal{F}$ is the final object of $\mathcal{T}$.
\end{defn}
Consider for example a constant object $\mathcal{F}=t^*\mathcal{L}$
over $\mathcal{T}$. Then the space of connected components of
$\mathcal{F}$ is
$$t_!\mathcal{F}=t_!t^*\mathcal{L}\simeq\mathcal{L}$$
since $t^*$ is fully faithful. Therefore a constant object
$\mathcal{F}=t^*\mathcal{L}$ of $\bar{X}_W$ over $\mathcal{T}$ is
connected over $\mathcal{T}$ if and only if $\mathcal{F}$ is the
final object of $\bar{X}_W$.

\begin{rem}
Note that $t_!\mathcal{F}$ is not a topological space in general.
However this is a topological space when $\mathcal{F}$ is
representable by an object $(Z_0,Z_v,f_v)$ such that
$y(Z_0)/y(W_F)=y(Z_0/W_F)$. Our terminology is justified by the fact
that any object of $\mathcal{T}$ is topological in nature.
\end{rem}

\subsection{The morphism from the Weil-\'etale topos to the Artin-Verdier étale topos}\label{subsect-gamma}

Let $\bar{X}$ be the Arakelov compactification of a number ring
$\mathcal{O}_F$. The set $\bar{X}$ is given with the Zariski topology. We recall below the definition of the Artin-Verdier
étale site of $\bar{X}$. We refer to \cite{On the WE} for more details.

A \emph{connected étale $\bar{X}$-scheme} is a map
$$\bar{U}=(U;U_{\infty})\longrightarrow\bar{X}=(X;X_{\infty}),$$ where $U$
is a connected étale $X$-scheme in the usual sense. The set
$U_{\infty}$ is a subset of $U(\mathbb{C})/\sim$, where
$U(\mathbb{C})/\sim$ is the quotient of the set of complex valued
points of $U$ under the equivalence relation defined by complex
conjugation. Moreover $U_{\infty}/X_{\infty}$ is unramified in the
sense that if $v\in X_{\infty}$ is real, then so is any point $w$ of
$U_{\infty}$ lying over $v$. An \emph{étale $\bar{X}$-scheme} is a
finite sum of connected étale $\bar{X}$-schemes, called the
\emph{connected components} of $\bar{X}$. A morphism
$\bar{U}\rightarrow\bar{V}$ in the category $Et_{\bar{X}}$ is a
morphism of $X$-schemes $U\rightarrow V$ inducing a map
$U_{\infty}\rightarrow V_{\infty}$ over $X_{\infty}$. The
\emph{Artin-Verdier étale site of $\bar{X}$} is defined by the
category $Et_{\bar{X}}$ endowed with the topology $\mathcal{J}_{et}$
generated by the pretopology for which the coverings are the
surjective families.

\begin{defn}
The \emph{Artin-Verdier étale topos} of
$\bar{X}$ is the category of sheaves of sets on the Artin-Verdier
étale site:
$$\bar{X}_{et}:=\widetilde{(Et_{\bar{X}},\mathcal{J}_{et})}.$$
\end{defn}

Let $v$ be a closed point of $\bar{X}$. Data \ref{choices-X} gives an embedding $G_{F_v}\hookrightarrow G_F$, hence we have an inertia subgroup $I_v\subset G_F$ at $v$. One can define the \emph{strict henselization} of $\bar{X}$ at $v$ as the projective
limit $\bar{X}^{sh}_v=\underleftarrow{lim}\,\,\bar{U},$
where $\bar{U}$ runs over the filtered system of étale neighborhoods in $\bar{X}$
of a geometric point over $v$. We refer to \cite{On the WE} Section 6.2 for a precise definition. For $v$ ultrametric, one has $\bar{X}^{sh}_v:=Spec(\mathcal{O}^{sh}_{\bar{X},v})$ where the ring $\mathcal{O}^{sh}_{\bar{X},v}$ is the strict henselization of the local ring $\mathcal{O}_{X,v}$. The generic point of $\bar{X}^{sh}_v$ is $Spec(\bar{F}^{I_v})$. For an archimedean valuation $v$, $\bar{X}^{sh}_v$ can be formally defined as the pair $(Spec(\bar{F}^{I_v}),v)\rightarrow(X,X_{\infty})$.
Hence for any closed point $v$ of $\bar{X}$, Data \ref{choices-X} gives a \emph{specialization map over $\bar{X}$}
\begin{equation}\label{specialization-map}
Spec(\bar{F})\rightarrow Spec(\bar{F}^{I_v})\hookrightarrow\bar{X}^{sh}_v.
\end{equation}

\begin{prop}\label{prop-morph-sites-etale-loc-sections}
There exists a morphism of left exact sites
$$\fonc{\gamma^*}{(Et_{\bar{X}};\mathcal{J}_{et})}{(T_{\bar{X}};\mathcal{J}_{ls})}{\bar{U}}{(U_0,U_v,f_v)}.$$
The functor $\gamma^*$ is fully faithful and the essential image of $\gamma^*$ consists in objects
$(U_0,U_v,f_v)$ where $U_0$ is a finite $W_F$-set.
\end{prop}
This result is a reformation of \cite{these} Proposition 4.61
and \cite{these} Proposition 4.62. Let us fix some notations. For any point $v\in\bar{X}$ we define the Galois group of the "residue field at $v$" as follows :
$$G_{k(v)}:=G_{F_v}/I_v$$
while the Weil group of the residue field at $v$ is defined as $W_{k(v)}:=W_{F_v}/W^1_{F_v}$. Note that we have a morphism
$\alpha_v:W_{k(v)}\rightarrow G_{k(v)}$ compatible with the Weil map $\theta_v:W_{F_v}\rightarrow W_F$ for any $v\in\bar{X}$. Note also that,  for an archimedean valuation $v$, the group $G_{k(v)}$ is trivial and $W_{k(v)}$ is isomorphic to $\mathbb{R}^{\times}_+$ as a topological group.

\begin{proof}
For any étale $\bar{X}$-scheme $\bar{U}$, we define an object
$\gamma^*(\bar{U})=(U_0,U_v,f_v)$ of $T_{\bar{X}}$ as follows.
An algebraic closure $\bar{F}/F$ has been chosen in Data
\ref{choices-X}. The generic point $\bar{U}\times_{\bar{X}}Spec\,F$
is the spectrum of a finite \'etale $F$-algebra. The Grothendieck-Galois
theory shows that this \'etale $F$-algebra is uniquely determined by the
finite $G_F$-set
$$U_0:=Hom_{Spec\, F}(Spec\, \bar{F},\bar{U}\times_{\bar{X}}Spec\,F)=Hom_{\bar{X}}(Spec\,\bar{F},\bar{U}).$$
Let $v$ be an ultrametric place of $F$. The maximal unramified
sub-extension of the algebraic closure $\bar{F_v}/F_v$ chosen in
Data \ref{choices-X} yields an algebraic closure of the residue
field $\overline{k(v)}/k(v)$. The scheme
$\bar{U}\times_{\bar{X}}Spec\,k(v)$ is the spectrum of an étale
$k(v)$-algebra, corresponding to the finite $G_{k(v)}$-set
$$U_v:=Hom_{Spec\, k(v)}(Spec\, \overline{k(v)},\bar{U}\times_{\bar{X}}Spec\,k(v))=Hom_{\bar{X}}(Spec\,\overline{k(v)},\bar{U})$$
Let $v$ be an ultrametric place of $F$. Here we define the set
$$U_v:=Hom_{\bar{X}}((\emptyset,v),\bar{U})=v\times_{X_{\infty}}U_{\infty}$$
For any closed point $v$ of $\bar{X}$, we have $U_v=Hom_{\bar{X}}(\bar{X}_v^{sh},\bar{U})$ hence the specialization map (\ref{specialization-map}) gives a $G_{F_v}$-equivariant map
$$f_v:U_v\longrightarrow U_0.$$
This map is bijective for almost all valuations and injective for all
valuations. For any place $v$ of $F$, the set $U_v$ is viewed as a
$W_{k(v)}$-topological space via the morphism $W_{k(v)}\rightarrow
G_{k(v)}$. Respectively, $U_0$ is viewed as a $W_{F}$-topological
space via $W_{F}\rightarrow G_{F}$. Then the map $f_v$ defined above
is $W_{F_v}$-equivariant. We obtain a functor
$$\gamma^*:Et_{\bar{X}}\longrightarrow T_{\bar{X}}.$$
Note that if $\bar{U}$ is the a finite sum of connected \'etale $\bar{X}$-schemes
$\bar{U}=\coprod\bar{U}_i$, then we have
$$\gamma^*(\bar{U})=\coprod\gamma^*(\bar{U}_i)$$
where the sum one the right hand side is understood in $T_{\bar{X}}$.
The functor $\gamma^*$ is easily seen to be left exact (i.e. it preserves the final object and
fiber products) and continuous (i.e. it preserves covering families). Hence we do have a morphism of left exact sites.

Let $\mathcal{U}=(U_0,U_v,f_v)$ be an object of $T_{\bar{X}}$ such that $U_0$ is a finite $W_F$-set. Writing $U_0$ as the sum of its $W_F$-orbits, we can decompose $\mathcal{U}=\coprod_{i\in I}\mathcal{U}_i$ as a sum in $T_{\bar{X}}$. The action of $W_F$ on $U_0$ factors through $W_F/W_F^0=G_F$, where $W_F^0$ is the connected of $1$ in $W_F$, since $U_0$ is finite. Hence we can see $U_0$ as finite $G_F$-sets. By Galois theory, $U_0$ corresponds to an essentially unique \'etale $F$-algebra $A=\prod_{i\in I} K_i$. Then for any $i\in I$ one has a finite set $S_i$ of places of $K_i$ and an isomorphism in $T_{\bar{X}}$:
$$\mathcal{U}_i\simeq \gamma^*(\overline{Spec(\mathcal{O}_{K_i})}-S_i)$$
This shows that the essential image of $\gamma^*$ consists in objects
$(U_0,U_v,f_v)$ such that $U_0$ is a finite $W_F$-set.

Let $\bar{U}$ and $\bar{U}'$ be two objects of $Et_{\bar{X}}$. We set $\gamma^*(\bar{U})=(U_0,U_v,f_v)$ and $\gamma^*(\bar{U}')=(U'_0,U'_v,f'_v)$. By functoriality, we have a map
\begin{equation}\label{fonctoriality}
Hom_{\bar{X}}(\bar{U},\bar{U}')\rightarrow Hom_{T_{\bar{X}}}((U_0,U_v,f_v),(U'_0,U'_v,f'_v)).
\end{equation}
We define the inverse map as follows. A morphism $\phi:(U_0,U_v,f_v)\rightarrow(U'_0,U'_v,f'_v)$ is given by a map of finite $G_F$-sets $\phi_0:U_0\rightarrow U'_0$. This map gives a uniquely determined morphism of $F$-algebras $A'\rightarrow A$, where $Spec(A):=\bar{U}\times_{\bar{X}}Spec(F)$ and $Spec(A'):=\bar{U}'\times_{\bar{X}}Spec(F)$, again by Galois theory. The morphism $A'\rightarrow A$ induces a morphism of \'etale $\bar{X}$-schemes $\widetilde{\phi}:\bar{U}\rightarrow\bar{U}'$ precisely because $\phi_0$ induces a map $\phi_v:U_v\rightarrow U'_v$ for any closed point $v$ of $\bar{X}$. Then $\phi\mapsto\widetilde{\phi}$ is the inverse isomorphism to (\ref{fonctoriality}). This shows that $\gamma^*$ is fully faithful.

\end{proof}

The next corollary follows immediately from the fact that a morphism
of left exact sites induces a morphism of topoi.
\begin{cor}\label{cor-morphism-gamma-WE-et}
There is a morphism of topoi
$$\gamma:\bar{X}_W\longrightarrow\bar{X}_{et}.$$
\end{cor}

\begin{rem}\label{rem-connected-U-connected-tU}
Let $\mathcal{F}$ be an object of $\bar{X}_W$ represented by an
étale $\bar{X}$-scheme $\bar{U}$. In other words, we assume that
$$\mathcal{F}=\gamma^*y(\bar{U})=y(\gamma^*\bar{U})=y(U_0,U_v,f_v)$$ where
$U_0$ is a  finite $G_F$-set. The space of connected components
$$t_!\mathcal{F}:=(j^*\mathcal{F})/y(W_F)=U_0/G_F$$ is the object of
$\mathcal{T}$ represented by the finite set $U_0/G_F$, which is the
set of connected components of $\bar{U}$ in the usual sense.
\end{rem}

\subsection{Structure of $\bar{X}_W$ at the closed points}

Let $v$ be a place of $F$. We consider the Weil group $W_{k(v)}$ and
the Galois group $G_{k(v)}$ of the residue field $k(v)$ at
$v\in\bar{X}$. Recall that, for $v$ archimedean, one has
$W_{k(v)}\simeq\mathbb{R}$ and $G_{k(v)}=\{1\}$. We consider the big
classifying topos $B_{W_{k(v)}}$, i.e. the category of
$y(W_{k(v)})$-objects in $\mathcal{T}$, and the small
classifying topos $B^{sm}_{G_{k(v)}}$, which is defined as the
category of continuous $G_{k(v)}$-sets. The category of
$W_{k(v)}$-topological spaces $B_{Top}{W_{k(v)}}$ is endowed with
the local section topology $\mathcal{J}_{ls}$. Then
$(B_{Top}{W_{k(v)}},\mathcal{J}_{ls})$ is a site for the classifying
topos $B_{W_{k(v)}}$. Respectively let $B_{fSets}G_{k(v)}$ be the
category of finite $G_{k(v)}$-sets endowed with the canonical
topology $\mathcal{J}_{can}$. The site
$(B_{fSets}G_{k(v)},\mathcal{J}_{can})$ is a site for the small
classifying topos $B^{sm}_{G_{k(v)}}$.

For any place $v$ of $F$, we have a morphism of left exact sites
$$\fonc{i_v^*}{(T_{\bar{X}},\mathcal{J}_{ls})}{(B_{Top}{W_{k(v)}},\mathcal{J}_{ls})}{(Z_0,Z_v,f_v)}{Z_v}$$
hence a morphism of topoi
$$i_v:B_{W_{k(v)}}\longrightarrow\bar{X}_W.$$
Assume that $v$ is ultrametric. The morphism of schemes
$$Spec\,k(v)\longrightarrow\bar{X}$$
induces a morphism of topoi
$$u_v:B^{sm}_{G_{k(v)}}\longrightarrow\bar{X}_{et}$$
since the étale topos of $Spec\,k(v)$ is equivalent to the category
$B^{sm}_{G_{k(v)}}$ of continuous $G_{k(v)}$-sets. This
equivalence is induced by the choice of an algebraic closure of
${k(v)}$ given in Data \ref{choices-X}. For $v$ archimedean, we still have a morphism
$$u_v:B^{sm}_{G_{k(v)}}=\underline{Sets}=Sh(v)\longrightarrow\bar{X}_{et}.$$
The category of finite
$G_{k(v)}$-sets endowed with the canonical topology is a site for
the small classifying topos $B^{sm}_{G_{k(v)}}$. We have a
commutative diagram of left exact sites
\[ \xymatrix{
(B_{Top}W_{k(v)},\mathcal{J}_{ls})&(B_{fSets}{G_{k(v)}},\mathcal{J}_{ls})\ar[l]_{\alpha^*_v}\\
(T_{\bar{X}},\mathcal{J}_{ls})\ar[u]^{i_v^*}&
(Et_{\bar{X}},\mathcal{J}_{ls})\ar[u]_{u_v^*}\ar[l]_{\gamma^*} }\]
where $u_v^*(\bar{U})$ is the finite $G_{k(v)}$-set
$$U_v:=Hom_{\bar{X}}(Spec\,\overline{k(v)},\bar{U}).$$
The diagram of sites above induces the commutative of topoi of the
following result, which is proven in \cite{Flach-moi}.

\begin{thm}\label{thm-pull-back}
For any closed point $v$ of $\bar{X}$, the following diagram is a
pull-back of topoi.
\[ \xymatrix{
B_{W_{k(v)}}\ar[d]_{i_v}\ar[r]^{\alpha_v}&B^{sm}_{G_{k(v)}}\ar[d]_{u_v}\\
\bar{X}_W\ar[r]^{\gamma}&\bar{X}_{et} }\]
\end{thm}

\begin{cor} For any closed point $v$ of $\bar{X}$, the
morphism $i_v$ is a closed embedding.
\end{cor}
\begin{proof}
It is well know that the morphism of \'etale topoi
$$u_v:B^{sm}_{G_{k(v)}}\longrightarrow\bar{X}_{et}$$
is a closed embedding. The result then follows from the fact that closed embeddings are stable under pull-backs.
Indeed, the image of $u_v$ is a closed subtopos $Im(u_v)$ of $\bar{X}_{et}$. But the inverse image of $Im(u_v)$ under $\gamma$ is precisely the image of $i_v$, as it follows from the previous theorem. Hence $Im(i_v)$ is a closed subtopos of $\bar{X}_W$, and $i_v$ induces an equivalence $B_{W_{k(v)}}\simeq Im(i_v)$.
\end{proof}

\subsection{The Weil-\'etale topos of an étale $\bar{X}$-scheme}

\begin{rem}
In this section we define the Weil-étale topos $\bar{U}_W$ for any étale $\bar{X}$-scheme $\bar{U}$. Such a definition must be functorial, i.e. one needs to obtain a pseudo-functor
$$\appl{Et_{\bar{X}}}{\mathfrak{Top}}{\bar{U}}{\bar{U}_W}$$
where $\mathfrak{Top}$ is the 2-category of topoi. According to Proposition \ref{prop-site-local-Licht-topos} below, there are two possible definitions for $\bar{U}_{W}$. If one defines $\bar{U}_W$ as in Definition \ref{defn-wettoposX} for any $\bar{U}$ étale over $\bar{X}$, then $\bar{U}\mapsto\bar{U}_W$ is not functorial. In order to get functoriality, we define $\bar{U}_W$ as a slice topos (see Definition \ref{def-wettoposUslice} below). The fact that $\bar{U}_W$ is equivalent to $\widetilde{(T_{\bar{U}},\mathcal{J}_{ls})}$ will be used as a technical tool in the remaining part of this paper.
\end{rem}
\begin{defn}\label{def-wettoposUslice}
Let $\bar{U}$ be an étale $\bar{X}$-scheme. We define the
\emph{Weil-\'etale topos} of $\bar{U}$ as the slice topos
$$\bar{U}_{W}:=\bar{X}_{W}/\gamma^*(\bar{U}).$$
\end{defn}
\begin{prop}
We have a pull-back of topoi
\[ \xymatrix{
\bar{U}_W\ar[d]\ar[r]^{\gamma_{\bar{U}}}&\bar{U}_{et}\ar[d]\\
\bar{X}_W\ar[r]^{\gamma}&\bar{X}_{et} }\] In other words one has an
equivalence $\bar{U}_{W}\simeq
\bar{X}_W\times_{\bar{X}_{et}}\bar{U}_{et}$, where the fiber product
is defined in the 2-category of topoi.
\end{prop}
\begin{proof}
One has a canonical equivalence $\bar{X}_{et}/\bar{U}\simeq\bar{U}_{et}$, as it follows from (see \cite{SGA4} III Prop 5.4)
$$\bar{X}_{et}/y\bar{U}:=\widetilde{(Et_{\bar{X}},\mathcal{J}_{et})}/\bar{U}\simeq
\widetilde{(Et_{\bar{X}}/\bar{U},\mathcal{J}_{ind})}=\widetilde{(Et_{\bar{U}},\mathcal{J}_{et})}=:\bar{U}_{et}.$$
We write below $\gamma^*\bar{U}$ (respectively $\bar{U}$) for the object $y(\gamma^*\bar{U})=\gamma^*(y\bar{U})$ (respectively $y\bar{U}$) of the topos $\bar{X}_W$ (respectively of $\bar{X}_{et}$). By (\cite{SGA4} IV Prop 5.11), the following commutative diagram
\[ \xymatrix{
\bar{X}_W/\gamma^*\bar{U}\ar[d]\ar[r]^{\gamma_{/\bar{U}}}&\bar{X}_{et}/\bar{U}\ar[d]\\
\bar{X}_W\ar[r]^{\gamma}&\bar{X}_{et} }\]
is a pull-back, i.e. 2-cartesian in the terminology of \cite{SGA4}, where the vertical arrows are the localization morphisms. The result then follows from the definition $\bar{U}_W:=\bar{X}_W/\gamma^*\bar{U}$.
\end{proof}

For any étale $\bar{X}$-scheme $\bar{U}$, a site for the topos
$\bar{U}_W$ is given by the category $T_{\bar{X}}/\gamma^*\bar{U}$
endowed with the topology induced by the local sections topology via
the forgetful functor $T_{\bar{X}}/\gamma^*\bar{U}\rightarrow
T_{\bar{X}}$. We want to define a site for $\bar{U}_W$ analogous to
$T_{\bar{X}}$. Let $\bar{U}$ be a connected étale $\bar{X}$-scheme.
Again, we need to make non-canonical choices.
\begin{choices}\label{choices-U}
\begin{enumerate}
\item We choose a geometric point
$q_{\bar{U}}:Spec\,\bar{F}\rightarrow\bar{U}$ over the geometric
point $Spec\,\bar{F}\rightarrow\bar{X}$ chosen in section
\ref{section-Lichtenbaum-topos} (1). In other words, the following
triangle
\[ \xymatrix{
&\bar{U}\ar[d]\\
Spec\,\bar{F}\ar[r]\ar[ru]^{q_{\bar{U}}}&\bar{X} }\] is required to
be commutative.  The geometric point $q_{\bar{U}}$ yields a
sub-extension $\bar{F}/K/F$, where $K$ is the function field of
$\bar{U}$.
\item For any closed point $u$ of $\bar{U}$ over $v\in\bar{X}$, we choose an embedding
$K_u\rightarrow\bar{F_v}$ such that the following diagram commutes.
\[ \xymatrix{
F_v\ar[r]&K_u\ar[r]&\bar{F_v}\\
F\ar[u]\ar[r]&K\ar[u]\ar[r]&\bar{F}\ar[u] }\]
\end{enumerate}
\end{choices}
Then the Weil group of $\bar{F}/K$ is defined by
$$W_K:=\varphi^{-1}G_K$$
where $\varphi:W_F\rightarrow G_F$ is the map chosen in
\ref{choices-X}(2). For any closed point $u$ of $\bar{U}$ over
$v\in\bar{X}$, the Weil group of $\bar{F}_v/K_u$ is defined by
$$W_{K_u}:=\varphi_v^{-1}G_{K_u}$$
where $\varphi_v:W_{F_v}\rightarrow G_{F_v}$ is the map chosen in
\ref{choices-X}(4). Finally, the Weil map
$\theta_v:W_{F_v}\rightarrow W_F$ of Data \ref{choices-X}(6) induces
a Weil map
$$\theta_u:W_{K_u}\rightarrow W_K.$$

\begin{defn}
Let $\bar{U}$ be a connected étale $\bar{X}$-scheme endowed with the
data \ref{choices-U}. We consider the category $T_{\bar{U}}$ of
objects $(Z_0,Z_u,f_u)_{u\in\bar{U}}$ defined as follows. The space $Z_0$
is locally compact and given with a continuous
action of $W_K$. For any point $u$ of $\bar{U}$, $Z_u$ is a locally compact topological space endowed with a continuous action of
$W_{k(u)}$. The map $f_u:Z_u\rightarrow Z_0$ is continuous and
$W_{{K_u}}$-equivariant .

The action of $W_K$ on $Z_0$ factors through $W_{L/K}$ for a finite
Galois sub-extension $\bar{F}/L/K$. The map $f_u$ is an
\emph{homeomorphism for almost all points $u$ of $\bar{U}$ and a
continuous injective map for all points of $\bar{U}$}. An arrow
$\phi:(Z_0,Z_u,f_u)\rightarrow(Z'_0,Z'_u,f'_u)$ in the category
$T_{\bar{U}}$ is a $W_K$-equivariant continuous map
$\phi:Z_0\rightarrow Z'_0$ inducing a continuous map $Z_u\rightarrow
Z'_u$ for any $u\in\bar{U}$. The category $T_{\bar{U}}$ is endowed
with the local section topology $\mathcal{J}_{ls}$.
\end{defn}
The argument of the proof of Proposition
\ref{prop-morph-sites-etale-loc-sections} gives a morphism of topoi
$$\widetilde{(T_{\bar{U}},\mathcal{J}_{ls})}\longrightarrow\bar{U}_{et}.$$
Moreover, the choices \ref{choices-U}
above define a morphism of topoi
$$\widetilde{(T_{\bar{U}},\mathcal{J}_{ls})}\longrightarrow\widetilde{(T_{\bar{X}},\mathcal{J}_{ls})}=:\bar{X}_{W}.$$
Indeed, we have a morphism of left exact sites
$$
\appl{(T_{\bar{X}},\mathcal{J}_{ls})}{(T_{\bar{U}},\mathcal{J}_{ls})}{(Z_0,Z_v,f_v)_{v\in\bar{X}}}{(Z_0,Z_u,f_u)_{u\in\bar{U}}}
$$
defined as follows. The space $Z_0$ on the right hand side is given with the action of $W_K$ induced by the morphism $W_K\hookrightarrow W_F$. For any closed point $u$ of $\bar{U}$ lying above $v\in\bar{X}$, the space $Z_u$ is $Z_v$ endowed with the action of $W_{k(u)}$ induced by the morphism $W_{k(u)}\hookrightarrow W_{k(v)}$, which in turn induced by the morphism $W_{K_u}\hookrightarrow W_{F_v}$.

We obtain a commutative square
\[ \xymatrix{
\widetilde{(T_{\bar{U}},\mathcal{J}_{ls})}\ar[d]\ar[r]&\bar{U}_{et}\ar[d]\\
\bar{X}_W\ar[r]^{\gamma}&\bar{X}_{et} }\]
since the corresponding diagram of sites is commutative. By the universal property
of fiber products in the 2-category of topoi, this commutative square gives rise to an essentially unique morphism
$$\widetilde{(T_{\bar{U}},\mathcal{J}_{ls})}\longrightarrow\bar{X}_W\times_{\bar{X}_{et}}\bar{U}_{et}\simeq \bar{U}_W.$$
\begin{prop}\label{prop-site-local-Licht-topos}
Let $\bar{U}$ be a connected étale $\bar{X}$-scheme endowed with the
data \ref{choices-U}. Then the morphism defined above
$$\widetilde{(T_{\bar{U}},\mathcal{J}_{ls})}\longrightarrow\bar{U}_W.$$
is an equivalence.
\end{prop}
\begin{proof}
Recall that $\gamma^*\bar{U}=(U_0,U_v,h_v)$, where
$U_0:=Hom_{\bar{X}}(Spec\,\bar{F},\bar{U})$ as a $W_F$-set. The
sub-extension $\bar{F}/K/F$ given by the point $q_{\bar{U}}$ yields an isomorphism of $W_F$-sets
$$U_0:=Hom_{F}(K,\bar{F})\simeq G_F/G_K\simeq W_F/W_K,$$
sending $q_{\bar{U}}\in U_0$ to the distinguished element of $W_F/W_K$.
This gives an isomorphism of categories
$$B_{Top}W_F/U_0\simeq B_{Top}W_F/(W_F/W_K).$$ Hence the functor
$$
\fonc{\Psi_0}{B_{Top}W_F/U_0}{B_{Top}W_K}{\phi_0:Z_0\rightarrow U_0}{Z_{u_0}:=\phi_0^{-1}(q_{\bar{U}})}
$$
is an equivalence of categories. Let $\phi:(Z_0,Z_v,f_v)\rightarrow (U_0,U_v,h_v)$ be an object of the slice category $T_{\bar{X}}/\gamma^*\bar{U}$ and let $u\in\bar{U}$ be a closed point lying above $v\in\bar{X}$. The action of $W_{K_u}$ on $$f_v^{-1}(Z_{u_0})\hookrightarrow Z_{u_0}:=\phi_0^{-1}(q_{\bar{U}})$$
via the map $W_{K_u}\rightarrow W_K$ is unramified, in the sense that $W^1_{K_u}$ acts trivially on $f_v^{-1}(Z_{u_0})$. Then we see $f_v^{-1}(Z_{u_0})$ as a $W_{k(u)}$-space, where $f_v^{-1}(Z_{u_0})$ is given with the topology induced by the inclusion $f_v^{-1}(Z_{u_0})\subseteq Z_v$. We define $Z_u$ to be the space
$$Z_u:=f_v^{-1}(Z_{u_0})$$
endowed with its $W_{k(u)}$-action. Finally, the $W_{F_v}$-equivariant map $f_v:Z_v\rightarrow Z_0$ induces a
$W_{K_u}$-equivariant map $g_u:Z_u\rightarrow Z_{u_0}$, which is injective and continuous everywhere and an homeomorphism almost everywhere. Then the equivalence $\Psi_0$ induces a functor
$$
\fonc{\Psi}{T_{\bar{X}}/\gamma^*\bar{U}}{T_{\bar{U}}}{(Z_0,Z_v,f_v)\rightarrow (U_0,U_v,h_v)}{(Z_{u_0},Z_u,g_u)}
$$
which is an equivalence as well. Moreover, the topology induced on $B_{Top}W_K$ by the local section
topology on $B_{Top}W_F$ via the functor (forget the map to $U_0$)
$$B_{Top}W_K\simeq
B_{Top}W_F/U_0\rightarrow B_{Top}W_F$$ is still the local section
topology on $B_{Top}W_K$. The same is true for any place $v$ of $F$, and we obtain an equivalence of sites:
$$(T_{\bar{X}}/\gamma^*\bar{U},\mathcal{J}_{ls})\longrightarrow (T_{\bar{U}},\mathcal{J}_{ls})$$
Therefore the induced morphism of topoi
$$\widetilde{(T_{\bar{U}},\mathcal{J}_{ls})}\longrightarrow\widetilde{(T_{\bar{X}}/\gamma^*\bar{U},\mathcal{J}_{ls})}
\simeq
\widetilde{(T_{\bar{X}},\mathcal{J}_{ls})}/\gamma^*y(\bar{U})=:\bar{U}_W$$
is an equivalence (see \cite{SGA4} III Prop. 5.4 for the last equivalence).
\end{proof}

\subsection{The Weil-\'etale topos of a function field}\label{subsect-WET-fction-field-case}

In this section we show that our definition of the (small) Weil-\'etale topos of a function field coincides with the definition given by Lichtenbaum in \cite{Lichtenbaum-finite-field}. More precisely, let $Y$ be an open subscheme of a smooth projective curve over a finite field $k$. The most natural definition for the Weil-\'etale topos is given by the category $\mathcal{S}_{et}(W_k,\overline{Y})$ of $W_k$-equivariant \'etale sheaves on the geometric curve $\overline{Y}=Y\otimes_k\overline{k}$. On the other hand, Definition \ref{defn-TX} yields a left exact category $T^{sm}_{Y}$ endowed with the local section topology $\mathcal{J}_{ls}$, where we replace $Top$ by $\underline{Sets}$. We define below an equivalence
$$\widetilde{(T^{sm}_{Y},\mathcal{J}_{ls})}\simeq\mathcal{S}_{et}(W_k,\overline{Y}).$$
In other words, we show that the artificial definition of the (small) Weil-\'etale topos coincides with the natural one in the case of a function field. This justifies the term "Weil-\'etale topos" for the topos defined in this paper.
\begin{choices}\label{choices-Y}
Let $Y$ be an open subscheme of a geometrically connected smooth projective curve over a finite field $k$ with function field $K$.
\begin{enumerate}
\item We choose a separable algebraic closure $\bar{K}/K$.
\item For closed point $y$ of $Y$, we choose a separable algebraic
closure $\bar{K_v}/K_v$ and a $K$-embedding $\bar{K}\rightarrow \bar{K_v}$.
\end{enumerate}
\end{choices}
We have a natural map $G_K\rightarrow G_k$ and the global Weil group $W_K$ is defined as the fiber product topological group
$W_K:=G_K\times_{G_k}W_k$. For any closed point $v$ of $Y$, one has $G_{k(v)}=G_{K_v}/I_{K_v}$, and $W_{K_v}:=G_{K_v}\times_{G_{k(v)}}W_{k(v)}$. There exists a unique Weil map $W_{K_v}\rightarrow W_K$ such that the following diagram is commutative
\[ \xymatrix{
W_{K_v}\ar[d]\ar[r]&G_{K_v}\ar[d]\\
W_K\ar[r]&G_K }\]

\begin{defn}
Let $T^{sm}_{Y}$ be the category of objects $(Z_0,Z_v,f_v)$ defined
as follows. The set $Z_0$ is endowed with a continuous
$W_K$-action. For any closed point $v$ of $Y$, $Z_v$ is a set
endowed with a continuous $W_{k(v)}$-action. The map
$f_v:Z_v\rightarrow Z_0$ is $W_{{K_v}}$-equivariant, when $Z_v$ and
$Z_0$ are seen as $W_{{K_v}}$-spaces via the maps
$W_{K_v}\rightarrow W_{K}$ and $q_v:W_{K_v}\rightarrow
W_{k(v)}$. We require the following facts:
\begin{itemize}
\item The map $f_v$ is \emph{bijective for almost all
closed points and injective for all closed points $v$ of $Y$}.
\item The action of $W_K$ on $Z_0$ factors through $W_{L/K}$, for
some finite Galois subextension $\bar{K}/L/K$.
\end{itemize}
A \emph{morphism} $$\phi:(Z_0,Z_v,f_v)\longrightarrow(Z'_0,Z'_v,f'_v)$$
in the category $T^{sm}_{Y}$ is a $W_K$-equivariant map
$\phi:Z_0\rightarrow Z'_0$ \emph{inducing} a $W_{K_v}$-equivariant map
$\phi_v:Z_v\rightarrow Z_v$ for all closed points $v$ of $Y$.

The category $T^{sm}_{Y}$ is endowed with the local section
topology $\mathcal{J}_{ls}$, i.e. the topology generated by the
pretopology for which a family
$$\{\varphi_i:(Z_{i,0},Z_{i,v},f_{i,v})\rightarrow
(Z_0,Z_v,f_v),\,i\in I\}$$ is a covering family if the map
$\coprod_{i\in I}Z_{i,v}\rightarrow Z_v$ is surjective, for any
point $v$ of $Y$.
\end{defn}

\begin{defn}
We define the \emph{small Weil-\'etale topos} $Y^{sm}_W$ as the topos of sheaves on the site $(T^{sm}_{Y},\mathcal{J}_{ls})$.
\end{defn}

Let $K\overline{k}$ be the function field of the geometric curve $\overline{Y}$. We have the sub-extension $\overline{K}/K\overline{k}/K$ and we set $G_{K\overline{k}}:=G(\overline{K}/K\overline{k})$. For any closed point $y\in Y$, we denote by $I_y$ the Galois group of the completion of $K\overline{k}$ at $y$. We choose maps $I_y\hookrightarrow G_{K\overline{k}}$ compatible with Data \ref{choices-Y}.

\begin{defn}
Let $T^{sm}_{\overline{Y}}$ be the category of objects $(Z_0,Z_y,f_y)$, where $y$ runs over the closed points of $\bar{Y}$, defined
as follows. The set $Z_0$ is endowed with a continuous
$G_{K\overline{k}}$-action. For any closed point $y$ of $\overline{Y}$, $Z_y$ is a set
endowed with a $I_y$-equivariant a map $f_y:Z_y\rightarrow Z_0$, where $I_y$ acts trivially on $Z_y$ and $I_y$ acts on
$Z_0$ via the map $I_y\rightarrow G_{K\overline{k}}$. Moreover, we assume that
\begin{itemize}
\item The map $f_y$ is \emph{bijective for almost all
closed points $y$ of $\overline{Y}$ and injective for all closed points $y$ of $Y$}.
\item The action of $G_{K\overline{k}}$ on $Z_0$ factors through $G(L/K\overline{k})$, for
some finite Galois sub-extension $\bar{K}/L/K\overline{k}$.
\end{itemize}
The morphisms in the category $T^{sm}_{\overline{Y}}$ are defined as above. The local section topology $\mathcal{J}_{ls}$ on the category  $T^{sm}_{\overline{Y}}$ is generated by the pretopology of surjective families as above.
\end{defn}
We consider below the category $Et_{\overline{Y}}$ of finitely presented \'etale $\overline{Y}$-schemes. The site $(Et_{\overline{Y}},\mathcal{J}_{et})$ is called the \emph{restricted \'etale site}. Since $\overline{Y}$ is quasi-compact and quasi-separated, the restricted \'etale site $(Et_{\overline{Y}},\mathcal{J}_{et})$ is a site for the \'etale topos of $\overline{Y}$, i.e. we have
$$\overline{Y}_{et}=\widetilde{(Et_{\overline{Y}},\mathcal{J}_{et})}.$$
\begin{prop}\label{prop-equiv-sites-etale-barY}
There is an equivalence
$$\widetilde{(T^{sm}_{\overline{Y}},\mathcal{J}_{ls})}\simeq\overline{Y}_{et}.$$
\end{prop}
\begin{proof}
The arguments of Proposition \ref{prop-morph-sites-etale-loc-sections} can be generalized to this context. This yields a natural functor
$Et_{\overline{Y}}\rightarrow T^{sm}_{\overline{Y}}$. This functor is not essentially surjective because an object $(Z_0,Z_y,f_y)$ of $T^{sm}_{\overline{Y}}$ can have an infinite number of connected components (i.e. $Z_0/G_{K\overline{k}}$ is infinite), while a finitely presented \'etale $\overline{Y}$-scheme has finitely many connected components. However, the previous functor is fully faithful, $Et_{\overline{Y}}$ is a topologically generating family of the site $(T^{sm}_{\overline{Y}},\mathcal{J}_{ls})$, and the \'etale topology on $Et_{\overline{Y}}$ is induced by the local section topology on $T^{sm}_{\overline{Y}}$. Hence the result follows from (\cite{SGA4} IV Corollary 1.2.1).
\end{proof}
We recall below some basic facts concerning truncated simplicial topoi. We refer to (\cite{these} Chapter 10 Section 1.2) for more details and references. A truncated simplicial topos $\mathcal{S}_{\bullet}$ is given by the usual diagram
$$\xymatrix{ \mathcal{S}_2\hspace{0.4cm} \tar[r] & \hspace{0.4cm}\mathcal{S}_1\hspace{0.4cm} \wdar[r] & \hspace{0.4cm}\mathcal{S}_0\ar[l]}$$
Given such truncated simplicial topos $\mathcal{S}_{\bullet}$, we define the category $Desc(\mathcal{S}_{\bullet})$ of objects of $S_0$ endowed with a descent data. One can prove that $Desc(\mathcal{S}_{\bullet})$ is always a topos. More precisely, $Desc(\mathcal{S}_{\bullet})$ is the inductive limit of the diagram $\mathcal{S}_{\bullet}$ in the 2-category of topoi.

The most simple non-trivial example is the following. Let $\mathcal{S}$ be a topos and let $U$ be an object of $\mathcal{S}$. We consider the truncated simplicial topos
$$(\mathcal{S},U)_{\bullet}:\hspace{0.5cm}\xymatrix{ \mathcal{S}/(U\times U\times U)\hspace{0.4cm} \tar[r] & \hspace{0.4cm}\mathcal{S}/(U\times U)\hspace{0.4cm} \wdar[r] & \hspace{0.4cm}\mathcal{S}/U\ar[l]}$$
where these morphisms of topoi are induced by the projections maps (of the form $U\times U\times U\rightarrow U\times U$ and $U\times U\rightarrow U$) and by the diagonal map $U\rightarrow U\times U$. It is well known that, if $U$ covers the final object of $\mathcal{S}$, then the natural morphism
$$Desc(\mathcal{S},U)_{\bullet}\longrightarrow\mathcal{S}$$
is an equivalence. In other words $\mathcal{S}/U\rightarrow\mathcal{S}$ is an effective descent morphism for any $U$ covering the final object of $\mathcal{S}$.

We will also use the following example. Let $G$ be a discrete group acting on a scheme $\bar{Y}$. The truncated simplicial scheme
$$\xymatrix{G\times G\times\overline{Y}\hspace{0.4cm} \tar[r] & \hspace{0.4cm}G\times\overline{Y}\hspace{0.4cm} \wdar[r] & \hspace{0.4cm}\overline{Y}\ar[l]}$$
defined by the action of the group $G$ on $\overline{Y}$ induces a truncated simplicial topos :
$$(G,\overline{Y}_{et})_{\bullet}:\hspace{0.5cm}\xymatrix{ (G\times G\times\overline{Y})_{et}\hspace{0.4cm} \tar[r] & \hspace{0.4cm}(G\times\overline{Y})_{et}\hspace{0.4cm} \wdar[r] & \hspace{0.4cm}\overline{Y}_{et}\ar[l]}$$
The descent topos of this truncated simplicial topos is precisely the category of $G$-equivariant \'etale sheaves on $\bar{Y}$ :
$$\mathcal{S}_{et}(G,\overline{Y}):=Desc((G,\overline{Y}_{et})_{\bullet}).$$

\begin{thm}\label{thm-equ-WEtoposY-nature-artif}
There is an equivalence
$$\widetilde{(T^{sm}_{Y},\mathcal{J}_{ls})}\simeq\mathcal{S}_{et}(W_k,\overline{Y}).$$
\end{thm}
\begin{proof}
Firstly, there is a canonical morphism of topoi
$$f:\widetilde{(T^{sm}_{Y},\mathcal{J}_{ls})}\longrightarrow B^{sm}_{W_k}$$
induced by the morphism $f^*$ of left exact sites defined as follows. The functor $f^*$ sends a $W_k$-set $Z$ to the object $(Z,Z,Id_Z)$ of $T^{sm}_{Y}$, where $W_K$ (respectively $W_{k(v)}$) acts on $Z$ via the map $W_K\rightarrow W_k$ (respectively via $W_{k(v)}\rightarrow W_k$). Let $EW_k$ be the object of $B^{sm}_{W_k}$ defined by the action of $W_k$ on itself by multiplications. One has $f^*(EW_k)=y(EW_k,EW_k,Id)$. Adapting the proof of Proposition \ref{prop-site-local-Licht-topos} to this context, we obtain the following equivalences :
$$\widetilde{(T^{sm}_{Y},\mathcal{J}_{ls})}/f^*(EW_k)\simeq\widetilde{(T^{sm}_{Y}/f^*EW_k,\mathcal{J}_{ls})}\simeq \widetilde{(T^{sm}_{\overline{Y}},\mathcal{J}_{ls})}.$$
By Proposition \ref{prop-equiv-sites-etale-barY}, we have
\begin{equation}\label{equiv-barY-et-localisation}
Y^{sm}_W/f^*(EW_k):=\widetilde{(T^{sm}_{Y},\mathcal{J}_{ls})}/f^*(EW_k)\simeq\widetilde{(T^{sm}_{\overline{Y}},\mathcal{J}_{ls})}\simeq
\widetilde{(Et_{\overline{Y}},\mathcal{J}_{et})}=:\overline{Y}_{et}.
\end{equation}
The morphism from $f^*EW_k$ to the final object of $\widetilde{(T^{sm}_{Y},\mathcal{J}_{ls})}$ is certainly an epimorphism, i.e. a covering morphism. We consider the truncated simplicial topos $\mathcal{S}^1_{\bullet}$ obtained by localization:
$$\xymatrix{Y_W^{sm}/(f^*EW_k\times f^*EW_k \times f^*EW_k) \tar[r] & Y_W^{sm}/(f^*EW_k\times f^*EW_k) \wdar[r] & Y_W^{sm}/f^*EW_k\ar[l]}$$
The descent topos of this truncated simplicial topos $\mathcal{S}^1_{\bullet}$ is canonically equivalent to $Y^{sm}_W$ since $f^*EW_k$ covers the final object:
$$Y^{sm}_W\simeq Desc(\mathcal{S}^1_{\bullet})$$
By (\ref{equiv-barY-et-localisation}), the truncated simplicial topos $\mathcal{S}^1_{\bullet}$ is equivalent to
$$\mathcal{S}^2_{\bullet}:\hspace{0.5cm}\xymatrix{ \overline{Y}_{et}/(g^*W_k\times g^*W_k)\hspace{0.4cm} \tar[r] & \hspace{0.4cm}\overline{Y}_{et}/g^*W_k\hspace{0.4cm} \wdar[r] & \hspace{0.4cm}\overline{Y}_{et}\ar[l]}$$
where $g:\overline{Y}_{et}\rightarrow\underline{Sets}$ is the unique morphism. One has
$(g^*W_k\times g^*W_k)=g^*(W_k\times W_k)$ and  $$g^*W_k=g^*(\coprod_{W_k}\{*\})=\coprod_{W_k}g^*(\{*\})=\coprod_{W_k}y(\overline{Y})=y(\coprod_{W_k}\overline{Y})=y(W_k\times\overline{Y})$$ where $\{*\}$ and $y(\overline{Y})$ are the final objects of $\underline{Sets}$ and $\overline{Y}_{et}$ respectively, since $g^*$ commutes with finite projective limits and arbitrary inductive limits. We obtain
$$\overline{Y}_{et}/g^*W_k=\overline{Y}_{et}/y(W_k\times\overline{Y})=(W_k\times\overline{Y})_{et}$$
since the projection $W_k\times\overline{Y}\rightarrow\overline{Y}$ is an \'etale morphism of schemes, and
$$\overline{Y}_{et}/(g^*W_k\times g^*W_k)=\overline{Y}_{et}/y(W_k\times W_k\times\overline{Y})=(W_k\times W_k\times\overline{Y})_{et}$$
The truncated simplicial topos $\mathcal{S}^2_{\bullet}$ is equivalent to
$$\mathcal{S}^3_{\bullet}:\hspace{0.5cm}\xymatrix{(W_k\times W_k\times\overline{Y})_{et}\hspace{0.4cm} \tar[r] & \hspace{0.4cm}(W_k\times\overline{Y})_{et}\hspace{0.4cm} \wdar[r] & \hspace{0.4cm}\overline{Y}_{et}\ar[l]}$$
where the maps of this simplicial topos are given by the group structure of $W_k$ and its action on $\overline{Y}$. Hence we have equivalences of truncated simplicial topoi :
$$\mathcal{S}^1_{\bullet}\simeq \mathcal{S}^2_{\bullet}\simeq \mathcal{S}^3_{\bullet}$$
inducing equivalences between the associated descent topoi:
$$Y^{sm}_W\simeq Desc(\mathcal{S}^1_{\bullet})\simeq Desc(\mathcal{S}^2_{\bullet})\simeq Desc(\mathcal{S}^3_{\bullet})\simeq\mathcal{S}_{et}(W_k,\overline{Y}).$$

\end{proof}

Recall that Lichtenbaum has defined in \cite{Lichtenbaum-finite-field} the Weil-\'etale cohomology as follows:
$$H_W^n(Y,\mathcal{A}):=R^n(\Gamma^{W_k}_{\overline{Y}})\mathcal{A}$$
Here $\mathcal{A}$ is a $W_k$-equivariant abelian \'etale sheaf on $\overline{Y}$ and $\Gamma^{W_k}_{\overline{Y}}$ is the functor of $W_k$-invariant global sections on $\overline{Y}$. This cohomology is precisely the cohomology of the Weil-\'etale topos $\mathcal{S}_{et}(W_k,\overline{Y})$. Indeed, the latter is defined as $R^n(\alpha_*)\mathcal{A}$
where $\alpha:\mathcal{S}_{et}(W_k,\overline{Y})\rightarrow \underline{Sets}$ is the unique map from the Weil-\'etale topos to the final topos. But we have canonically $\alpha_*=\Gamma^{W_k}_{\overline{Y}}$.

\section{The Weil-\'etale fundamental group}\label{sect-fund-group}

\subsection{Local sections}
For $W$ a locally compact topological group and $I$ a closed subgroup of $W$, it is not known in general that the continuous projection $W\rightarrow W/I$ admits local sections. The result below, due to P. Mostert , shows that local sections do exist when $W/I$ is finite dimensional. We denote below by $\textrm{dim}(X)$ the covering dimension of the space $X$ in the sense of Lebesgue. Note that, for any locally compact space $X$, we have the inequality
$$\textrm{cd}(X)\leq\textrm{dim}(X)$$
where $\textrm{cd}(X)$ denotes the cohomological dimension that is used in \cite{local sections}.
\begin{thm}\label{thm-loc-sections}
Let $W$ be a locally compact topological group and let $I\subseteq W$ be a closed subgroup such that $W/I$ is finite dimensional. Then the continuous projection $W\rightarrow W/I$
has local sections. If $W/I$ is 0-dimensional, then the projection $W\rightarrow W/I$ has a global continuous section.
\end{thm}
\begin{proof}
This is \cite{local sections} Theorem 8.
\end{proof}

\begin{cor}\label{cor-local-sections}
Let $W$ be a finite dimensional locally compact topological group and let $I\subseteq W$ be a closed subgroup. Then the continuous projection $W\rightarrow W/I$
has local sections.
\end{cor}
\begin{proof}
By \cite{local sections} Corollary 2, if $W$ is finite dimensional then so is $W/I$, and the result follows from the previous theorem.
\end{proof}

\begin{lem}\label{lem-C_L-finite-dim}
Let $L/K$ be a finite Galois extension of number fields. The the idèle class group $C_L$ and relative Weil group $W_{L/K}$ have both finite dimension.
\end{lem}
\begin{proof}
Let us first note that $C_L$ is an open (and closed) subgroup of $W_{L/K}$, hence $$\textrm{dim}(W_{L/K})=\textrm{dim}(C_L).$$
Global class field theory provides us with an exact sequence of topological groups
$$1\rightarrow C^0_L\rightarrow C_L\rightarrow G_L^{ab}\rightarrow 1,$$
where $C_L^0$ is the connected component of $C_L$. We mean by the term exact sequence that $C_L^0$ is a closed normal subgroup of $C_L$ endowed with the induced topology and that $G_L^{ab}$ becomes isomorphic to the group $C_L/C_L^0$ endowed with the quotient topology. The space $G_L^{ab}$ is profinite hence compact and totally disconnected. Hence $\textrm{dim}(G_L^{ab})=0$. By Theorem \ref{thm-loc-sections}, the continuous map $C_L\rightarrow G_L^{ab}$ has a global continous section. We obtain an homeomorphism $C_L\simeq C_L^0\times G_L^{ab}$ (which is not a group morphism). We obtain
$$\textrm{dim}(C_L)\leq \textrm{dim}(C^0_L)+ \textrm{dim}(G_L^{ab})= \textrm{dim}(C^0_L).$$
But the connected component $C^0_L$ is the product of $\mathbb{R}$ with $r_1(L)+r_2(L)-1$ solenoids and $r_2(L)$ circles. Recall that a solenoid is a filtered projective limit of circles:
$$\mathbb{V}:=\underleftarrow{lim}\,\mathbb{S}^1$$
hence $\mathbb{V}$ is of dimension 1. We obtain
$$\textrm{dim}(W_{L/K})=\textrm{dim}(C_L)\leq \textrm{dim}(C^0_L)\leq r_1(L)+2\,r_2(L)=[L:\mathbb{Q}]$$
\end{proof}

\begin{cor}\label{cor-exple-of-loc-sections}
Let $L/K$ be a finite Galois extension of number fields. The map $W_K\rightarrow W_{L/K}$ has local sections. The relative Weil group $W_{L/K,S}$ has finite dimension and $W_K\rightarrow W_{L/K,S}$ has local sections. For two Galois extensions $L'/L/K$, the map $W_{L'/K,S}\rightarrow W_{L/K,S}$ has local sections.
\end{cor}
\begin{proof}
The group $W_{L/K,S}$ is the quotient of $W_{L/K}$ by a closed subgroup. Hence $W_{L/K,S}$ has finite dimension by Lemma \ref{lem-C_L-finite-dim} and \cite{local sections} Corollary 2. The maps $W_K\rightarrow W_{L/K}$, $W_K\rightarrow W_{L/K,S}$ and
$W_{L'/K,S}\rightarrow W_{L/K,S}$ are all quotient maps of locally compact groups by closed subgroups with finite dimensional targets. Those results follow from Theorem \ref{thm-loc-sections}.
\end{proof}

\subsection{Weil groups}\label{subsect-Weil-group}

\subsubsection{}
Again we consider a connected étale $\bar{X}$-scheme $\bar{U}$
endowed with the data \ref{choices-U}. Thus we have a geometric
point $q_{\bar{U}}:Spec\,\bar{F}\rightarrow\bar{U}$ over
$Spec\,\bar{F}\rightarrow\bar{X}$, i.e. a sub-extension
$\bar{F}/K/F$, where $K$ is the function field of $\bar{U}$. The
Weil group of $K$ is $W_K:=\varphi^{-1}G_K$. If $u$ is a closed
point of $\bar{U}$ lying over $v\in\bar{X}$, then the Weil map
$\theta_v:W_{F_v}\rightarrow W_F$ of Data \ref{choices-X}(6) induces
a Weil map $\theta_u:W_{K_u}\rightarrow W_K$.
We denote by $W^1_{K_u}\subset W_{K_u}$ the maximal compact subgroup, which should be thought of as the inertia subgroup.
\begin{defn}
We define the Weil group $W(\bar{U},q_{\bar{U}})$ of the pair
$(\bar{U},q_{\bar{U}})$ as the quotient
$$W(\bar{U},q_{\bar{U}}):=W_{K}/N_{\bar{U}}$$
where $N_{\bar{U}}$ is the closure of the normal subgroup in $W_K$
generated by the images of the maps
$$W^1_{K_u}\hookrightarrow W_{K_u}\rightarrow W_K$$
where $u$ runs through the closed points of $\bar{U}$.
\end{defn}
We will show below that this group $W(\bar{U},q_{\bar{U}})$ is the limit in the category of topological
groups of a projective system of topological groups, i.e. of a topological pro-group. We can either consider this topological pro-group or we can consider its limit as a topological group. It turns out that a topological pro-group contains more information than its limit computed in the category of topological groups. For example, there exist non-trivial
strict pro-groups whose limit, computed in the category of topological
groups, is the trivial group. Topos theory
provides a natural framework to overcome this kind of pathologies.

Let $(\bar{U},q_{\bar{U}})$ be as above and let $\bar{F}/L/K$ be a
finite Galois subextension, where the algebraic closure $\bar{F}/K$
is given by the geometric point $q_{\bar{U}}$.
\begin{defn}\label{def-W_LU}
We consider the topological group $W(\bar{U},L)$ defined
as the quotient
$$W(\bar{U},L):=W_{L/K}/N_{\bar{U},L},$$
where $N_{\bar{U},L}$ is the closure of the normal subgroup in
$W_{L/K}$ generated by the images of the maps
$$W^1_{K_u}\hookrightarrow W_{K_u}\rightarrow W_{L/K}$$
where $u$ runs through the closed points of $\bar{U}$.
\end{defn}

\begin{defn}\label{def-pro-top-grp}
We denote by $\underline{W}(\bar{U},q_{\bar{U}})$ the strict
topological pro-group
$$\underline{W}(\bar{U},q_{\bar{U}})\,:=\,\{{W}(\bar{U},L)\,\,;\,\,\mbox{$\overline{F}/L/K$ \emph{finite Galois}}\}$$
indexed over the system of finite Galois subextension of
$\bar{F}/K$.
\end{defn}
Recall that the term "strict" means that the transition maps
$${W}(\bar{U},L')\longrightarrow {W}(\bar{U},L)$$
have local sections. This follows from Theorem \ref{thm-loc-sections}, Lemma \ref{lem-C_L-finite-dim} and from the fact that ${W}(\bar{U},L)$ is the quotient of $W_{L/K}$ by a compact subgroup.

\begin{prop}\label{prop-prop8}
The canonical morphism of topological groups
$$\alpha:{W}(\bar{U},q_{\bar{U}})\longrightarrow\underleftarrow{lim}\,{W}(\bar{U},L)$$
is an isomorphism, where the right hand side is the projective limit computed in the category of topological groups. Moreover the map
$${W}(\bar{U},q_{\bar{U}})\longrightarrow\,{W}(\bar{U},L)$$
has local sections.
\end{prop}

\begin{proof}
Consider the product decompositions
$${W}(\bar{U},q_{\bar{U}})\simeq{W}^1(\bar{U},q_{\bar{U}})\times\mathbb{R}\mbox{ and }
{W}(\bar{U},L)\simeq{W}^1(\bar{U},L)\times\mathbb{R}$$ where
$W^1(-)$ is the maximal compact subgroup of $W(-)$. We have
$$\underleftarrow{lim}\,{W}(\bar{U},L)\simeq\underleftarrow{lim}\,{W}^1(\bar{U},L)\times\mathbb{R}$$
since projective limits commute between themselves (in particular with products). Hence it is enough to show that
$$\alpha^1:{W}^1(\bar{U},q_{\bar{U}})\longrightarrow\underleftarrow{lim}\,{W}^1(\bar{U},L)$$
is an isomorphism of topological groups. For any $L/K$ finite
Galois, we have an exact sequence of topological groups:
$$1\rightarrow N_{\bar{U},L}\rightarrow W^1_{L/K}\rightarrow{W}^1(\bar{U},L)\rightarrow
1$$ Passing to the limit we obtain an exact sequence (since projective limits are left exact)
\begin{equation}\label{uneexactsequ}
1\rightarrow \underleftarrow{lim}\,N_{\bar{U},L}\rightarrow
\underleftarrow{lim}\,W^1_{L/K}\rightarrow\underleftarrow{lim}\,{W}^1(\bar{U},L)
\end{equation}
By definition of $N_{\bar{U}}$ and $N_{\bar{U},L}$, the inclusion
$N_{\bar{U}}\hookrightarrow W^1_{K}$ factors through
$\underleftarrow{lim}\,N_{\bar{U},L}$. We obtain an injective
continuous map
$$n:N_{\bar{U}}\hookrightarrow \underleftarrow{lim}\,N_{\bar{U},L}$$
which has dense image, since all the maps $N_{\bar{U}}\hookrightarrow N_{\bar{U},L}$ are surjective.
This morphism $n$ is an isomorphism of topological
groups because $N_{\bar{U}}$ is compact. The exact sequence
(\ref{uneexactsequ}) then shows that $\alpha^1$ is injective. On the other hand, the map
$${W}^1(\bar{U},q_{\bar{U}})\longrightarrow\,{W}^1(\bar{U},L)$$
is surjective for any $L$, hence $\alpha^1$ has dense image. Therefore
$\alpha^1$ is surjective and bicontinuous, since ${W}^1(\bar{U},q_{\bar{U}})$ is
compact. Finally, the map $${W}(\bar{U},q_{\bar{U}})\longrightarrow\,{W}(\bar{U},L)$$ has local sections by Theorem \ref{thm-loc-sections}.
\end{proof}

\subsubsection{}
Let $\bar{V}$ be another connected étale $\bar{X}$-scheme. The generic point of
$\bar{V}$ is the spectrum of a number field $L$ and we denote by $S$ the finite set of places
of $L$ not corresponding to a point of $\bar{V}$. The $S$-idèle class group $C_{L,S}$ of $L$ is defined by the following
exact sequence of topological groups
\begin{equation}\label{exact-sequence-S-idele}
0\rightarrow
\prod_{w\in\bar{V}}\mathcal{O}^{\times}_{L_w}\rightarrow
C_L\rightarrow C_{L,S}\rightarrow0
\end{equation}
where $\prod_{w\in\bar{V}}\mathcal{O}^{\times}_{L_w}$ is the product
of the local units
$\mathcal{O}^{\times}_{L_w}:=Ker(L_w^*\rightarrow\mathbb{R}_{>0})$
indexed by the sets of places of $L$ corresponding to a point of
$\bar{V}$.
\begin{defn}\label{def-CU}
For any connected étale $\bar{X}$-scheme $\bar{V}$ with function
field $L$, we define the \emph{formation module $C_{\bar{V}}$ of
$\bar{V}$} as the $S$-idèle class group of $L$
$$C_{\bar{V}}:=C_{L,S}$$
where $S$ is the set of places
of $L$ not corresponding to a point of $\bar{V}$.
\end{defn}

The geometric point $q_{\bar{U}}:Spec\,\bar{F}\rightarrow\bar{U}$
gives a point of the étale topos
$$q_{\bar{U}}:\underline{Sets}\longrightarrow\bar{U}_{et}$$
and the étale fundamental group $\pi_1(\bar{U}_{et},q_{\bar{U}})$ is
well defined as a profinite group. This group is the Galois group of
the maximal sub-extension of $\bar{F}/K$ unramified at any place of
$K$ corresponding to of $\bar{U}$ (regardless if such a place is
ultrametric or archimedean). More geometrically, we consider the
filtered set of pointed Galois étale cover
$\{(\bar{V},q_{\bar{V}})\rightarrow(\bar{U},q_{\bar{U}})\}$ to
define the étale fundamental group
$$\pi_1(\bar{U}_{et},q_{\bar{U}}):=\underleftarrow{lim}_{(\bar{V},q_{\bar{V}})}\,Gal(\bar{V}/\bar{U})$$
The pair
\begin{equation}\label{class-formation-pi}
(\pi_1(\bar{U}_{et},q_{\bar{U}}),\underrightarrow{lim}_{(\bar{V},q_{\bar{V}})}\,C_{\bar{V}})
\end{equation}
is a (topological) class formation (see \cite{Neukirch} Proposition 8.3.8 and \cite{Neukirch} Theorem 8.3.12). This follows from the fact that if $L/K$ is a Galois extension unramified over $\bar{U}$, then the
$G_{L/K}$-module $\prod_{w\in\bar{V}}\mathcal{O}^{\times}_{L_w}$ in
(\ref{exact-sequence-S-idele}) is cohomologically trivial. Therefore, one can consider the Weil group associated to this class formation (see \cite{Tate}).
More precisely, one has a compatible system of fundamental class
leading to a projective system of extensions
$$1\rightarrow C_{\bar{V}}\rightarrow W_{\bar{V}/\bar{U}}\rightarrow Gal(\bar{V}/\bar{U})\rightarrow 1.$$
This projective system is indexed by the filtered set of pointed
Galois cover of $(\bar{U},q_{\bar{U}})$.
\begin{defn}
The \emph{Weil group of the class formation
(\ref{class-formation-pi})} is the projective limit
$$W_{\bar{U},q_{\bar{U}}}:=\underleftarrow{lim}_{(\bar{V},q_{\bar{V}})}\,W_{\bar{V}/\bar{U}}$$
computed in the category of topological groups.
\end{defn}
We have a canonical map
$W(\bar{U},p_{\bar{U}})\rightarrow\underleftarrow{lim}_{(\bar{V},q_{\bar{V}})}\,W_{\bar{V}/\bar{U}}=:W_{\bar{U},q_{\bar{U}}}$.

\subsubsection{}
If $W$ is an Hausdorff topological group, we denote $W^c$ the
closure of the commutator subgroup of $W$, and by $W^{ab}=W/W^c$ the
maximal Hausdorff abelian quotient of $W$.
\begin{lem}
We have topological isomorphisms
$${W}(\bar{U},q_{\bar{U}})^{ab}\simeq C_{\bar{U}}\mbox{ and }{W}(\bar{U},L)^{ab}\simeq C_{\bar{U}}$$
for any finite Galois extension $L/K$.
\end{lem}
\begin{proof}
Recall that $K$ is the number field of $\bar{U}$ and that $L/K$ is a
finite Galois extension. We have $W_{L/K}^{ab}\simeq C_K$. The
morphism $W^1_{L/K}\rightarrow{W}^1(\bar{U},L)$ is surjective and
closed. Hence ${W}(\bar{U},L)^c$ is the image of $W_{L/K}^c$. On the
other hand, the image of $N_{\bar{U},L}\subset W_{L/K}$ in $C_K$ is
$\prod_{v\in\bar{U}}\mathcal{O}_{K_v}^{\times}$. Since quotients
commute between themselves, we have
$${W}(\bar{U},L)^{ab}=W_{L/K}^{ab}/\prod_{v\in\,\bar{U}}\mathcal{O}_{K_v}^{\times}=:C_{\bar{U}}.$$
The proof concerning ${W}(\bar{U},q_{\bar{U}})$ is similar.
\end{proof}

\begin{cor}
The topological pro-group
$$\underline{W}(\bar{U},q_{\bar{U}})^{ab}\,:=\,\{{W}(\bar{U},L)^{ab}\,\,;\,\,\mbox{$L/K$ \emph{finite Galois}}\}$$
is constant and can therefore be identified with a usual
topological group. One has an isomorphism of topological groups
$$\underline{W}(\bar{U},q_{\bar{U}})^{ab}\simeq C_{\bar{U}}.$$
\end{cor}

\subsection{Normal subgroups}

Let $\mathcal{G}$ be a group object in a topos $\mathcal{S}$. A \emph{subgroup} of $\mathcal{G}$ is an
equivalence class of monomorphisms of group objects
$\mathcal{H}\hookrightarrow\mathcal{G}$. A \emph{quotient} of
$\mathcal{G}$, the dual notion, is an equivalence class of
epimorphisms of group objects
$\mathcal{G}\twoheadrightarrow\mathcal{Q}$.

Such a subgroup is said to be \emph{normal} if the conjugation
action of $\mathcal{G}$ on itself induces an action on
$\mathcal{H}$. In other words, $\mathcal{H}$ is normal if there
exists a commutative diagram
\[ \xymatrix{
\mathcal{G}\times\mathcal{G}\ar[r]^{c}&\mathcal{G}\\
\mathcal{G}\times\mathcal{H}\ar[r]\ar[u]&\mathcal{H}\ar[u] }\] where
$c$ is the conjugation action on $\mathcal{G}$ (which can be defined
on sections, or more directly as the conjugation action on a group
object in any category). If such an induced action of $\mathcal{G}$
on $\mathcal{H}$ does exist, then it is unique since
$\mathcal{H}\hookrightarrow\mathcal{G}$ is mono.

Let $\mathcal{H}\hookrightarrow\mathcal{G}$ be a subgroup. Consider
the quotient $\mathcal{G}/\mathcal{H}$ in $\mathcal{T}$ of the
equivalence relation
$$\mathcal{H}\times\mathcal{G}\rightrightarrows\mathcal{G},$$
where the arrows are given by projection and multiplication. Then
$\mathcal{G}/\mathcal{H}$ has a group structure compatible with the
group structure on $\mathcal{G}$ (i.e. the map
$\mathcal{G}\rightarrow\mathcal{G}/\mathcal{H}$ is a group morphism)
if and only if $\mathcal{H}$ is normal in $\mathcal{G}$.

In particular, let $f:\mathcal{G}'\rightarrow \mathcal{G}$ be a
morphism of group objects. The kernel of $f$ is defined as
$$Ker(f):=\mathcal{G}'\times_{\mathcal{G}}*$$
where $*\rightarrow\mathcal{G}$ is the unit section. Then $Ker(f)$
is a normal subgroup of $\mathcal{G}'$.

\subsubsection{Normal subgroup generated by a subgroup.}

Let $i:\mathcal{H}\hookrightarrow\mathcal{G}$ be a subgroup.
Consider the category of triangles
\[ \xymatrix{
&{N}\ar[d]\\
\mathcal{H}\ar[r]^i\ar[ur]&\mathcal{G} }\] where the maps are all
monomorphisms of groups and ${N}$ is normal in $\mathcal{G}$. This
category is not empty since it contains $Id_{G}$ as the final
object. The \emph{normal subgroup generated by $\mathcal{H}$} in
$\mathcal{G}$ is the projective limit in $\mathcal{T}$
$$\mathcal{N}(\mathcal{H}):=\underleftarrow{lim}\,{N},$$
More precisely, $\mathcal{N}(\mathcal{H})$ is the projective limit
of the functor that sends a triangle as above to ${N}$. We check
immediately that $\mathcal{N}(\mathcal{H})$ is the smallest normal
subgroup of $\mathcal{G}$ containing $\mathcal{H}$.

\subsubsection{Subgroup generated by a family of subgroups.}

Let $\{\mathcal{H}_j\hookrightarrow\mathcal{G},\,j\in J\}$ be a
family of subgroup of $\mathcal{G}$. We define analogously the
subgroup
$$<\mathcal{H}_j,\,j\in J>\hookrightarrow\mathcal{G}$$
generated by the $\mathcal{H}_j$'s in $\mathcal{G}$. In what
follows, we denote by $$\mathcal{H}_1\vee \mathcal{H}_2$$ the
subgroup generated by two subgroups $\mathcal{H}_1$ and
$\mathcal{H}_2$ in $\mathcal{G}$.

\subsubsection{} We consider now subgroups of representable group objects in
$\mathcal{T}$. Let $G$ is a (locally compact) topological group. A \emph{topological subgroup of $G$} is a subgroup $H\subseteq G$ endowed with the induced topology. A \emph{topological quotient of $G$} is a quotient $G/H$ endowed with the quotient topology, where $H$ is a normal subgroup.
\begin{lem}\label{lemma-pblms}
Let $y(G)$ be a group of $\mathcal{T}$ representable by a
topological group $G$. The following are equivalent.
\begin{enumerate}
\item $G$ is discrete.
\item Any subgroup of $yG$ is representable by a topological
subgroup of $G$.
\item Any quotient of $yG$ is representable by a topological
quotient of $G$.
\end{enumerate}
\end{lem}
\begin{proof}
By (\cite{these} Lemma 10.29) the unique morphism
$$e_{\mathcal{T}}:\mathcal{T}\longrightarrow\underline{Sets}$$
is \emph{hyperconnected}. This means that, for any set $I$, the
(ordered) set $Sub_{\underline{Sets}}(I)$ of subobjects of $I$ in
$\underline{Sets}$ is in $1-1$ correspondence  with the set
$Sub_{\mathcal{T}}(e_{\mathcal{T}}^*I)$ of subobjects of
$e_{\mathcal{T}}^*I$ in $\mathcal{T}$. Note that
$Sub_{\underline{Sets}}(I)$ is just the family of subsets of $I$,
and that $e_{\mathcal{T}}^*I=y(I)$ is the sheaf of $\mathcal{T}$
represented by the discrete topological space $I$. Thus we have
$(1)\Rightarrow(2)$. We have also $(1)\Rightarrow(3)$ for the same reason. Let us write a more direct proof of this fact using $(1)\Rightarrow(2)$. Let $G$ be a discrete group. If
$$f:yG\twoheadrightarrow\mathcal{Q}$$ is a quotient in $\mathcal{T}$,
then the kernel
$$Ker(f):=yG\times_{\mathcal{Q}}*\,\hookrightarrow\,yG$$ is a subobject of $yG$
in $\mathcal{T}$. Therefore $Ker(f)=y(K)$ is a representable by a
usual subgroup $K\subseteq G$, and we have
$$\mathcal{Q}=y(G)/y(K)=y(G/K)$$
since the map $G\rightarrow G/K$ has (obviously) local sections (see
\cite{MatFlach} Lemma 4).

We claim that $(2)\Rightarrow(1)$ and $(3)\Rightarrow(1)$. Let $G$
be a non-discrete topological group. We denote by $G^{\delta}$ the
abstract group $G$ endowed with the discrete topology. The injective
continuous map $G^{\delta}\rightarrow G$ yields a monomorphism in
$\mathcal{T}$ :
$$yG^{\delta}\hookrightarrow yG.$$
This map is not an isomorphism. Indeed, the Yoneda functor is fully faithful and the identity map $G\rightarrow G^{\delta}$ is
not continuous. Hence $yG^{\delta}$ is a proper subgroup of $yG$.
But $yG^{\delta}$ is not representable by a topological subgroup of $G$, since
the induced morphism on global sections
$$e_{\mathcal{T},*}(yG^{\delta}):=Hom_{Top}(*,G^{\delta})=G^{\delta}\longrightarrow e_{\mathcal{T},*}(yG):=Hom_{Top}(*,G)=G^{\delta}$$
is an isomorphism. Similarly the quotient
$$yG/yG^{\delta}$$
is not representable by a quotient of $G$, since the kernel of the
map $yG\rightarrow yG/yG^{\delta}$ is not representable by a
topological subgroup of $G$ (this is $yG^{\delta}$).
\end{proof}

\begin{lem}\label{painfullemma1}
Let $W$ be a locally compact finite dimensional topological group and let
$N_1$ and $N_2$ be two normal compact subgroups of $W$. Let $N_1\vee
N_2$ be the normal topological subgroup of $W$ generated by $N_1$
and $N_2$, endowed with the induced topology. Then $N_1\vee
N_2$ is compact and the canonical map
$$yN_1\vee yN_2\longrightarrow y(N_1\vee N_2)$$
is an isomorphism of subgroups of $yW$. Moreover, one has
$$yW/(yN_1\vee yN_2)\,\simeq y(W/N_1\vee N_2).$$
\end{lem}

\begin{proof}

The subgroup $y(N_1\vee N_2)\hookrightarrow y(W)$ contains both
$yN_1$ and $yN_2$. Hence it contains $(yN_1\vee yN_2)$ as well, i.e.
one has
$$(yN_1\vee yN_2)\hookrightarrow y(N_1\vee N_2)\hookrightarrow yW$$
We show below that the inverse inclusion holds.

Any element of $N_1\vee N_2$ is of the form $xy$ for $x\in N_1$ and
$y\in N_2$. We have a continuous map
$$N_1\times
N_2\rightarrow W\times W\rightarrow W,$$ where the second map is the
multiplication. The image of this map is precisely $N_1\vee N_2$
hence we obtain a surjective continuous map
$$N_1\times N_2\longrightarrow N_1\vee N_2.$$
This shows that $N_1\vee N_2$ is compact. Note that this map is not a morphism of groups in general. This map
induces a bijective continuous map
\begin{equation}\label{map-painfullemma}
N_1\times N_2/(N_1\cap N_2)\longrightarrow N_1\vee N_2
\end{equation}
where the group $(N_1\cap N_2)$ acts on the space $N_1\times N_2$ by
$$\sigma (x,y)=(x\sigma^{-1},\sigma y)$$
for any $\sigma\in(N_1\cap N_2)$ and $(x,y)\in N_1\times N_2$. The
map (\ref{map-painfullemma}) is also closed since $N_1\times N_2/(N_1\cap N_2)$ is compact,
hence we get an homeomorphism
$$N_1\times N_2/(N_1\cap N_2)\simeq N_1\vee N_2.$$
The map $N_1\times N_2\rightarrow N_1\times N_2/(N_1\cap N_2)$ is a
local section cover by Corollary \ref{cor-local-sections}. Indeed, $(N_1\cap N_2)$ is a closed subgroup of
the compact group $N_1\times N_2$, and $N_1\times N_2$ is finite dimensional since $N_1$ and $N_2$ are two compact subgroups of $W$ which is finite dimensional (see \cite{local sections} Corollary 2). Hence the map
$$y(N_1\times N_2)\rightarrow y(N_1\vee N_2)$$
is an epimorphism in $\mathcal{T}$ (again, this is not a morphism of
groups in general).

It follows that $y(N_1\vee N_2)$ is the image of the map
$$y(N_1\times N_2)\rightarrow y(W\times W)\rightarrow yW,$$ where
the second map is the multiplication. In other words, one has the
epi-mono factorization
$$y(N_1\times N_2)\twoheadrightarrow y(N_1\vee N_2)\hookrightarrow yW.$$
But the image of $y(N_1\times N_2)$ in $y(W)$ is contained in
$(yN_1\vee yN_2)$ (check this on sections), hence we have
$$y(N_1\vee N_2)\hookrightarrow(yN_1\vee yN_2)\hookrightarrow yW$$
We obtain $y(N_1\vee N_2)=(yN_1\vee yN_2)$ (recall that the set of
subgroups of $yW$ has is an ordered set). In particular, one has
$$yW/(y(N_1)\vee y(N_2))=yW/y(N_1\vee N_2)=y(W/N_1\vee N_2)$$
where the last equality follows from the fact that $W\rightarrow
W/N_1\vee N_2$ has local sections, since $N_1\vee N_2$ is a compact
subgroup of the locally compact and finite dimensional group $W$ (see Corollary \ref{cor-local-sections}).
\end{proof}

\begin{rem}
The previous result generalizes immediately to the case of a finite
number of compact normal topological subgroups $\{N_j\subset
W,\,1\leq j\leq n\}$.
\end{rem}
Let $L/K$ be a finite Galois extension of number fields (inside the
fixed algebraic closure $\bar{F}/K$) and let
$$1\rightarrow C_L\rightarrow W_{L/K}\rightarrow G_{L/K}\rightarrow 1$$
be the associated relative Weil group. Let $v$ be a place of $K$ and
let $\widetilde{W}^1_{K_v}$ be the image of the composite morphism
$${W}^1_{K_v}\hookrightarrow{W}_{K_v}\hookrightarrow W_K\twoheadrightarrow W_{L/K}$$
endowed with the induced topology. We consider the topological
normal subgroup $N(\widetilde{W}^1_{K_v})$ of $W_{L/K}$ generated by
$\widetilde{W}^1_{K_v}$. We consider also the normal subgroup
$\mathcal{N}(y\widetilde{W}^1_{K_v})$ of $yW_{L/K}$ generated by
$y\widetilde{W}^1_{K_v}$.
\begin{lem}\label{lemma-compactitude}
With the notations above, the group $N(\widetilde{W}^1_{K_v})$ is a
compact subgroup of $W_{L/K}$.
\end{lem}

\begin{proof}
Let $\bar{F}/L/K$ be a finite Galois sub-extension. The image of
${W}^1_{K_v}$ in $W_{K/F}$ is topologically isomorphic to
$W^1_{L_w/K_v}$, i.e. one has
$$\widetilde{W}^1_{K_w}\simeq W^1_{L_w/K_v}.$$ Here $W^1_{L_w/K_v}$ is the maximal compact
subgroup of $W_{L_w/K_v}$, which is in turn given by the group
extension
$$1\rightarrow L_w^{\times}\rightarrow W_{L_w/K_v}\rightarrow G(L_w/K_v)\rightarrow1$$
where $w$ is a place of $L$ lying above $v$. More precisely,
$W^1_{L_w/K_v}$ is given by the following extension
$$1\rightarrow\mathcal{O}_{L_w}^{\times}\rightarrow W^1_{L_w/K_v}\rightarrow I(L_w/K_v)\rightarrow1$$
where $I(L_w/K_v)$ is the usual inertia subgroup of $G(L_w/K_v)$. The map $W^1_{L_w/K_v}\rightarrow W_{L/K}$ sits in the
(injective) morphism of group extensions
\[ \xymatrix{
1\ar[r]&\mathcal{O}_{L_w}^{\times}\ar[r]\ar[d]&W^1_{L_w/K_v}\ar[r]\ar[d]&I(L_w/K_v)\ar[d]\ar[r]&1\\
1\ar[r]&C_L\ar[r]&W_{L/K}\ar[r]^{\varphi}&G(L/K)\ar[r]&1 }\] We thus
have
$$W^1_{L_w/K_v}\cap C_L=\mathcal{O}_{L_w}^{\times}$$
where the intersection makes sense inside $W_{L/K}$. The
conjugation action of $W_{L/K}$ on $C_L$ corresponds to the Galois
action. Hence for any $\eta\in W_{L/K}$, we set $\sigma=\varphi(\eta)$
and we have
$$\eta\,\mathcal{O}_{L_w}^{\times}\eta^{-1}=\mathcal{O}_{L_{\sigma(w)}}^{\times}\subset C_L.$$
We denote by $$N_v:=N(\widetilde{W}^1_{K_v})$$ the normal subgroup generated by $W^1_{L_w/K_v}$
in $W_{L/K}$. We obtain
$$\prod_{w\mid v}\mathcal{O}_{L_w}^{\times}\subset N_v\cap C_L$$
and a quotient map (hence surjective, continuous and open)
$$C_L^v:=C_L/\prod_{w\mid v}\mathcal{O}_{L_w}^{\times}\longrightarrow C_L/(N_v\cap C_L).$$
On the other hand, we have
$$\varphi(N_v)=G(L/K^{un})\subset G(L/K)$$ where $K^{un}/K$ is the maximal
subextension of $L/K$ unramified above $v$, since $\varphi(N_v)$ is the
normal subgroup of $G(L/K)$ generated by $I(L_w/K_v)$. We have the
following commutative diagram with exact rows :
\[ \xymatrix{
1\ar[r]&N_v\cap C_L\ar[r]\ar[d]&N_v\ar[r]\ar[d]&G(L/K^{un})\ar[d]\ar[r]&1\\
1\ar[r]&C_L\ar[r]\ar[d]&W_{L/K}\ar[r]^{\varphi}\ar[d]&G(L/K)\ar[r]\ar[d]&1\\
1\ar[r]&C_L/C_L\cap N_v\ar[r]\ar[d]&W_{L/K}/N_v\ar[r]\ar[d]&G(K^{un}/K)\ar[d]\ar[r]&1\\
1\ar[r]&C^v_{K^{un}}\ar[r]&W^v_{K^{un}/K}\ar[r]&G(K^{un}/K)\ar[r]&1}\]
In the diagram above, $W^v_{K^{un}/K}$ is the extension of
$G(K^{un}/K)$ by $C_{K^{un}}^v:=C_{K^{un}}/\prod_{w\mid
v}\mathcal{O}_{K^{un}_w}^{\times}$ corresponding to the fundamental
class (note that $\prod_{w\mid v}\mathcal{O}_{K^{un}_w}^{\times}$ is
a cohomologically trivial $G(K^{un}/K)$-module since $K^{un}/K$ is unramified at $v$). It can be seen from the diagram above that $C_L/C_L\cap N_v$ is a $G(L/K^{un})$-invariant quotient of $C_L$. To reach the same conclusion, one can also observe that the group $C_L\cap N_v$ contains the group generated by the
family
$$\{\alpha\sigma(\alpha)^{-1}=\alpha\eta\alpha^{-1}\eta^{-1},\,\alpha\in C_L,\, \sigma \in G(L/K^{un}),\, \eta \in N_v,\, \sigma:=\varphi(\eta)\}$$
since $\alpha\sigma(\alpha)^{-1}\in C_L$ and $\alpha\eta\alpha^{-1}\eta^{-1}\in N_v$ for any $\alpha\in C_L$ and $\eta\in N_v$.
Let $H_0(G(L/K^{un}),C^v_L)$ be the maximal $G(L/K^{un})$-invariant
quotient of $C^v_L$, \emph{endowed with the quotient topology}. We
obtain a continuous surjective open map
$$H_0(G(L/K^{un}),C^v_L)\longrightarrow C_L/C_L\cap N_v.$$
Considering the norm map, we obtain a commutative triangle
\[ \xymatrix{
H_0(G(L/K^{un}),C^v_L)\ar[r]\ar[rd]^{\textmd{N}}&C_L/C_L\cap N_v\ar[d]\\
&H^0(G(L/K^{un}),C^v_L)}\]
More precisely, the norm map $\textmd{N}$ can be decomposed as follows :
\begin{equation}\label{Norm-map}
\textmd{N}:H_0(G(L/K^{un}),C^v_L)\twoheadrightarrow C_L/C_L\cap N_v\twoheadrightarrow C^v_{K^{un}}\hookrightarrow H^0(G(L/K^{un}),C^v_L).
\end{equation}
The kernel and cokernel of the norm map $\textmd{N}$ are given by the following exact sequence
{\small{
$$0\rightarrow\widehat{H}^{-1}(G(L/K^{un}),C^v_L)\rightarrow
H_0(G(L/K^{un}),C^v_L) \rightarrow
H^0(G(L/K^{un}),C^v_L)\rightarrow\widehat{H}^{0}(G(L/K^{un}),C^v_L)\rightarrow1$$}}
Using class field theory, we prove easily that
$\widehat{H}^{-1}(G(L/K^{un}),C^v_L)$ and
$\widehat{H}^{0}(G(L/K^{un}),C^v_L)$ are both finite.
In particular, the continuous open and surjective map
$$C_L/C_L\cap N_v\longrightarrow C^v_{K^{un}}$$
has finite kernel. It is a finite
étale Galois cover (in the topological sense), hence a local homeomorphism.
Hence $C_L/C_L\cap N_v$ is Hausdorff, i.e. $C_L\cap N_v$ is closed
in $C_L$. But $N_v$ is contained in $W^1_{L/K}$, hence $C_L\cap N_v$ is a closed subgroup of
$C_L^1$, where $C_L^1$ denotes the maximal compact subgroup of $C_L$. Therefore $C_L\cap
N_v$ is compact, and $N_v$ is an extension of the finite group
$G(L/K^{un})$ by $C_L\cap N_v$. Hence $N_v$ is compact as well.
\end{proof}

\begin{lem}\label{lem-la-clef}
We keep the notations of Lemma \ref{lemma-compactitude}. One has the
equality
$$\mathcal{N}(y\widetilde{W}^1_{K_v})=yN(\widetilde{W}^1_{K_v}).$$
of subgroups of $y(W_{L/K})$ in $\mathcal{T}$.
\end{lem}

\begin{proof}
Following the notations of the previous proof, we set
$$N_v:=N(\widetilde{W}^1_{K_v})\mbox{ and }\mathcal{N}_v:=\mathcal{N}(y\widetilde{W}^1_{K_v}).$$
We have the following morphism of exact sequences of group objects
in $\mathcal{T}$, where the vertical maps are all monomorphisms.
\[ \xymatrix{
1\ar[r]&\mathcal{N}_v\times_{yW_{L/K}} yC_L\ar[r]\ar[d]&\mathcal{N}_v\ar[r]\ar[d]&yG(L/K^{un})\ar[d]\ar[r]&1\\
1\ar[r]&yC_L\ar[r]&yW_{L/K}\ar[r]^{\varphi}&yG(L/K)\ar[r]&1
}\] The subgroup $\mathcal{N}_v\times_{yW_{L/K}} yC_L$ contains
\begin{equation}\label{first-subgroup}
y{W}^1_{L_w/K_v}\times_{yW_{L/K}} yC_L=y({W}^1_{L_w/K_v}\cap
C_L)=y(\mathcal{O}^{\times}_{L_w})
\end{equation}
since the Yoneda functor commutes with fiber products. Hence
$\mathcal{N}_v\times_{yW_{L/K}} yC_L$ contains the conjugates in $yW_{L/K}$ of the
subgroup (\ref{first-subgroup}) :
$$\eta (y\mathcal{O}^{\times}_{L_w})\eta^{-1}=y(\eta\mathcal{O}^{\times}_{L_w}\eta^{-1})=y\mathcal{O}^{\times}_{L_\sigma(w)}$$
for any $\eta\in W_{L/K}$ with $\sigma=\phi(\eta)$. Thus
$\mathcal{N}_v\times_{yW_{L/K}} yC_L$ contains the subgroup of
$yW_{L/K}$ generated by those subgroups :
$$<y\mathcal{O}^{\times}_{L_\sigma(w)},\,\sigma\in G(L/K)>=y(<\mathcal{O}^{\times}_{L_\sigma(w)},\,\sigma\in G(L/K)>)
=y(\prod_{w\mid v}\mathcal{O}_{L_w}^{\times}),$$ where the first
identity follows from Lemma \ref{painfullemma1}. Let $\sigma\in
G(L/K^{un})$, and consider the topological subgroup of $C_L$ given by
$$(Id-\sigma)(C_L):=\{\alpha\sigma(\alpha)^{-1},\,\alpha\in C_L\}.$$
Then $(Id-\sigma)(C_L)$ is compact, since it is the image of the
continuous morphism
$$
\appl{C_L^1}{C_L}{\alpha}{\alpha\sigma(\alpha)^{-1}}
$$
where $C_L^1$ is the maximal compact subgroup of $C_L$. Using this
fact and an argument similar to the proof of Lemma
\ref{painfullemma1}, we see that we have the inclusion
$$y((Id-\sigma)(C_L))\hookrightarrow\mathcal{N}_v\times_{yW_{L/K}} yC_L$$
of subgroups of $yW_{L/K}$. Therefore, $\mathcal{N}_v\times_{yW_{L/K}} yC_L$ contains the subgroup
of $yW_{L/K}$ generated by all the subgroups considered above :
$$<y(\prod_{w\mid v}\mathcal{O}_{L_w}^{\times})\mbox{ ; }y(Id-\sigma)(C_L) \mbox{ $\forall$ }\sigma\in G(L/K^{un})>
\,\,\hookrightarrow\,\,\mathcal{N}_v\times_{yW_{L/K}} yC_L.$$ We
have above a finite number of compact subgroups of $W_{L/K}$. By
Lemma \ref{painfullemma1}, we obtain
$$y(\Theta):=y(<\prod_{w\mid v}\mathcal{O}_{L_w}^{\times}\mbox{ ; }(Id-\sigma)(C_L) \mbox{ $\forall$ }\sigma\in
G(L/K^{un})>) \,\,\hookrightarrow\,\,\mathcal{N}_v\times_{yW_{L/K}}
yC_L.$$ where $\Theta:=<\prod_{w\mid
v}\mathcal{O}_{L_w}^{\times}\mbox{ ; }(Id-\sigma)(C_L) \mbox{
$\forall$ }\sigma\in G(L/K^{un})>$ is a topological subgroup of $C_L$. Note that we have
$$C_L/\Theta=H_0(G(L/K^{un}),C_L^v).$$ The proof of the previous lemma
shows that $\Theta$ is a subgroup of finite index in $N_v\cap C_L$, since the norm map $\textrm{N}$ has finite kernel (see (\ref{Norm-map})). More precisely, we have the following exact sequence of topological groups
$$1\rightarrow H'\rightarrow C_L/\Theta=H_0(G(L/K^{un}),C_L^v)\rightarrow C_L/N_v\cap C_L\rightarrow1$$
where $H'$ is a finite subgroup of $\widehat{H}^{-1}(G(L/K^{un}),C_L^v)$. In particular, $\Theta$ is open in $N_v\cap
C_L$. We have monomorphisms
$$y\Theta\,\hookrightarrow\, \mathcal{N}_v\times_{yW_{L/K}} yC_L\,\hookrightarrow\, y(N_v\cap C_L).$$
This implies that $\mathcal{N}_v\times_{yW_{L/K}} yC_L$ is
representable by a topological group, as it follows from Lemma \ref{lemma-pblms}. 

Now the exact sequence
$$1\rightarrow\mathcal{N}_v\times_{yW_{L/K}} yC_L\rightarrow\mathcal{N}_v\rightarrow yG(L/K^{un})\rightarrow 1$$
and the fact that the Yoneda functor $y:Top\rightarrow\mathcal{T}$
commutes with (disjoint) sums (of topological spaces) show that
$\mathcal{N}_v$ is itself representable. Hence $\mathcal{N}_v$ is
representable by a topological group $N'_v$, and we have continuous
injective morphisms of topological groups
$$\Theta\hookrightarrow N'_v\hookrightarrow N_v$$
since the maps $y\Theta\hookrightarrow$ and  $yN'_v\hookrightarrow yN_v$
are both monomorphisms in $\mathcal{T}$. But $\Theta$ is open in
$N_v$, hence $N'_v$ is a topological subgroup of $N_v$, i.e.
$N'_v\subseteq N_v$ is endowed with the induced topology.

Moreover $\mathcal{N}_v=yN'_v$ is normal in $yW_{L/K}$, hence so is
$N'_v$ in $W_{L/K}$ (since Yoneda is fully faithful). Finally $N'_v$
must contain $\widetilde{W}^1_{K_v}$ and we get
$N'_v=N(\widetilde{W}^1_{K_v})=N_v$ hence
$$\mathcal{N}_v=yN'_v=yN(\widetilde{W}^1_{K_v}).$$
\end{proof}

\subsection{A generating family for the Weil-\'etale topos}\label{subsect-site}

Let $\bar{U}$ be a connected \'etale $\bar{X}$-scheme endowed with Data \ref{choices-U}. We denote by $K$ the function field of $\bar{U}$. In this section, we define a simple topologically generating family for the site
$(T_{\bar{U}},\mathcal{J}_{ls})$ (hence a generating family for the topos $\bar{U}_W$). This has
already been used to show that $\bar{U}_W$ is connected and locally
connected over $\mathcal{T}$, and this will be necessary to compute
the fundamental group of $\bar{U}_W$.

Let us fix a finite Galois sub-extension $\bar{F}/L/K$, an open
subset $V$ of $\bar{U}$, a point $u$ of $\bar{U}$ and a locally compact topological space $T$. In this section, we denote by $N$ the \emph{closed normal
subgroup} of $W_{L/K}$ generated by the subgroups
$\widetilde{W}_{K_v}^1\subseteq W_{L/K}$ for any closed point $v\in V$. Let
$(N,\widetilde{W}_{K_u}^1)$ be the subgroup of $W_{L/K}$ generated
by $N$ and $\widetilde{W}_{K_u}^1$. This subgroup is compact hence
closed. We define an object of $T_{\bar{U}}$
$$\mathcal{G}_{L,V,u,T}:=(G_{0}\times T,G_{v}\times T,g_{v})$$
as follows. If $u$ is not in $V$, we consider
$$G_{0}=W_{L/K}/(N,\widetilde{W}^1_{K_u})$$ as a $W_K$-space and
$$G_{v}=W_{L/K}/(N,\widetilde{W}^1_{K_u})$$ as a $W_{k(v)}$-space for
any point $v$ of $V\subseteq\bar{U}$. Then we set $G_{u}=W_{k(u)}$
on which $W_{k(u)}$ acts by multiplication, and $G_{v}=\emptyset$
anywhere else. The group $W_K$ (respectively $W_{k(v)}$) acts on the
first factor of $G_0\times T$ (respectively of $G_v\times T$). The
map
$$g_{v}:G_{v}\times T\longrightarrow G_{0}\times T$$ is the identity for
any point $v$ of $V\subseteq\bar{U}$, and the continuous map
$$g_{u}:W_{k(u)}\times T=W_{K_u}/W^1_{K_u}\times T\longrightarrow
W_{L/K}/(N,\widetilde{W}^1_{K_u})\times T$$ is induced by
the Weil map $W_{K_u}\rightarrow W_K$.

If $u\in V$ we define $\mathcal{G}_{L,V,u,T}$ as above except that
we set
$$G_{u}=W_{L/K}/(N,\widetilde{W}^1_{K_u})=W_{L/K}/N.$$

\begin{Notation}
We denote by $\mathcal{G}_{L,V,u,T}$ the object of $T_{\bar{U}}$ defined above. If $T=*$ is the one point space, then we set
$\mathcal{G}_{L,V,u}:=\mathcal{G}_{L,V,u,*}$.
\end{Notation}

For any space $T$ of $Top$, one has a product decomposition in $T_{\bar{U}}$ :
$$\mathcal{G}_{L,V,u,T}=\mathcal{G}_{L,V,u}\times t^*T$$
where $t^*T=(T,T,Id_T)$ is the constant object of $T_{\bar{U}}$ associated to the space $T$.

\begin{defn}
Let $\mathbb{G}_{\bar{U}}$ be the full subcategory of $T_{\bar{U}}$
consisting in objects of the form $\mathcal{G}_{L,V,u,T}$, where
$\bar{F}/L/K$ is a finite Galois sub-extension, $V$ is an open
subset of $\bar{U}$, $u$ is a point of $\bar{U}$ and $T$ is a locally compact topological space.
The category $\mathbb{G}_{\bar{U}}$ is
endowed with the local section topology $\mathcal{J}_{ls}$.
\end{defn}

\begin{thm}\label{thm-site}
The canonical morphism
$$\bar{U}_W\longrightarrow\widetilde{(\mathbb{G}_{\bar{U}},\mathcal{J}_{ls})}$$
is an equivalence.
\end{thm}
\begin{proof}
We have a composition of fully faithful functors
$$\mathbb{G}_{\bar{U}}\rightarrow T_{\bar{U}}\rightarrow\bar{U}_W$$
where the second functor is the Yoneda embedding. The local section
topology on $\mathbb{G}_{\bar{U}}$ is the topology induced
by the local section topology on $T_{\bar{U}}$ via the inclusion
$\mathbb{G}_{\bar{U}}\rightarrow T_{\bar{U}}$. Hence
$\mathcal{J}_{ls}$ on $\mathbb{G}_{\bar{U}}$ is the topology induced
by the canonical topology of $\bar{U}_W$ via the composite functor defined
above. But the Yoneda embedding takes a topologically generating
family of a site to a generating family of the corresponding topos.
Hence it remains to show that $\mathbb{G}_{\bar{U}}$ is a
topologically generating family for the site $(T_{\bar{U}},\mathcal{J}_{ls})$. In other words, we
need to prove that any object of $T_{\bar{U}}$ admits a local
section cover by objects of $\mathbb{G}_{\bar{X}}$.

Let $(Z_0,Z_v,f_v)$ be an object of  $T_{\bar{U}}$. The action of $W_K$ on $Z_0$
factors through $W_{L/K}$, for a finite Galois extension $L/K$.
Since the group $W_{L/K}$ is locally compact, its action on the
space $Z_{0}$ yields a continuous morphism of topological groups
$$\rho:W_{L/K}\longrightarrow \underline{Aut}_{Top}(Z_{0})$$
where the group $\underline{Aut}_{Top}(Z_{0})$, of homemorphisms of
$Z_{0}$, is endowed with the compact-open topology. The space
$Z_{0}$ is Hausdorff hence so is the topological group
$\underline{Aut}_{Top}(Z_{0})$. It follows that the kernel of $\rho$
is a closed normal subgroup of $W_{L/K}$ : $$Ker(\rho)\subseteq
W_{K/F}.$$
Let $V$ be the open set of points of $\bar{U}$ such that $f_v$ is an
homeomorphism. Take the generic point $u=u_0$ of $\bar{U}$ and $T=Z_0$ as a topological space. Let
$N$ be the closed normal subgroup of $W_{L/K}$ generated by the
subgroups $\widetilde{W}_{K_v}^1\subseteq W_{L/K}$ for any closed point
$v\in V$. The action of $W_K$ on $Z_0$ factors through $W_{L/K}/N$,
since the kernel of the continuous morphism $\rho$ is closed in
$W_{L/K}$. Hence $\rho$ induces a continuous morphism
$$W_{L/K}/N\longrightarrow \underline{Aut}_{Top}(Z_{0}).$$
Such an action is given by a continuous map
$$G_0\times T:=W_{L/K}/N\times Z_0\longrightarrow Z_0$$
which is $W_K$-equivariant. This map has an obvious global
continuous section. We obtain a morphism in $T_{\bar{U}}$
$$\mathcal{G}_{L,V,u_0,Z_0}\longrightarrow (Z_0,Z_v,f_v)$$
which is a global section cover over any point $v\in V$.

Let $u\in\bar{U}-V$. Here we consider
$$\mathcal{G}_{L,V,u,Z_u}=(G_{0}\times Z_u,G_{v}\times Z_u,g_{v}),$$
with $G_v=G_0=W_{L/K}/(N,\widetilde{W}^1_{K_u})$ for any $v\in V$. The second projection gives a $W_{k(u)}$-equivariant continuous
map
$$\phi_u:G_u\times Z_u:=W_{k(u)}\times Z_u\longrightarrow Z_u$$
which has a global continuous section. Then there exists a
unique morphism in $T_{\bar{U}}$
$$\phi:\mathcal{G}_{L,V,u,Z_u}\longrightarrow (Z_0,Z_v,f_v)$$
inducing $\phi_u$ at the point $u\in\bar{U}$. Indeed, the given $W_{k(u)}$-equivariant continuous map $f_u:Z_u\rightarrow
Z_0$ provide us with a $W_{K}$-equivariant map
$$\phi_0:G_{0}\times Z_u:=W_{L/K}/(N,\widetilde{W}^1_{K_u})\times Z_u\longrightarrow
Z_0$$
For any point $v$ of $V$, the same map $\phi_v:=\phi_0$ is also $W_{k(v)}$-equivariant and continuous :
$$\phi_v:G_{v}\times Z_u:=W_{L/K}/(N,\widetilde{W}^1_{K_u})\times Z_u\longrightarrow
Z_0\simeq Z_v.$$

We have obtained a local section cover of $\mathcal{Z}:=(Z_0,Z_v,f_v)$ by objects of
$\mathbb{G}_{\bar{U}}$ :
$$\{\mathcal{G}_{L,V,u_0,Z_0}\rightarrow\mathcal{Z},\, \mathcal{G}_{L,V,u,Z_u}\rightarrow\mathcal{Z}\mbox{ for $u\in\bar{U}-V$}\}$$
Hence (the essential image of)
$\mathbb{G}_{\bar{U}}$ is a generating full subcategory of
$\bar{U}_W$ endowed with the topology induced by the canonical
topology. The result then follows from (\cite{SGA4} IV Corollary
1.2.1).
\end{proof}

\begin{cor}\label{cor-sympa}
Consider the full subcategory $\mathbb{C}_{\bar{U}}$ of
$T_{\bar{U}}$ consisting in objects $(Z_0,Z_v,f_v)$ such that the
canonical morphism in $\mathcal{T}$
$$yZ_0/yW_K\longrightarrow y(Z_0/W_K)$$ is an isomorphism with
$Z_0/W_K$ locally compact. Then $\mathbb{C}_{\bar{U}}$ is a topologically
generating family of $T_{\bar{U}}$.
\end{cor}
\begin{proof}It is enough to show that $$\mathcal{G}_{L,V,u,T}=(G_{0}\times T,G_{v}\times T,g_{v})$$
satisfies those properties. The map $$W_{L/K}\rightarrow
W_{L/K}/(N,\widetilde{W}^1_{K_u})$$ admits local sections by Corollary \ref{cor-local-sections}, since
$W_{L/K}$ is locally compact and finite dimensional, and
$(N,\widetilde{W}^1_{K_u})$ is compact hence closed. The map $W_K\rightarrow
W_{L/K}$ admits local sections by Corollary \ref{cor-exple-of-loc-sections}. We obtain
an epimorphism in $\mathcal{T}$
$$yW_{K}\,\twoheadrightarrow\,y(W_{L/K}/(N,\widetilde{W}^1_{K_u})).$$
Hence the quotient of the action of $yW_{K}$ on
$y(W_{L/K}/(N,\widetilde{W}^1_{K_u}))$ is the final object of
$\mathcal{T}$. Thus the quotient of
$$y(G_0\times T):=y(W_{L/K}/(N,\widetilde{W}^1_{K_u})\times
T)=y(W_{L/K}/(N,\widetilde{W}^1_{K_u}))\times yT$$ under the action
of $yW_{K}$ is $yT$, since inductive limits (in particular quotients
of group actions) are universal in $\mathcal{T}$.

On the other hand, the quotient of the topological space
$$G_0\times T:=W_{L/K}/(N,\widetilde{W}^1_{K_u})\times T$$ by the
action of the topological group $W_K$ is the locally compact space $T$.
\end{proof}
\begin{rem}
The space of connected components of $\mathcal{G}_{L,V,u,T}$ is
$$t_!\mathcal{G}_{L,V,u,T}=T.$$
\end{rem}

\subsection{The category $SLC_{\mathcal{T}}(\bar{U}_W)$ of sums of locally constant sheaves}\label{subsect-loc-cstant}
Let $\bar{U}$ be a connected \'etale $\bar{X}$-scheme endowed with Data \ref{choices-U}. In this section, we denote by
$t:\bar{U}_W\rightarrow\mathcal{T}$ the canonical map. This morphism $t$ is connected and locally connected (see Theorem \ref{thm-big-WEfundgrp} (i)).

\subsubsection{}Recall that an object $\mathcal{L}$ of $\bar{U}_W$ is said to be \emph{locally
constant over $\mathcal{T}$} if there exists a covering morphism
$\mathcal{F}\rightarrow 1$ of the final object of $\bar{U}_W$,
an object ${S}$ of $\mathcal{T}$ and an isomorphism  over $\mathcal{F}$
$$\mathcal{L}\times \mathcal{F}\simeq
t^*S\times \mathcal{F}.$$
\begin{defn}
An object $\mathcal{L}$ of $\bar{U}_W$ is said to be \emph{locally
component-wise constant over $\mathcal{T}$} if there exists a epimorphism
$\mathcal{F}\rightarrow 1$ where $1$ denotes the final object of $\bar{U}_W$,
an object $S\rightarrow t_!\mathcal{F}$ of $\mathcal{T}/t_!\mathcal{F}$ and an isomorphism  over $\mathcal{F}$ $$\mathcal{L}\times\mathcal{F}\simeq
t^*S\times_{t^*t_!\mathcal{F}}\mathcal{F}.$$
\end{defn}

\begin{prop}
An object $\mathcal{L}$ of $\bar{U}_W$ is locally
component-wise constant if and only if $\mathcal{L}$ is locally constant.
\end{prop}
\begin{proof}
Any locally constant object is locally component-wise
constant. Indeed, if $\mathcal{L}$ is locally constant then one has
$$\mathcal{L}\times \mathcal{F}\simeq t^*S\times \mathcal{F}=t^*S\times t^*t_!\mathcal{F}\times_{t^*t_!\mathcal{F}}\mathcal{F}
=t^*(S\times t_!\mathcal{F})\times_{t^*t_!\mathcal{F}}\mathcal{F}.$$

The converse is also true. Let $\mathcal{L}$ be a locally
component-wise constant object. There exist $\mathcal{F}$ covering the final object, $S\rightarrow t_!\mathcal{F}$ and an isomorphism
over $\mathcal{F}$ $$\mathcal{L}\times\mathcal{F}\simeq
t^*S\times_{t^*t_!\mathcal{F}}\mathcal{F}.$$
By Theorem \ref{thm-site}, there exists an epimorphic family $\{\mathcal{F}_i\rightarrow\mathcal{F},i\in I\}$
where $\mathcal{F}_i$ is represented by an object $\mathcal{G}_{L_i,V_i,u_i,T_i}$ of $\mathbb{G}_{\bar{U}}$. Choosing a point of $T_i$ for any element $i$ of the set $I$, we obtain
a map
$$\mathcal{G}:=\coprod_{i\in I}y\mathcal{G}_{L_i,V_i,u_i,*}\rightarrow \coprod_{i\in I}y\mathcal{G}_{L_i,V_i,u_i,T_i}\rightarrow\mathcal{F}\rightarrow 1$$ which is a cover of the final object of $\bar{U}_W$.
Then we have
$$\mathcal{L}\times\mathcal{G}=\mathcal{L}\times\mathcal{F}\times_{\mathcal{F}}\mathcal{G}
\simeq t^*S\times_{t^*t_!\mathcal{F}}\mathcal{F}\times_{\mathcal{F}}\mathcal{G}
=t^*S\times_{t^*t_!\mathcal{F}}\mathcal{G}
=t^*(S\times_{t_!\mathcal{F}}t_!\mathcal{G})\times_{t^*t_!\mathcal{G}}\mathcal{G}.$$
Hence one can assume that $\mathcal{F}=\mathcal{G}$. Note that $t_!\mathcal{G}$ is the object of $\mathcal{T}$ represented by the discrete set $I$, so that $S\rightarrow t_!\mathcal{G}=I$ can be seen as a family of objects ${S}_i$ of $\mathcal{T}$, indexed by the set $I$. We set $\mathcal{G}_i:=\mathcal{G}_{L_i,V_i,u_i,*}$ and we have
$\mathcal{L}\times\mathcal{G}_i\simeq{S}_i\times\mathcal{G}_i$. For any $i,j\in I$ we consider an object $\mathcal{K}=\mathcal{G}_{L,V,u,*}$
of $\mathbb{G}_{\bar{U}}$ endowed with a map $\mathcal{K}\rightarrow\mathcal{G}_i\times\mathcal{G}_j$. Then we have an isomorphism in the slice topos $\bar{U}_W/\mathcal{K}$
\begin{equation}\label{iso-for-componentwise}
{S}_i\times\mathcal{K}\simeq\mathcal{L}\times\mathcal{K}\simeq{S}_j\times\mathcal{K}.
\end{equation}
But $\mathcal{K}$ is connected over $\mathcal{T}$ (i.e. $t_!\mathcal{K}$ is the final object of $\mathcal{T}$) and it follows that
$\bar{U}_W/\mathcal{K}\rightarrow\mathcal{T}$ is connected, so that there exists a (unique) isomorphism $S_i\simeq S_j$ in $\mathcal{T}$ inducing (\ref{iso-for-componentwise}). We obtain an isomorphism $S\simeq \coprod_{I}S_i\simeq S_{i_0}\times I$ over $I$ and one has
$$\mathcal{L}\times\mathcal{G}=\mathcal{S}\times_{t_!\mathcal{G}}\mathcal{G}=\mathcal{S}\times_{I}\mathcal{G}
\simeq S_{i_0}\times\mathcal{G}$$
where $i_0$ is some fixed element of $I$. Hence $\mathcal{L}$ is locally constant.
\end{proof}
The category of "sums" of locally constant objects can be defined as follows (see \cite{Bunge-Moerdijk} section 2, and \cite{Bunge} for more details). For any $\mathcal{F}$ covering the final object of $\bar{U}_W$, one defines the push-out topos
\[ \xymatrix{
\bar{U}_W/\mathcal{F}\ar[r]\ar[d]&\bar{U}_W\ar[d]^{\sigma_{\mathcal{F}}}\\
\mathcal{T}/t!\mathcal{F}\ar[r]&\textsf{Spl}_{\mathcal{F}}(\bar{U}_W)
}\]
By definition of the push-out topos, an object of $\textsf{Spl}_{\mathcal{F}}(\bar{U}_W)$ is a triple $(\mathcal{L},S,\chi)$ where $\mathcal{L}$ is an object of $\bar{U}_W$, $S$ an object of $\mathcal{T}/t!\mathcal{F}$ and $\chi$ is an isomorphism in $\bar{U}_W/\mathcal{F}$ $$\mathcal{L}\times\mathcal{F}\simeq
t^*S\times_{t^*t_!\mathcal{F}}\mathcal{F}.$$
The morphisms in the category $\textsf{Spl}_{\mathcal{F}}(\bar{U}_W)$ are the obvious ones.

The inverse image functor
$$
\fonc{\sigma^*_{\mathcal{F}}}{\textsf{Spl}_{\mathcal{F}}(\bar{U}_W)}{\bar{U}_W}{(\mathcal{L},S,\chi)}{\mathcal{L}}
$$
is fully faithful, and its essential image is precisely the full subcategory of $\bar{U}_W$ consisting in locally component-wise constant objects split by $\mathcal{F}$.

Given two epimorphisms $\mathcal{F}\rightarrow e$ and $\mathcal{F}'\rightarrow e$ and any map $\mathcal{F}'\rightarrow\mathcal{F}$,
we have a canonical morphism $\textsf{Spl}_{\mathcal{F}'}(\bar{U}_W)\rightarrow \textsf{Spl}_{\mathcal{F}}(\bar{U}_W)$
such that the triangle
\[ \xymatrix{
\bar{U}_W\ar[r]^{\sigma_{\mathcal{F}'}\,\,\,\,\,\,\,\,\,}\ar[dr]_{\sigma_{\mathcal{F}}}&\textsf{Spl}_{\mathcal{F}'}(\bar{U}_W)
\ar[d]^{\rho_{\mathcal{F'},\mathcal{F}}}\\
&\textsf{Spl}_{\mathcal{F}}(\bar{U}_W)
}\]
is commutative. Hence two different maps $f_1:\mathcal{F}'\rightarrow\mathcal{F}$ and $f_2:\mathcal{F}'\rightarrow\mathcal{F}$ yield two morphisms $\rho^1_{\mathcal{F'},\mathcal{F}}$ and $\rho^2_{\mathcal{F'},\mathcal{F}}$ that are isomorphic.

\begin{defn}
The topos $SLC_{\mathcal{T}}(\bar{U}_W)$ is defined as the projective limit topos
$$SLC_{\mathcal{T}}(\bar{U}_W):=\underleftarrow{lim}\,\textsf{Spl}_{\mathcal{F}}(\bar{U}_W)$$
where $\mathcal{F}$ runs over a small cofinal system of coverings of the final object of $\bar{U}_W$.
\end{defn}
The canonical morphism
\begin{equation}\label{mapsigma}
\sigma:\bar{U}_W\longrightarrow SLC_{\mathcal{T}}(\bar{U}_W),
\end{equation}
induced by the compatible maps $\sigma_{\mathcal{F}}$, is connected and locally connected (see \cite{Bunge-Moerdijk} Theorem 2.2) so that $SLC_{\mathcal{T}}(\bar{U}_W)$ can be seen as a full subcategory of $\bar{U}_W$, which we call the category of sums of locally constant objects.

\subsubsection{}The purpose of the fundamental group is to classify the category of sums of locally constant objects.
The Weil-étale topos $\bar{U}_W$ is connected and locally connected over $\mathcal{T}$ (see Theorem \ref{thm-big-WEfundgrp} (i)).
Consider a $\mathcal{T}$-point $p$ of $\bar{U}_W$ (see Theorem \ref{thm-big-WEfundgrp} (ii)), i.e. a section of the structure map
$$t:\bar{U}_W\longrightarrow\mathcal{T}.$$
Composing $p$ and the morphism (\ref{mapsigma}), we obtain a point
$$\widetilde{p}:\mathcal{T}\longrightarrow\bar{U}_W\longrightarrow
SLC_{\mathcal{T}}(\bar{U}_W)$$ of the topos $SLC_{\mathcal{T}}(\bar{U}_W)$ over
$\mathcal{T}$. The theory of the fundamental group in the context of
topos theory shows the following. We refer to \cite{Moerdijk-Prodiscrete}
and \cite{Bunge-Moerdijk} Section 1, or \cite{Bunge-Moerdijk} Section 2 (and \cite{Bunge} for more details) for a different approach. There exists a "pro-discrete
localic group" $\pi_1(\bar{U}_W,p)$ in $\mathcal{T}$ well defined up to a
canonical isomorphism and an equivalence
$$B_{\pi_1(\bar{U}_W,p)}\simeq SLC_{\mathcal{T}}(\bar{U}_W),$$
where $B_{\pi_1(\bar{U}_W,p)}$ is the classifying topos of $\pi_1(\bar{U}_W,p)$. Moreover, the
equivalence above identifies the inverse image of the point
$\widetilde{p}:\mathcal{T}\rightarrow SLC_{\mathcal{T}}(\bar{U}_W)$ with the
forgetful functor $B_{\pi_1(\bar{U}_W,p)}\rightarrow\mathcal{T}$. In our situation, the "pro-discrete
localic group" $\pi_1(\bar{U}_W,p)$ is in fact (the "limit" of) a strict pro-group in $\mathcal{T}$, as it follows from Theorem \ref{thm-big-WEfundgrp}. More precisely, $\pi_1(\bar{U}_W,p)$ is pro-represented by a strict locally compact topological pro-group in the sense of Definition \ref{def-pro-top-grp}, and $B_{\pi_1(\bar{U}_W,p)}$ is the classifying topos of $\pi_1(\bar{U}_W,p)$ in the sense of Definition \ref{def-class-strict-top-progrp}.

\subsection{Computation of the fundamental group}
Recall that one has a morphism
$$j:B_{W_F}\longrightarrow\bar{X}_W.$$
\begin{lem}\label{loc-cnstant-generic}
If $\mathcal{L}$ is a locally constant object of $\bar{X}_W$ over
$\mathcal{T}$, then the adjunction map
$$\mathcal{L}\longrightarrow j_*j^*\mathcal{L}$$
is an isomorphism.
\end{lem}
\begin{proof}
Let $\mathcal{L}$ be a locally constant object of $\bar{X}_W$ over
$\mathcal{T}$. There exist an object $\mathcal{S}$ of $\mathcal{T}$,
an epimorphism $\mathcal{F}\rightarrow e$ where
$e$ is the final object of $\bar{X}_W$, and an isomorphism
$\mathcal{L}\times\mathcal{F}\simeq
t^*\mathcal{S}\times\mathcal{F}$ over $\mathcal{F}$. Consider the morphism defined by base change of the adjunction map:
\begin{equation}\label{adjunction-local}
\mathcal{L}\times\mathcal{F}\longrightarrow
j_*j^*\mathcal{L}\times\mathcal{F}.
\end{equation}
For any object $\mathcal{U}\rightarrow\mathcal{F}$ of
$\bar{X}_W/\mathcal{F}$ one has (using several adjunctions):
\begin{eqnarray*}
Hom_{\bar{X}_W/\mathcal{F}}(\mathcal{U},j_*j^*\mathcal{L}\times\mathcal{F})
&=&Hom_{\bar{X}_W}(\mathcal{U},j_*j^*\mathcal{L})\\
&=&Hom_{B_{W_F}}(j^*\mathcal{U},j^*\mathcal{L})\\
&=&Hom_{B_{W_F}/j^*\mathcal{F}}(j^*\mathcal{U},j^*(\mathcal{L}\times\mathcal{F}))\\
&\simeq&Hom_{B_{W_F}/j^*\mathcal{F}}(j^*\mathcal{U},j^*(t^*\mathcal{S}\times\mathcal{F}))\\
&=&Hom_{B_{W_F}}(j^*\mathcal{U},j^*t^*\mathcal{S})\\
&=&Hom_{\bar{X}_W}(\mathcal{U},j_*j^*t^*\mathcal{S})\\
&=&Hom_{\bar{X}_W/\mathcal{F}}(\mathcal{U},j_*j^*t^*\mathcal{S}\times\mathcal{F})
\end{eqnarray*}
Hence we have an isomorphism over $\mathcal{F}$
$$j_*j^*\mathcal{L}\times\mathcal{F}\simeq j_*j^*t^*\mathcal{S}\times\mathcal{F},$$
and a commutative diagram
\[ \xymatrix{
\mathcal{L}\times\mathcal{F}\ar[d]\ar[r]^{\simeq}&t^*\mathcal{S}\times\mathcal{F}\ar[d]_{\simeq}\\
j_*j^*\mathcal{L}\times\mathcal{F}\ar[r]^{\simeq}&j_*j^*t^*\mathcal{S}\times\mathcal{F}
}\] where the map $t^*\mathcal{S}\times\mathcal{F}\rightarrow
j_*j^*t^*\mathcal{S}\times\mathcal{F}$ is an isomorphism by
Corollary \ref{cor-constant-generic}. This shows that the morphism
(\ref{adjunction-local}) is an isomorphism. But $\mathcal{F}\rightarrow e$ is epimorphic, so that the base change functor $\bar{X}_W\rightarrow\bar{X}_W/\mathcal{F}$ is faithful, hence conservative.
Therefore the adjunction map
$$\mathcal{L}\longrightarrow j_*j^*\mathcal{L}$$
is an isomorphism.

\end{proof}

The following theorem is the main result of this paper.
Data \ref{choices-X} gives a geometric point $q_{\bar{X}}:Spec\,\overline{F}\rightarrow\bar{X}$.
Then we defined a $\mathcal{T}$-point of $\bar{X}_W$ (see Proposition \ref{prop-point-XL-and-j}):
$$p_{\bar{X}}:\mathcal{T}\longrightarrow \bar{X}_W.$$
Recall also that the Weil-\'etale topos of a connected étale $\bar{X}$-scheme $\bar{U}$ is defined as the slice topos $$\bar{U}_W:=\bar{X}_W/\gamma^*\bar{U}.$$
We consider below the topological pro-group $\underline{W}(\bar{U},q_{\bar{U}})$
introduced in Definition \ref{def-pro-top-grp}.

\begin{thm}\label{thm-big-WEfundgrp}
For any connected étale $\bar{X}$-scheme $\bar{U}$, one has

\emph{\,\,(i)} The topos $\bar{U}_W$ is connected and locally connected over $\mathcal{T}$.

\emph{\,(ii)} A geometric point $q_{\bar{U}}$ of the scheme $\bar{U}$ over $q_{\bar{X}}$ induces a $\mathcal{T}$-valued point
$p_{\bar{U}}$ over $p_{\bar{X}}$ of the Weil-\'etale topos $\bar{U}_W$, and respectively.

\emph{(iii)} One has an isomorphism of topological pro-groups
$$\pi_1(\bar{U}_W,p_{\bar{U}})\simeq \underline{W}(\bar{U},q_{\bar{U}}).$$
\end{thm}
\begin{proof}

$\mathbf{(i)}$ Composing the localization map
$l_{\bar{U}}:\bar{X}_{W}/\gamma^*\bar{U}\rightarrow\bar{X}_{W}$ with $t$, we
obtain the canonical morphism
$$t_{\bar{U}}:\bar{U}_{W}:=\bar{X}_{W}/\gamma^*\bar{U}\longrightarrow\bar{X}_{W}\longrightarrow\mathcal{T}.$$
The morphism $l_{\bar{U}}$ is locally connected, since it is a localization map, i.e. a local homeomorphism (the left adjoint of $l^*_{\bar{U}}$ is $l_{\bar{U}!}(\mathcal{F}\rightarrow\gamma^*\bar{U}):=\mathcal{F}$). By \cite{elephant} C3.3.2, the class of locally connected morphisms is closed under composition. Hence $t_{\bar{U}}$ is locally connected, i.e. $t_{\bar{U}}^*$ has a $\mathcal{T}$-indexed left adjoint $t_{\bar{U}!}$. This functor is defined as follows $t_{\bar{U}!}=t_!\circ l_{\bar{U}!}$, so that we have
$$t_{\bar{U}!}=t_!\circ l_{\bar{U}!}(\mathcal{F}\rightarrow\gamma^*\bar{U})=t_!(\mathcal{F}).$$
for any object $\mathcal{F}\rightarrow\gamma^*\bar{U}$ of the slice topos $\bar{U}_{W}$. Let $Id_{\gamma^*\bar{U}}$ be the final object of $\bar{U}_{W}$. Then
$$t_{\bar{U}!}(Id_{\gamma^*\bar{U}})=t_!(\gamma^*\bar{U})=\{*\}$$
is the final object of $\mathcal{T}$ since $\bar{U}$ is connected (see Remark \ref{rem-connected-U-connected-tU}). It follows from (\cite{elephant} C3.3.3) that $t_{\bar{U}}:\bar{U}_{W}\rightarrow\mathcal{T}$ is connected and locally connected.

One can also give the following easier - but less canonical - argument. By Proposition \ref{prop-site-local-Licht-topos},
$(T_{\bar{U}},\mathcal{J}_{ls})$ is a site for the topos $U_{L}$. The proof of Proposition \ref{prop-nice} is still valid by replacing
$T_{\bar{X}}$ with $T_{\bar{U}}$ (without any other change). This shows that $\bar{U}_{W}$ is connected and locally connected
over $\mathcal{T}$. \\

$\mathbf{(ii)}$ A geometric point
$q_{\bar{U}}:Spec\,\overline{F}\rightarrow\bar{U}$ gives a point of the
étale topos
$$q_{\bar{U}}:\underline{Sets}\longrightarrow\bar{U}_{et}$$
where $q^*_{\bar{U}}$ is the usual fiber functor. We obtain a morphism
$$p_{\bar{U}}=q_{\bar{U}}\times_{q_{\bar{X}}}p_{\bar{X}}:
\mathcal{T}=\underline{Sets}\times_{\underline{Sets}}\mathcal{T}\longrightarrow\bar{U}_{et}\times_{\bar{X}_{et}}\bar{X}_{W}=:\bar{X}_{W}$$
defined over $\mathcal{T}$. One can recover the
geometric point $q_{\bar{U}}$ from $p_{\bar{U}}$. Indeed, let
$p_{\bar{U}}:\mathcal{T}\rightarrow\bar{U}_W$ be a
$\mathcal{T}$-point of $\bar{U}_W$. Then we have a point of the
étale topos
\begin{equation}\label{pointUet}
\underline{Sets}\longrightarrow\mathcal{T}\longrightarrow\bar{U}_W\longrightarrow\bar{U}_{et},
\end{equation}
where the map $\underline{Sets}\rightarrow\mathcal{T}$ is the
canonical one (see \cite{SGA4} IV4.10). By (\cite{SGA4} VIII Theorem
7.9), the category of points of the étale topos of a scheme is
equivalent to the category of geometric points (algebraic and
separable) and specialization maps. Then the map (\ref{pointUet})
corresponds to the given geometric point $q_{\bar{U}}$ of $\bar{U}$. However, two distinct $\mathcal{T}$-points of $\bar{U}_W$ over $p_{\bar{X}}$ can induce the same $\underline{Sets}$-valued point of $\bar{U}_{et}$, hence the same geometric point.\\

$\mathbf{(iii)}$  We make the choices listed in Data \ref{choices-U}. Proposition \ref{prop-site-local-Licht-topos} yields an equivalence $$\widetilde{(T_{\bar{U}},\mathcal{J}_{ls})}\longrightarrow\bar{U}_W.$$
This equivalence provides us with the morphism
$$j:B_{W_K}\longrightarrow\widetilde{(T_{\bar{U}},\mathcal{J}_{ls})}\simeq\bar{U}_W$$
corresponding to the generic point of the connected
\'etale $\bar{X}$-scheme $\bar{U}$. Then the $\mathcal{T}$-point $p_{\bar{U}}$ defined in \textrm{(ii)}, using the geometric point $q_{\bar{U}}$ of $\bar{U}$ given by Data \ref{choices-U}, is isomorphic to the map defined over $\mathcal{T}$ :
$$p:=j\circ u:\mathcal{T}\longrightarrow B_{W_K}\longrightarrow\widetilde{(T_{\bar{U}},\mathcal{J}_{ls})}\simeq\bar{U}_W$$
where $\mathcal{T}\rightarrow B_{W_K}$ is the canonical
$\mathcal{T}$-point of $B_{W_K}$ (see
Proposition \ref{prop-point-XL-and-j}). In order to ease the notations, we denote here by $j$ and $p$ the maps $j_{\bar{U}}$ and $p_{\bar{U}}$. Finally, we denote by $u:B_{W_{K}}\rightarrow\mathcal{T}$ the
canonical map, i.e. the map induced by the morphism of groups $W_K\rightarrow1$.

If $\mathcal{L}$ is an
object of $\bar{U}_W$, then $j^*\mathcal{L}$ is the object
$p^*\mathcal{L}$ of $\mathcal{T}$ endowed with an action of
$y(W_K)$. In other words, $j^*\mathcal{L}$ comes with a morphism of groups in
$\mathcal{T}$ :
$$y(W_K)\longrightarrow\underline{Aut}_{\mathcal{T}}(p^*\mathcal{L}).$$
The following proof consists in two steps :\\

\textbf{Step 1 : We define a projective system of
Galois torsors in
the topos} $\bar{U}_W$.\\

Let $\bar{F}/L/K$ be a finite Galois subextension given by a
geometric point $q_{\bar{U}}:Spec\,\bar{F}\rightarrow\bar{U}$ over
$\bar{X}$. Consider the topological group $W(\bar{U},L)$ of
definition \ref{def-W_LU}. The morphism of left exact sites
\begin{equation}\label{morphism-sites-from-UL-toT}
{t_{\bar{U}}^*}:(Top,\mathcal{J}_{op})\longrightarrow
(T_{\bar{U}},\mathcal{J}_{ls})
\end{equation}
factors through the morphism
\begin{equation}\label{morphism-sites-from-UL-to-BW(UL)}
\appl{(B_{Top}W(\bar{U},L),\mathcal{J}_{ls})}{(T_{\bar{U}},\mathcal{J}_{ls})}{Z}{(Z,Z,Id_Z)}
\end{equation}
where $W_{K}$ acts on $Z$ via the morphism $W_{K}\rightarrow
W(\bar{U},L)$. Respectively, $W_{k(u)}$ acts on $Z$ via the morphism
$W_{k(v)}=W_{K_u}/W^1_{K_u}\rightarrow W(\bar{U},L)$. We obtain a
commutative diagram of topoi
\[ \xymatrix{
B_{W_K}\ar[d]_{j_{\bar{U}}}\ar[r]&B_{W(\bar{U},L)}\ar[d]\\
\bar{U}_{W}\ar[ru]^{\pi}\ar[r]_{t_{\bar{U}}}&\mathcal{T} }\] where
the map $B_{W_K}\rightarrow B_{W(\bar{U},L)}$ is induced by the
surjection $${W_K}\longrightarrow W_K/N(\bar{U},L)=W(\bar{U},L).$$
The map $\pi:\bar{U}_W\rightarrow B_{W(\bar{U},L)}$ corresponds to
the torsor
$$\textsc{Tors}(\bar{U},L):=\pi^*EW(\bar{U},L)$$ where
$EW(\bar{U},L)$ is the universal torsor of $B_{W(\bar{U},L)}$ given
by $W(\bar{U},L)$ acting on itself by multiplications. Note that
$\textsc{Tors}(\bar{U},L)$ is a torsor of
group $W(\bar{U},L)$, which is \emph{connected} over $\mathcal{T}$. Indeed, its space of connected components
$$t_{\bar{U},!}\textsc{Tors}(\bar{U},L)=yW(\bar{U},L)/yW_K$$
is the final object of $\mathcal{T}$, since $yW_K\rightarrow yW(\bar{U},L)$ is an epimorphism in $\mathcal{T}$. 
The last claim follows from the fact that $W_K \rightarrow W(\bar{U},L)$ has local sections since $W(\bar{U},L)$ is finite dimensional (see Theorem \ref{thm-loc-sections} and Lemma \ref{lem-C_L-finite-dim}).
The topological pro-group
$$\underline{W}(\bar{U},q_{\bar{U}}):=\{W(\bar{U},L),\mbox{ for $\bar{F}/L/K$ finite Galois}\}$$
yields a projective system of connected torsors
\begin{equation}\label{pro-torsor}
\{\textsc{Tors}(\bar{U},L),\mbox{ for $\bar{F}/L/K$ finite Galois}\}
\end{equation}
This projective system of torsors is given by compatible maps to
classifying topoi. By the universal property of projective limits,
the pro-torsor (\ref{pro-torsor}) corresponds to an essentially unique morphism
$$\bar{U}_W\longrightarrow\underleftarrow{lim}\,B_{W(\bar{U},L)}=:B_{\underline{W}(\bar{U},q_{\bar{U}})}$$
into the classifying topos of the topological pro-group
$\underline{W}(\bar{U},q_{\bar{U}})$.\\

\textbf{Step 2 : The pro-torsor (\ref{pro-torsor}) is universal.}\\

In other words, we have to show that any locally constant object
$\mathcal{L}$ of $\bar{U}_W$ over $\mathcal{T}$ is trivialized by a
torsor of the form $\textsc{Tors}(\bar{U},L)$. This is the technical
part of the proof.

\textbf{Step 2.1.} Let $\mathcal{L}$ be such a locally constant
object. There exist an object $\mathcal{S}$ of $\mathcal{T}$, an
epimorphism $\mathcal{F}\rightarrow e$ where
$e$ is the final object of $\bar{U}_W$, and an isomorphism
\begin{equation}\label{iso-loc-contst}
\mathcal{L}\times\mathcal{F}\simeq
t^*\mathcal{S}\times\mathcal{F}
\end{equation}
over $\mathcal{F}$. Since the full-subcategory
$\mathbb{G}_{\bar{U}}$ of $\bar{U}_W$ defined in section
\ref{subsect-site} is a generating subcategory (see Theorem
\ref{thm-site}), one can assume that $\mathcal{F}$ is
representable by a sum of objects in $\mathbb{G}_{\bar{U}}$:
$$\mathcal{F}=\coprod_{i\in I}\mathcal{F}_i=y\,\mathcal{G}_{L_i,V_i,\mathrm{u}_i,T_i}.$$
For any index $i\in I$, a point of the topological space
$T_i\neq\emptyset$ yields a morphism
$$\mathcal{G}_{L_i,V_i,\mathrm{u}_i}:=\mathcal{G}_{L_i,V_i,\mathrm{u}_i,*}\longrightarrow\mathcal{G}_{L_i,V_i,\mathrm{u}_i,T_i}$$
in the category $T_{\bar{U}}$, where $*$ denotes the one point space
as usual.

Recall that $\mathcal{G}_{L_i,V_i,\mathrm{u}_i}$ is defined as
follows. Here $\bar{F}/L_i/K$ be a finite Galois sub-extension,
$V_i$ is an open subset of $\bar{U}$, $\mathrm{u}_i$ is a point of
$\bar{U}$ and  $T_i$ is a separated topological space. We denote by
$N_i$ the closed normal subgroup of $W_{L_i/K}$ generated by the
subgroups $\widetilde{W}_{K_v}^1\subseteq W_{L_i/K}$ for any point
$v\in V_i$. Let $(N_i,\widetilde{W}_{K_{\mathrm{u}_i}}^1)$ be the
compact subgroup of $W_{L_i/K}$ generated by $N_i$ and
$\widetilde{W}_{K_{\mathrm{u}}}^1$. The object
$$\mathcal{G}_i:=y\mathcal{G}_{L_i,V_i,\mathrm{u}_i}=y(G_{i,0},G_{i,u},g_{i,u})$$
is then defined as follows. Assume that $\mathrm{u}_i$ is not in $V_i$. We
consider
$$G_{i,0}=W_{L_i/K}/(N_i,\widetilde{W}^1_{K_{\mathrm{u}_i}})$$ as a $W_K$-space and
$$G_{i,v}=W_{L_i/K}/(N_i,\widetilde{W}^1_{K_{\mathrm{u}_i}})$$ as a $W_{k(v)}$-space for
any point $v$ of $V_i\subseteq\bar{U}$. Then we set
$G_{i,\mathrm{u}_i}=W_{k(\mathrm{u}_i)}$ on which
$W_{k(\mathrm{u}_i)}$ acts by multiplication, and
$G_{i,u}=\emptyset$ anywhere else.

Note that the image of $\mathcal{G}_i$ in the final object
$y\bar{U}$ of $\bar{U}_W$, i.e. the support of the sheaf
$\mathcal{G}_i$, is precisely the subobject of $y\bar{U}$ given by
$$V_i\cup\{\mathrm{u}_i\}\hookrightarrow\bar{U}.$$ The
family $$\{\mathcal{G}_i\rightarrow y\bar{U},\,i\in I\}$$
is an epimorphic family, i.e. a covering family of the final object of the topos $\bar{U}_W$ for the canonical topology.
Indeed, the corresponding family of $T_{\bar{U}}$ is a local section cover, as it follows from the facts that a map from an non-empty space to the one point space is a local section cover and that we have
\begin{equation}\label{recouvrement}
\bigcup_{i\in I}\,(V_i\cup\{\mathrm{u}_i\})=\bar{U}.
\end{equation}
\textbf{Step 2.2.} Applying the base change functor along the map (given by any point of $T_i$)
$$\mathcal{G}_i:=y\mathcal{G}_{L_i,V_i,\mathrm{u}_i}\longrightarrow y\mathcal{G}_{L_i,V_i,\mathrm{u}_i,T_i}=\mathcal{F}_i$$
to the trivialization (\ref{iso-loc-contst}), we obtain an
isomorphism over $\mathcal{G}_i$:
$$\mathcal{L}\times\mathcal{G}_i=(\mathcal{L}\times\mathcal{F}_i)\times_{\mathcal{F}_i}\mathcal{G}_i\simeq
(t^*\mathcal{S}\times\mathcal{F}_i)\times_{\mathcal{F}_i}\mathcal{G}_i=
t^*\mathcal{S}\times \mathcal{G}_i$$ where $\mathcal{S}$ is an
object of $\mathcal{T}$. Applying in turn the functor $j^*$ we get
an isomorphism
\begin{equation}\label{une-ptite-derniere}
j^*\mathcal{L}\times yG_{i,0}=j^*(\mathcal{L}\times
\mathcal{G}_i)\simeq
j^*(t^*\mathcal{S}\times\mathcal{G}_i)=j^*t^*\mathcal{S}\times
j^*\mathcal{G}_i=u^*\mathcal{S}\times yG_{i,0} \end{equation} over
$y(G_{i,0})=j^*\mathcal{G}_i$, i.e. an isomorphism in the topos
$B_{W_K}/yG_{i,0}$.

Assume for a moment that the action of $yW_K$ on $p^*\mathcal{L}$
(given by the object $j^*\mathcal{L}$ of $B_{W_K}$) factors through
$W_{L_i/K}$. In other words, suppose that one has a commutative
triangle
\[\xymatrix{
yW_K\ar[d]_{}\ar[rd]&\\
yW_{L_i/K}\ar[r]& \underline{Aut}_{\mathcal{T}}(p^*\mathcal{L})}\]
Then $j^*\mathcal{L}$ is an object of the full subcategory
$$B_{W_{L_i/K}}\hookrightarrow B_{W_K}.$$ Recall that
$$G_{i,0}=W_{L_i/K}/(N,\widetilde{W}^1_{K_{\mathrm{u}_i}}).$$
But there is a canonical equivalence
$$B_{W_{L_i/K}}/yG_{i,0}:=B_{W_{L_i/K}}/y(W_{L_i/K}/(N_i,W_{\widetilde{K}_{\mathrm{u}_i}}^1))=
B_{W_{L_i/K}}/(yW_{L_i/K}/y(N_i,\widetilde{W}^1_{K_{\mathrm{u}_i}})\simeq
B_{(N_i,\widetilde{W}^1_{K_{\mathrm{u}_i}})},$$ where the second
equality follows (by \cite{MatFlach} Lemma 3) from the fact that the projection $$W_{L_i/K}\longrightarrow
W_{L_i/K}/(N_i,\widetilde{W}^1_{K_{\mathrm{u}_i}})$$ admits local
sections, as it follows from Corollary \ref{cor-local-sections} and Lemma \ref{lem-C_L-finite-dim}. Let us make this equivalence
more explicit. The homogeneous space $G_{i,0}$ has a distinguished
(non-equivariant) point $*\rightarrow G_{i,0}$, and we have
$$
\appl{B_{W_{L_i/K}}/yG_{i,0}}{B_{(N_i,\widetilde{W}^1_{K_{\mathrm{u}_i}})}}{(\mathcal{X}\rightarrow
yG_{i,0})}{\mathcal{X}\times_{yG_{i,0}}*}
$$
Under this equivalence, the base change functor
$$
\appl{B_{W_{L_i/K}}}{B_{W_{L_i/K}}/yG_{i,0}}{\mathcal{F}}{\mathcal{F}\times
yG_{i,0}}
$$
takes a $yW_{L_i/K}$-object $\mathcal{F}$ of $\mathcal{T}$ to the
same object of $\mathcal{T}$ : $$\mathcal{F}=\mathcal{F}\times
yG_{i,0}\times_{yG_{i,0}}*$$ endowed with the induced
$y(N_i,\widetilde{W}^1_{K_{\mathrm{u}_i}})$-action. Therefore
(\ref{une-ptite-derniere}) means that
$y(N_i,\widetilde{W}^1_{K_{\mathrm{u}_i}})$ acts trivially on
$j^*\mathcal{L}$, i.e. $y(N_i,\widetilde{W}^1_{K_{\mathrm{u}_i}})$
is in the kernel of the map
\begin{equation}\label{vraiderniere}
yW_{L_i/K}\longrightarrow{\underline{Aut}_{\mathcal{T}}(p^*\mathcal{L})}.
\end{equation}
Hence the action of $yW_{L_i/K}$ on
$p^*\mathcal{L}$ factors through $$yW_{L_i/K}/y(N_i,\widetilde{W}^1_{K_{\mathrm{u}_i}})=y(W_{L_i/K}/(N_i,\widetilde{W}^1_{K_{\mathrm{u}_i}})).$$
The same argument shows that the action of $yW_K$ on
$p^*\mathcal{L}$ factors through $yW_{L_i/K}$, i.e. that the
commutative triangle considered above exists. Indeed, by
(\ref{une-ptite-derniere}) one has an isomorphism
$$j^*\mathcal{L}\times yW_{L_i/K}=(j^*\mathcal{L}\times yG_{i,0})\times_{yG_{i,0}} yW_{L_i/K}\simeq
(u^*\mathcal{S}\times yG_{i,0})\times_{yG_{i,0}}
yW_{L_i/K}=u^*\mathcal{S}\times yW_{L_i/K}$$ where we consider the
object of $B_{W_K}$
$$yW_{L_i/K}=y(W_{K}/W_{L_i}^c)=yW_{K}/yW_{L_i}^c.$$
Note that $W_{L_i}^c$, which is the closure of the commutator
subgroup of $W_{L_i}$, is compact in $W_K$. Then the previous argument shows that the action of $yW_{K}$ on
$p^*\mathcal{L}$ factors through $yW_{L_i/K}=yW_{K}/yW_{L_i}^c$ (this last identification is valid by Theorem \ref{thm-loc-sections}, Lemma \ref{lem-C_L-finite-dim} and \cite{MatFlach} Lemma 3). In summary, we have proven the following
\begin{prop}\label{prop-between-proof}
The action $yW_K\rightarrow
\underline{Aut}_{\mathcal{T}}(p^*\mathcal{L})$ induces a morphism
$$\rho_i:\,yW_{L_i/K}\longrightarrow{\underline{Aut}_{\mathcal{T}}(p^*\mathcal{L})}$$
for any $i\in I$, and we have
$$y(N_i,\widetilde{W}^1_{K_{\mathrm{u}_i}})\,\hookrightarrow\,Ker(\rho_i).$$
\end{prop}

\textbf{Step 2.3.}  Choose an object $\mathcal{G}_{i_0}$ (i.e. an
index $i_0\in I$). We have a morphism
$$\rho_{i_0}:yW_{L_{i_0}/K}\longrightarrow{\underline{Aut}_{\mathcal{T}}(p^*\mathcal{L})}.$$
Let $N_{i_0}$ be the closed normal subgroup of $W_{L_{i_0}/K}$
defined above. We have an open subset $V_{i_0}\subseteq\bar{U}$ such
that $\widetilde{W}_{K_v}^1\subseteq N_{i_0}$ for any $v\in
V_{i_0}$. We claim that, for any $u\in\bar{U}-V_{i_0}$, the subgroup
$y\widetilde{W}_{K_u}^1\subset yW_{L_{i_0}/K}$ is in the kernel of
$\rho_{i_0}$, i.e. one has
\begin{equation}\label{uneidee}
y\widetilde{W}_{K_u}^1\,\hookrightarrow\,Ker(\rho_{i_0})\mbox{ ; for
any }u\in\bar{U}-V_{i_0}.
\end{equation}
Let $u\in\bar{U}-V_{i_0}$. By (\ref{recouvrement}), there exists an
index $i\in I$ such that the support of $\mathcal{G}_i$ contains
$u$, i.e. one has :
$$u\in V_i\cup\{\mathrm{u}_i\}.$$
Proposition \ref{prop-between-proof} then shows that the subgroup
$y\widetilde{W}_{K_u}^1\subset yW_{L_{i}/K}$ is in the kernel of the
morphism
$$\rho_i:yW_{L_i/K}\longrightarrow{\underline{Aut}_{\mathcal{T}}(p^*\mathcal{L})}.$$
Hence the subgroup $yW_{K_u}^1\subset yW_{K}$ is in the kernel of
the morphism
$$\rho:yW_{K}\longrightarrow{\underline{Aut}_{\mathcal{T}}(p^*\mathcal{L})}.$$
It follows that the image of $yW_{K_u}^1$ in $yW_{L_{i_0}/K}$ is in
the kernel of $\rho_{i_0}$. But the continuous surjection
$$W_{K_u}^1\,\twoheadrightarrow\,\widetilde{W}_{K_u}^1\subset W_{L_{i_0}/K}$$
admits local sections (using Theorem \ref{thm-loc-sections} and the fact that $\widetilde{W}_{K_u}^1$ is finite dimensional), hence induces an epimorphism in
$\mathcal{T}$. Thus the image of $yW_{K_u}^1$ in $yW_{L_{i_0}/K}$ is
$y\widetilde{W}_{K_u}^1\subset y(W_{L_{i_0}/K})$, which is therefore
in the kernel of $\rho_{i_0}$. We have proven (\ref{uneidee}).

Let $N(\widetilde{W}^1_{K_u})$ be the normal topological subgroup of
$W_{L_{i_0}/K}$ generated by the subgroup $\widetilde{W}_{K_u}^1$.
By Lemma \ref{lem-la-clef}, $N(\widetilde{W}^1_{K_u})$ is compact
and we have
$$\mathcal{N}(y\widetilde{W}^1_{K_u})=yN(\widetilde{W}^1_{K_u}),$$
where $\mathcal{N}(y\widetilde{W}^1_{K_u})$ is the normal subgroup
of $y{W}_{L_{i_0}/K}$ generated by $y\widetilde{W}^1_{K_u}$. We
obtain
$$yN(\widetilde{W}^1_{K_u})=\mathcal{N}(y\widetilde{W}^1_{K_u})\,\,\hookrightarrow\,\,Ker(\rho_{i_0}).$$
Therefore, the subgroup of $y{W}_{L_{i_0}/K}$ generated by
$yN_{i_0}$ and $yN(\widetilde{W}^1_{K_u})$, for any
$u\in\bar{U}-V_{i_0}$, is contained in $Ker(\rho_{i_0})$ :
$$<yN_{i_0}\,;\,yN(\widetilde{W}^1_{K_u})\mbox{ ; for any }u\in\bar{U}-V_{i_0}>\,\hookrightarrow\,Ker(\rho_{i_0})$$
The topological subgroup of ${W}_{L_{i_0}/K}$, generated by
$N_{i_0}$ and $N(\widetilde{W}^1_{K_u})$ for any
$u\in\bar{U}-V_{i_0}$, is normal and compact. Hence this subgroup is
precisely $N(\bar{U},L_{i_0})$, which is the closed normal subgroup
of $W_{L_{i_0}/K}$ generated by all the subgroups
$\widetilde{W}_{K_u}^1\subset W_{L_{i_0}/K}$ for any $u\in\bar{U}$
(see section \ref{subsect-Weil-group}).

Lemma \ref{painfullemma1} then shows that
$$yN(\bar{U},L_{i_0})\,\hookrightarrow\,Ker(\rho_{i_0})$$
and that the morphism $\rho_{i_0}$ induces a morphism from
$$y{W}_{L_{i_0}/K}/yN(\bar{U},L_{i_0})=y({W}_{L_{i_0}/K}/N(\bar{U},L_{i_0}))=yW(\bar{U},L_{i_0})$$
to the automorphism group
$$\underline{Aut}_{\mathcal{T}}(p^*\mathcal{L}).$$
Therefore we have an isomorphism
$$j^*\mathcal{L}\times yW(\bar{U},L_{i_0})\simeq u^*\mathcal{S}\times yW(\bar{U},L_{i_0})$$
in the topos $B_{W_K}/yW(\bar{U},L_{i_0})$, where $\mathcal{S}$ is
an object of $\mathcal{T}$. In other words we have
$$j^*\mathcal{L}\times j^*\textrm{Tors}(\bar{U},L_{i_0})\simeq u^*\mathcal{S}\times
j^*\textsc{Tors}(\bar{U},L_{i_0})$$ since
$$j^*\textsc{Tors}(\bar{U},L_{i_0})=yW(\bar{U},L_{i_0})$$
where $\textsc{Tors}(\bar{U},L_{i_0})$ is the torsor corresponding to
the morphism
$$\pi:\bar{U}_W\longrightarrow B_{W(\bar{U},L_{i_0})}$$ defined
in Step 1.

\textbf{Step 2.4.} The torsor $\textsc{Tors}(\bar{U},L_{i_0})$ is
locally constant over $\mathcal{T}$, since any torsor is trivialized
by itself, hence Lemma \ref{loc-cnstant-generic} applies. We obtain
an isomorphism over $\textsc{Tors}(\bar{U},L_{i_0})$ :
\begin{eqnarray*}
\mathcal{L}\times\textsc{Tors}(\bar{U},L_{i_0})
&\simeq&j_*j^*(\mathcal{L}\times\textsc{Tors}(\bar{U},L_{i_0}))\\
&\simeq&j_*j^*\mathcal{L}\times j_*j^*\textsc{Tors}(\bar{U},L_{i_0})\\
&\simeq&j_*(j^*\mathcal{L}\times j^*\textsc{Tors}(\bar{U},L_{i_0}))\\
&\simeq&j_*(u^*\mathcal{S}\times j^*\textsc{Tors}(\bar{U},L_{i_0}))\\
&\simeq&j_*u^*\mathcal{S}\times j_*j^*\textsc{Tors}(\bar{U},L_{i_0})\\
&\simeq&j_*j^*t^*\mathcal{S}\times j_*j^*\textsc{Tors}(\bar{U},L_{i_0})\\
&\simeq&t^*\mathcal{S}\times\textsc{Tors}(\bar{U},L_{i_0}).
\end{eqnarray*}
This shows that any locally constant object $\mathcal{L}$ of
$\bar{U}_W$ over $\mathcal{T}$ is trivialized by a torsor of the
form $\textsc{Tors}(\bar{U},L)$. Hence the pro-torsor
(\ref{pro-torsor})
$$\{\textsc{Tors}(\bar{U},L),\mbox{ for $\bar{F}/L/K$ finite Galois}\}$$
is universal. The pro-group object of $\mathcal{T}$ defined by this
pro-torsor is the projective system of its Galois groups :
$$y\,\underline{W}(\bar{U},q_{\bar{U}}):=\{yW(\bar{U},L),\mbox{ for $\bar{F}/L/K$ finite Galois}\}.$$
Equivalently, this pro-group object of $\mathcal{T}$ is obtain by
applying the fiber functor $p_{\bar{U}}^{*}$ :
$$y\,\underline{W}(\bar{U},q_{\bar{U}}):= p_{\bar{U}}^{*}\,\,\{\textsc{Tors}(\bar{U},L),\mbox{ for $\bar{F}/L/K$ finite Galois}\}.$$
This yields an isomorphism of pro-group objects in $\mathcal{T}$
$$\pi_1(\bar{U}_W,p_{\bar{U}})\,\simeq \,y\,\underline{W}(\bar{U},q_{\bar{U}}).$$
Hence $\pi_1(\bar{U}_W,p_{\bar{U}})$ is a topological pro-group
canonically isomorphic to $\underline{W}(\bar{U},q_{\bar{U}})$.
\end{proof}

\section{Weil-\'etale Cohomology with coefficients in $\widetilde{\mathbb{R}}$}
Let $\bar{U}$ be a connected \'etale $\bar{X}$-scheme. For any topos $t:\mathcal{E}\rightarrow\mathcal{T}$ defined over $\mathcal{T}$, we set $\widetilde{\mathbb{R}}:=t^*(y\mathbb{R})$, where $y\mathbb{R}$ is the sheaf of $\mathcal{T}$ represented by the standard topological group $\mathbb{R}$.

\begin{lem}
Let $j:B_{W_K}\rightarrow \bar{U}_W$ be the canonical map. We have $j_*\widetilde{\mathbb{R}}=\widetilde{\mathbb{R}}$ and
$R^nj_*\widetilde{\mathbb{R}}=0$ for any $n\geq1$.
\end{lem}
\begin{proof}
The identification $j_*\widetilde{\mathbb{R}}=\widetilde{\mathbb{R}}$ follows immediately from
$$Hom_{T_{\bar{U}}}((Z_0,Z_u,f_u),(\mathbb{R},\mathbb{R},Id))=Hom_{B_{Top}W_K}(Z_0,\mathbb{R}).$$
where $(Z_0,Z_u,f_u)$ is any object of $T_{\bar{U}}$.
By Theorem \ref{thm-site}, the site $(\mathbb{G}_{\bar{U}},\mathcal{J}_{ls})$ is a site for $\bar{U}_W$. Then $R^nj_*\widetilde{\mathbb{R}}$
is the sheaf on $(\mathbb{G}_{\bar{U}},\mathcal{J}_{ls})$ associated to the presheaf
$$
\fonc{P^nj_*\widetilde{\mathbb{R}}}{\mathbb{G}_{\bar{U}}}{Ab}{\mathcal{G}_{L,V,u,T}}
{H^n(B_{W_K}/(j^*\mathcal{G}_{L,V,u,T}),\widetilde{\mathbb{R}})}
$$
for any $n\geq1$. Recall that one has
$$j^*\mathcal{G}_{L,V,u,T}= W_{L/K}/(N,\widetilde{W}^1_{K_u})\times T$$
where $N$ is the closure of the normal subgroup of $W_{L/K}$ generated by the images of the maps $W^1_{K_v}\rightarrow W_{L/K}$
where $v$ runs over the closed points of $V\subset\bar{U}$ (see section \ref{subsect-site}). One can write $W_{L/K}/(N,\widetilde{W}^1_{K_u})=W_K/\Lambda$
where $\Lambda$ is a compact subgroup of $W_K$. The map $W_K\rightarrow W_K/\Lambda$ has local sections as it follows from Theorem \ref{thm-loc-sections} and from the fact that $W_K/\Lambda=W_{L/K}/(N,\widetilde{W}^1_{K_u})$ is finite dimensional. We obtain $yW_K/y\Lambda=y(W_K/\Lambda)$, and the following identifications:
\begin{eqnarray*}
B_{W_K}/(j^*\mathcal{G}_{L,V,u,T})&=&B_{W_K}/y(W_{L/K}/\Lambda\times T)\\
&=&B_{W_K}/(yW_{L/K}/y\Lambda\times yT)\\
&=&(B_{W_K}/(yW_{L/K}/y\Lambda))/ yT\\
&=&B_{\Lambda}/T
\end{eqnarray*}
Therefore, for any $n\geq1$, one has
$$P^nj_*\widetilde{\mathbb{R}}(\mathcal{G}_{L,V,u,T})=H^n(B_{\Lambda}/T,\widetilde{\mathbb{R}}).$$
Consider the pull-back square
\[ \xymatrix{
\mathcal{T}/T\ar[r]^a\ar[d]^b&\mathcal{T}\ar[d]^d\\
B_{\Lambda}/T\ar[r]^c&B_{\Lambda}
}\]
This pull-back square is obtained by localization since $B_{\Lambda}/E\Lambda=\mathcal{T}$ and $(B_{\Lambda}/T)/(E\Lambda\times T)=\mathcal{T}/T$. One checks immediately that such a pull-back satisfies the Beck-Chevalley condition $d^*c_*\simeq a_*b^*$ (this is a special case of the Beck-Chevalley condition for locally connected morphisms). But $b^*$ is a localization functor, hence it preserves injective abelian objects. We obtain
\begin{equation}\label{une-BC-cond}
d^*R^m(c_*)\simeq R^m(a_*)b^*.
\end{equation}
The sheaf $R^m(a_*)(\widetilde{\mathbb{R}})$ is the sheaf associated to the presheaf
$$
\fonc{P^m(a_*)(\widetilde{\mathbb{R}})}{Top}{Ab}{T'}{H^m(\mathcal{T}/(T\times T'),\widetilde{\mathbb{R}})=H^m(Sh(T\times T'),C^0(\mathbb{R}))}
$$
where $C^0(\mathbb{R})$ denotes the sheaf of germs of continuous real valued functions on the locally compact space $T\times T'$, and $Sh(T\times T')$ is the topos of sheaves (i.e. of \'etal\'e spaces) on $T\times T'$. The isomorphism $H^m(\mathcal{T}/(T\times T'),\widetilde{\mathbb{R}})=H^m(Sh(T\times T'),C^0(\mathbb{R}))$ follows from the fact that the big topos $\mathcal{T}/(T\times T')$ of the space $T\times T'$ is cohomologically equivalent to $Sh(T\times T')$ (see \cite{SGA4} IV 4.10). But $T\times T'$ is locally compact hence paracompact, so that the sheaf $C^0(\mathbb{R})$ is "fin" on $T\times T'$ hence acyclic for the global sections functor. We obtain $P^m(a_*)(\widetilde{\mathbb{R}})=0$ for any $m\geq 1$, so $R^m(a_*)(\widetilde{\mathbb{R}})=0$ for any $m\geq 1$. Then it follows from (\ref{une-BC-cond}) that $R^m(c_*)(\widetilde{\mathbb{R}})=0$ for any $m\geq 1$, since $d^*$ is faithful. Moreover, by (\ref{une-BC-cond}) with $m=0$, the sheaf $c_*(\widetilde{\mathbb{R}})$ can be identified with $a_*(\widetilde{\mathbb{R}})$ with trivial $y\Lambda$-action, which is in turn represented by the space $\underline{Hom}_{Top}(T,\mathbb{R})$ on which $\Lambda$ acts trivially. Hence the Leray spectal sequence
$$H^n(B_{\Lambda},R^m(c_*)(\widetilde{\mathbb{R}}))\Rightarrow H^{n+m}(B_{\Lambda}/T,\widetilde{\mathbb{R}})$$
degenerates and yields
$$H^n(B_{\Lambda}/T,\widetilde{\mathbb{R}})\simeq H^n(B_{\Lambda},\underline{Hom}_{Top}(T,\mathbb{R}))$$
By (\cite{MatFlach} Corollary 8), we have $H^n(B_{\Lambda},\underline{Hom}_{Top}(T,\mathbb{R}))=0$ for any $n\geq1$, since $\Lambda$ is compact and $\underline{Hom}_{Top}(T,\mathbb{R})$ is a locally convex, Hausdorff and quasi-complete real vector space. We have shown the following:
$$P^nj_*\widetilde{\mathbb{R}}(\mathcal{G}_{L,V,u,T})=H^n(B_{\Lambda}/T,\widetilde{\mathbb{R}})=0$$ for any $n\geq1$ and any object $\mathcal{G}_{L,V,u,T}$ of $\mathbb{G}_{\bar{U}}$. Hence $R^nj_*\widetilde{\mathbb{R}}=0$ for any $n\geq1$.
\end{proof}
\begin{prop}\label{prop-R-cohomology}
We have $H^n(\bar{U}_W,\widetilde{\mathbb{R}})=\mathbb{R}$ for $n=0,1$ and $H^n(\bar{U}_W,\widetilde{\mathbb{R}})=0$ for $n\geq2$.
\end{prop}
\begin{proof}
We use the spectral sequence associated with the morphism $j:B_{W_K}\rightarrow \bar{U}_W$ and obtain
$H^n(\bar{U}_W,\widetilde{\mathbb{R}})\simeq H^n(B_{W_K},\widetilde{\mathbb{R}})$, thanks to the previous Lemma. The latter group can be computed using the product decomposition $W_K=W^1_K\times\mathbb{R}$ and the fact that $W^1_K$ is compact (see \cite{MatFlach}).
\end{proof}

\section{Consequences of the main result}\label{sect-consequences}

\subsection{Direct consequences} In this section, $\bar{U}$ denotes a connected étale
$\bar{X}$-scheme with function field $K$. We consider the classifying topos of the topological pro-group
$\underline{W}(\bar{U},q_{\bar{U}})$, which is defined as the
projective limit:
$$B_{\underline{W}(\bar{U},q_{\bar{U}})}:=\underleftarrow{lim}\,B_{{W}(\bar{U},L)}$$
Recall from section \ref{subsect-loc-cstant} the definition of the category $SLC_{\mathcal{T}}(\bar{U}_W)$
of sums of locally constant objects over
$\mathcal{T}$. The following result, which is an immediate
consequence -in fact a rewriting- of the previous theorem, gives an
explicit description of the category of sums of locally constant
objects.
\begin{cor}
There is an equivalence defined over $\mathcal{T}$ and compatible with the point $p_{\bar{U}}$:
$$SLC_{\mathcal{T}}(\bar{U}_W)\simeq B_{\underline{W}(\bar{U},q_{\bar{U}})}$$
This equivalence is canonically induced by Data \ref{choices-U}.
\end{cor}
\begin{cor} The fundamental group
$\pi_1(\bar{U}_W,p_{\bar{U}})$ is pro-representable by a locally compact strict pro-group indexed over a filtered poset.
\end{cor}
If $\mathcal{G}$ is a group object of
$\mathcal{T}$, then we consider the internal Hom group object
$$\mathcal{G}^D:=\underline{Hom}_{\mathcal{T}}(\mathcal{G},y\mathbb{S}^1).$$
For a locally compact topological group $G$, one can show that
$$(yG)^{DD}\simeq y(G^{ab})$$
is represented by the maximal Hausdorff abelian quotient $G^{ab}$ of $G$ (see \cite{Fund-group-I}). Let $\underline{\mathcal{G}}$ be a pro-group object of
$\mathcal{T}$ given by a covariant functor $\underline{\mathcal{G}}:I\rightarrow Gr(\mathcal{T})$, where $Gr(\mathcal{T})$ denotes the category of groups in $\mathcal{T}$, and $I$ is a small filtered category. We consider the  pro-abelian group object $\underline{\mathcal{G}}^{DD}$ of
$\mathcal{T}$ defined as the composite functor $$(-)^{DD}\circ\underline{\mathcal{G}}:I\longrightarrow Gr(\mathcal{T})\longrightarrow Ab(\mathcal{T}).$$
Recall from Definition \ref{def-CU} the definition of the abelian topological group $C_{\bar{U}}$.
\begin{cor}\label{cor-pi1-WE-ab=CU}
The pro-group object $\pi_1(\bar{U}_W,p_{\bar{U}})^{DD}$ of
$\mathcal{T}$ is essentially constant, hence can be identified with an actual topological group.
Then we have a canonical isomorphism of topological groups
$$r_{\bar{U}}:\,C_{\bar{U}}\simeq \pi_1(\bar{U}_W,p_{\bar{U}})^{DD}.$$
\end{cor}
\begin{proof}
The pro-group object $\pi_1(\bar{U}_W,p_{\bar{U}})^{DD}$ is the projective system of abelian objects given by the groups $(y{W}(\bar{U},L))^{DD}$ for $\overline{K}/L/K$ finite and Galois. But one has $$(y{W}(\bar{U},L))^{DD}=y({W}(\bar{U},L)^{ab})=yC_{\bar{U}}.$$
for any $\overline{K}/L/K$. The second equality has been proved in section \ref{subsect-Weil-group}.
\end{proof}

We simply denote by $t:\bar{U}_W\rightarrow\mathcal{T}$ the canonical map. Since the Weil-étale topos $\bar{U}_W$ is defined over the base topos $\mathcal{T}$, the cohomology groups of $\bar{U}_W$ have a topological structure. To make this precise, we introduce the following notion.
\begin{defn}
The \emph{$\mathcal{T}$-cohomology of $\bar{U}_W$ with coefficients in $\mathcal{A}$} is defined as
$$H^n_{\mathcal{T}}(\bar{U}_W,\mathcal{A}):=R^n(t_*)(\mathcal{A})$$
\end{defn}

\begin{cor}\label{corH^0}
For any abelian object $\mathcal{A}$ of $\mathcal{T}$, one has
$$H^0_{\mathcal{T}}(\bar{U}_W,t^*\mathcal{A})=\mathcal{A}\mbox{ and }H^0(\bar{U}_W,t^*\mathcal{A})=\mathcal{A}(*)$$
where $\mathcal{A}(*)$ denotes the group of global sections of the
abelian object $\mathcal{A}$ of $\mathcal{T}$.
\end{cor}
\begin{proof}
We have
$$H^0_{\mathcal{T}}(\bar{U}_W,t^*\mathcal{A}):=t_*t^*\mathcal{A}=\mathcal{A}$$
since $t:\bar{U}_W\rightarrow\mathcal{T}$ is connected, i.e. $t^*$
is fully faithful. Let $e_{\mathcal{T}}$ be the unique map
$e_{\mathcal{T}}:\mathcal{T}\rightarrow\underline{Sets}$. We have
$$H^0(\bar{U}_W,t^*\mathcal{A}):=(e_{\mathcal{T}*}\circ t_*)\,t^*\mathcal{A}=e_{\mathcal{T}*}\mathcal{A}=\mathcal{A}(*).$$
\end{proof}

\begin{cor}\label{corH^1representable}
For any abelian locally compact topological group ${A}$, one has
$$H^1_{\mathcal{T}}(\bar{U}_W,t^*{A})=\underline{Hom}_{Top}(C_{\bar{U}},{A})
\mbox{ and }H^1(\bar{U}_W,t^*{A})={Hom}_{cont}(C_{\bar{U}},{A}).$$
\end{cor}

\begin{proof}
Let $A$ be an abelian locally compact group. One has
\begin{eqnarray*}
H^1_{\mathcal{T}}(\bar{U}_W,t^*yA)&=&\underline{Hom}_{\mathcal{T}}(\pi_1(\bar{U}_W,p_{\bar{U}}),yA)\\
&=&\underrightarrow{lim}\,\underline{Hom}_{\mathcal{T}}(yW(L,\bar{U}),yA)\\
&=&\underrightarrow{lim}\,y(\underline{Hom}_{Top}(W(L,\bar{U}),A))\\
&=&\underrightarrow{lim}\,y(\underline{Hom}_{Top}(W(L,\bar{U})^{ab},A))\\
&=&y(\underline{Hom}_{Top}(C_{\bar{U}},A)).
\end{eqnarray*}
Here $\underline{Hom}_{Top}(C_{\bar{U}},A)$ is the group of continuous morphisms from $C_{\bar{U}}$ to $A$, endowed with the compact-open topology.

Consider the unique map
$e_{\mathcal{T}}:\mathcal{T}\rightarrow\underline{Sets}$. This map has a canonical section $s_{\mathcal{T}}$ such that $e_{\mathcal{T}*}\simeq s^*_{\mathcal{T}}$. Hence the direct image functor $e_{\mathcal{T}*}:\mathcal{T}\rightarrow\underline{Sets}$ commutes with arbitrary inductive limits (see \cite{SGA4} IV.4.10). Then the first cohomology group
\begin{eqnarray*}
H^1(\bar{U}_W,t^*A)&=&e_{\mathcal{T}*}H^1_{\mathcal{T}}(\bar{U}_W,t^*A)\\
&=&e_{\mathcal{T}*}y\underline{Hom}_{Top}(C_{\bar{U}},A)\\
&=&{Hom}_{cont}(C_{\bar{U}},A)
\end{eqnarray*}
is the discrete group of continuous morphisms from $C_{\bar{U}}$ to $A$.
\end{proof}

\begin{cor}
There is a \emph{fundamental class}
$$\theta_{\bar{U}}\in
H^1(\bar{U}_W,\widetilde{\mathbb{R}})={Hom}_{cont}(C_{\bar{U}},\mathbb{R})$$ given by the
canonical continuous morphism
$$\theta_{\bar{U}}:C_{\bar{U}}\longrightarrow\mathbb{R}.$$
\end{cor}

\begin{rem}
Recall that $C_{\overline{Spec\,\mathbb{Z}}}=Pic(\overline{Spec\,\mathbb{Z}})=\mathbb{R}^{\times}_+$.
For any $\bar{U}$, the fundamental class $\theta_{\bar{U}}$ is the
pull back of the logarithm morphism:
$$\theta_{\overline{Spec\,\mathbb{Z}}}:=\mbox{\emph{log}}\in
H^1(\overline{Spec\,\mathbb{Z}}_W,\widetilde{\mathbb{R}})=Hom_{cont}(\mathbb{R}^{\times}_+,\mathbb{R})
$$
along the map $\bar{U}\rightarrow\overline{Spec\,\mathbb{Z}}$.
\end{rem}

The maximal compact subgroup of $C_{\bar{U}}$, i.e. the kernel of the absolute value map $C_{\bar{U}}\rightarrow\mathbb{R}^{\times}_+$, is denoted by $C^1_{\bar{U}}$. The Pontraygin dual $(C^1_{\bar{U}})^D$ is a discrete abelian group.
\begin{prop}
For any connected \'etale $\bar{X}$-scheme $\bar{U}$, we have canonically
\begin{eqnarray*}
H^n(\bar{U}_{W},\mathbb{Z})
&=&\mathbb{Z}\mbox{ for $n=0$}\\
&=&0 \mbox{ for $n=1$}\\
&=&(C^1_{\bar{U}})^D \mbox{ for $n=2$}.
\end{eqnarray*}
\end{prop}
\begin{proof}
The result for $n=0$ follows from Corollary \ref{corH^0}. By  Corollary \ref{corH^1representable}, we have $$H^1(\bar{U}_{W},\mathbb{Z})=Hom_{c}(C_{\bar{U}},\mathbb{Z})=0.$$
Moreover we have an isomorphism
\begin{eqnarray*}
H^1(\bar{U}_{W},\widetilde{\mathbb{S}}^1)=Hom_{c}(C_{\bar{U}},\mathbb{S}^1)=C_{\bar{U}}^D.
\end{eqnarray*}
The exact sequence of topological groups $0\rightarrow\mathbb{Z}\rightarrow\mathbb{R}\rightarrow\mathbb{S}^1\rightarrow0$
induces an exact sequence
$0\rightarrow\mathbb{Z}\rightarrow \widetilde{\mathbb{R}}\rightarrow \widetilde{\mathbb{S}}^1\rightarrow0$
of abelian sheaves on $\bar{U}_{W}$, where $\widetilde{\mathbb{R}}:=t^*(y\mathbb{R})$ and $\widetilde{\mathbb{S}}^1:=t^*(y\mathbb{S}^1)$. The induced long exact sequence
$$0=H^1(\bar{U}_{W},\mathbb{Z})\rightarrow H^1(\bar{U}_{W},\widetilde{\mathbb{R}})\rightarrow
H^1(\bar{U}_{W},\widetilde{\mathbb{S}}^1)\rightarrow
H^2(\bar{U}_{W},\mathbb{Z})\rightarrow
H^2(\bar{U}_{W},\widetilde{\mathbb{R}})=0$$
is canonically identified with
$$0\rightarrow Hom_{c}(C_{\bar{U}},\mathbb{R})\rightarrow Hom_{c}(C_{\bar{U}},\mathbb{S}^1)\rightarrow
H^2(\bar{U}_{W},\mathbb{Z})\rightarrow 0$$
and we obtain $H^2(\bar{U}_{W},\mathbb{Z})=(C^1_{\bar{U}})^D$.
\end{proof}

\subsection{The Weil-étale topos and the axioms for the conjectural Lichtenbaum topos}
Lichtenbaum conjectured in \cite{Lichtenbaum} the existence of a Grothendieck topology for an arithmetic scheme $X$ such that the Euler characteristic of the cohomology groups of the constant sheaf $\mathbb{Z}$ with compact support at infinity gives, up to sign, the leading term of the zeta-function $\zeta_X(s)$ at $s=0$. We call the category of sheaves on this conjectural site the \emph{conjectural Lichtenbaum topos}, which we denote by  $\bar{X}_L$. In \cite{Fund-group-I} Section 5.2 we gave a list of axioms that should be satisfied by the conjectural topos $\bar{X}_L$, in the case where $X=Spec(\mathcal{O}_F)$. We refer to them as Axioms $(1)-(9)$. We also showed in \cite{Fund-group-I} that any topos satisfying these axioms gives rise to complexes of \'etale sheaves computing the expected Lichtenbaum cohomology. The main motivation for the present work is to provide an example of a topos (the Weil-étale topos) satisfying Axioms $(1)-(9)$. This shows that that Axioms $(1)-(9)$ are consistent, and this gives a natural computation of the base change from the Weil-étale cohomology to the étale cohomology (see Corollary \ref{thm-Lichtenbaum-conj} below). Axioms $(1)-(9)$ are recalled in the proof of Theorem \ref{thm-existence-proof}.

The morphism $\gamma:\bar{U}_W\rightarrow\bar{U}_{et}$ induces a morphism $\varphi_{\bar{U}}$ of fundamental pro-groups. Applying the functor $(-)^{DD}$, we obtain a morphism $\varphi_{\bar{U}}^{DD}$ of abelian fundamental pro-groups.
\begin{cor}\label{cor-WEtopo-CFT}
The composite morphism
$$\varphi_{\bar{U}}^{DD}\circ r_{\bar{U}}:\,C_{\bar{U}}\simeq \pi_1(\bar{U}_W)^{ab}\longrightarrow \pi_1(\bar{U}_{et})^{ab}$$
is the reciprocity law of class field theory.
\end{cor}
\begin{proof}
 The fundamental group
$$\pi_1(\bar{U}_W,p_{\bar{U}})=\underline{W}(\bar{U},q_{\bar{U}}):=\{W(\bar{U},L)\mbox{ , for $\overline{K}/L/K$ finite Galois}\}$$ can be seen as the automorphism group of the pro-torsor
$$\{\textsc{Tors}(\bar{U},L):=\pi^*EW(\bar{U},L),\mbox{ for $\bar{F}/L/K$ finite Galois}\}$$
in $\bar{X}_W$. Consider the morphism of fundamental groups induced by $\gamma$:
$$\varphi_{\bar{U}}:\pi_1(\bar{U}_W,p_{\bar{U}})\longrightarrow\pi_1(\bar{U}_{et},q_{\bar{U}})$$
It follows from the definition of $\gamma$, in terms of morphism of left exact sites (see Proposition \ref{prop-morph-sites-etale-loc-sections}), that $\varphi_{\bar{U}}$ is the morphism of topological pro-groups
$$\{W(\bar{U},L)\mbox{, $\overline{K}/L/K$ finite Galois}\}\rightarrow \{G(L'/K)\mbox{, $\overline{K}/L'/K$ finite Galois unramified over $\bar{U}$}\}$$
is given by the compatible family of morphisms $W(\bar{U},L)\rightarrow G(L^{un}/K)$, where $L^{un}$ is the maximal sub-extension of $L/K$ which is unramified over $\bar{U}$. Indeed, the previous statement follows from the fact that the following square is commutative, where $K_{\bar{U}}/K$ is the maximal sub-extension of $\overline{K}/K$ unramified over $\bar{U}$:
\[ \xymatrix{
\bar{U}_{W}\ar[r]^{\gamma}\ar[d]&\bar{U}_{et}\ar[d]\\
B_{\underline{W}(\bar{U},q_{\bar{U}})}\ar[r]&B^{sm}_{G(K_{\bar{U}}/K)}
}\]
The commutativity of this square in turn follows from the description of these morphisms in terms of morphisms of sites, which is given in Proposition \ref{prop-morph-sites-etale-loc-sections} and (\ref{morphism-sites-from-UL-to-BW(UL)}).

Hence the morphism $\varphi_{\bar{U}}^{DD}\circ r_{\bar{U}}$ is given by the family of compatible morphisms
$$C_{\bar{U}}\simeq W(\bar{U},L)^{ab}\rightarrow G(L^{un}/K)^{ab}$$
indexed over the finite Galois sub-extensions $\overline{K}/L/K$. Let us fix such a sub-extension $L/K$. We consider the usual relative Weil group $W_{L/K}$, which is given with maps
$W_{L/K}\rightarrow G_{L/K}$ and $C_K\simeq W_{L/K}^{ab}$, where $C_K$ is the id\`ele class group of $K$. The corollary now follows from the commutative diagram
\[ \xymatrix{
C_K\ar[r]^{\simeq}\ar[d]&W_{L/K}^{ab}\ar[d]\ar[r]&G(L/K)^{ab}\ar[d]\\
C_{\bar{U}}\ar[r]^{\simeq\,\,\,\,}&W(\bar{U},L)^{ab}\ar[r]&G(L^{un}/K)^{ab} }\]
since the first row is the reciprocity map of class field theory.
\end{proof}
\begin{rem}\label{remark-WK-WU}
Let $K_{\bar{U}}/K$ be the maximal sub-extension of $\overline{K}/K$ unramified over $\bar{U}$. The map $$\underleftarrow{lim}\,\varphi_{\bar{U}}:\underleftarrow{lim}\,\pi_1(\bar{U}_W,p_{\bar{U}})=W(\bar{U},q_{\bar{U}})\longrightarrow \underleftarrow{lim}\,\pi_1(\bar{U}_{et},q_{\bar{U}})=G(K_{\bar{U}}/K)$$
sits in the following commutative square
\[ \xymatrix{
W_{K}\ar[d]\ar[r]&G_K\ar[d]\\
W(\bar{U},q_{\bar{U}})\ar[r]&G(K_{\bar{U}}/K) }\]
\end{rem}
\begin{thm}\label{thm-existence-proof}
The Weil-\'etale topos $\bar{X}_W$ satisfies Axioms $(1)-(9)$ of \cite{Fund-group-I} Section 5.2.
\end{thm}
\begin{proof}
Recall from  \cite{Fund-group-I} Section 5.2 the following expected properties of the conjectural Lichtenbaum topos.
\begin{enumerate}
\item \emph{There is a morphism $\gamma:\bar{X}_{W}\rightarrow \bar{X}_{et}$.}

\item \emph{The topos $\bar{X}_{W}$ is defined over $\mathcal{T}$. The structure map $\bar{X}_{W}\rightarrow\mathcal{T}$
is connected locally connected and $\bar{X}_{W}$ has a $\mathcal{T}$-point $p$. For any connected étale
$\bar{X}$-scheme $\bar{U}$, the object $\gamma^*\bar{U}$ of $\bar{X}_{W}$ is connected over $\mathcal{T}$.}

\item \emph{There is a canonical isomorphism $r_{\bar{U}}:C_{\bar{U}}\simeq\pi_1(\bar{U}_{W})^{ab}$
such that the composition
$$C_{\bar{U}}\simeq\pi_1(\bar{U}_{W})^{ab}\rightarrow\pi_1(\bar{U}_{et})^{ab}$$
is the reciprocity law of class field theory, where the second morphism is induced by $\gamma$.}

\item \emph{The isomorphism $r_{\bar{U}}$ is covariantly
functorial for any map $\bar{V}\rightarrow\bar{U}$ of connected \'etale $\bar{X}$-schemes.}

\item \emph{For any Galois étale cover $\bar{V}\rightarrow\bar{U}$ of étale $\bar{X}$-schemes, the
conjugation action on $\pi_1(\bar{V}_{W})^{ab}$ corresponds to the Galois action on $C_{\bar{V}}$}.

\item \emph{The isomorphism $r_{\bar{U}}$ is contravariantly functorial for
an étale cover.}

\item \emph{For any closed point $v$ of $\bar{X}$, one has a pull-back of
topoi:
\[ \xymatrix{
B_{W_{k(v)}}\ar[d]_{i_v}\ar[r]^{\alpha_v}&B^{sm}_{G_{k(v)}}\ar[d]_{u_v}\\
\bar{X}_{W}\ar[r]^{\gamma}&\bar{X}_{et} }\]
}

\item \emph{For any closed point $w$ of a connected étale $\bar{X}$-scheme $\bar{U}$, the composition
$$B_{W_{k(w)}}\longrightarrow \bar{U}_{W}\longrightarrow B_{C_{\bar{U}}}$$
is the morphism of classifying topoi induced by the canonical
morphism of topological groups ${W_{k(w)}}\rightarrow{C_{\bar{U}}}$.}

\item \emph{For any \'etale $\bar{X}$-scheme $\bar{U}$, one has $H^n(\bar{U}_{W},\widetilde{\mathbb{R}})=0$ for any $n\geq2$.}\\
\end{enumerate}

Indeed, Axiom (1) is given by Corollary \ref{cor-morphism-gamma-WE-et} and Axiom (2) is given by Theorem \ref{thm-big-WEfundgrp} (i) and (ii). Axiom (3) is given by Corollary \ref{cor-pi1-WE-ab=CU} and Corollary \ref{cor-WEtopo-CFT}. Axioms (4), (5), (6) follow from the usual functorial properties of the Weil group (see Remark \ref{remark-WK-WU}). Axiom (7) is given by Theorem \ref{thm-pull-back}. Axiom (8) follows immediately from the description of the morphisms
$$\bar{U}_W\rightarrow B_{W(\bar{U},L)}\rightarrow B_{C_{\bar{U}}}\mbox{ and }i_v:B_{W_{k(v)}}\rightarrow\bar{U}_W$$
in terms of morphisms of left exact sites (see (\ref{morphism-sites-from-UL-to-BW(UL)}) and Theorem \ref{thm-pull-back} respectively).
Finally, Axiom (9) is given by Proposition \ref{prop-R-cohomology}.
\end{proof}

We denote by $\varphi:X_{W}\rightarrow \bar{X}_{W}$ the natural open embedding, and by $H_c^n(X_{W},\mathcal{A}):=H^n(\bar{X}_{W},\varphi_!\mathcal{A})$ the cohomology with compact support with coefficients in the abelian sheaf $\mathcal{A}$.
\begin{cor}\label{thm-Lichtenbaum-conj}\textbf{\emph{(Lichtenbaum's formalism)}}
Assume that $F$ is totally imaginary. We denote by $\tau_{\leq2}R\gamma_*$ the truncated functor of the total derived functor $R\gamma_*$. Then one has:
\begin{itemize}
\item $\mathbb{H}^n(\bar{X}_{et},\tau_{\leq2}R\gamma_*(\varphi_!\mathbb{Z}))$ is finitely generated and zero for $n\geq4$.
\item The canonical map $$\mathbb{H}^n(\bar{X}_{et},\tau_{\leq2}R\gamma_*(\varphi_!\mathbb{Z}))\otimes\mathbb{R}\longrightarrow H_c^n(X_{W},\widetilde{\mathbb{R}})$$
is an isomorphism for any $n\geq0$.
\item There exists a fundamental class $\theta\in{H}^1(\bar{X}_{W},\widetilde{\mathbb{R}})$. The complex of finite dimensional vector spaces
    $$...\rightarrow{H}_c^{n-1}(X_{W},\widetilde{\mathbb{R}})\rightarrow {H}_c^n(X_{W},\widetilde{\mathbb{R}})\rightarrow {H}_c^{n+1}(X_{W},\widetilde{\mathbb{R}})\rightarrow...$$
    defined by cup product with $\theta$, is acyclic.
\item The vanishing order of the Dedekind zeta function $\zeta_F(s)$ at $s=0$ is given by
 $$\textsl{ord}_{s=0}\zeta_F(s)=\sum_{n\geq0}(-1)^n\,n\,\textsl{dim}_{\mathbb{R}}\,H_c^n(X_{W},\widetilde{\mathbb{R}})$$
\item The leading term coefficient $\zeta^*_F(s)$ at $s=0$ is given by the Lichtenbaum Euler characteristic:
$$\zeta^*_F(s)=\pm\prod_{n\geq0}|\mathbb{H}^n(\bar{X}_{et},\tau_{\leq2}R\gamma_*(\varphi_!\mathbb{Z}))_{\textsl{tors}}|^{(-1)^n}/
\textsl{det}(H_c^n(X_{L},\widetilde{\mathbb{R}}),\theta,B^*)$$
where $B^n$ is a basis of $\mathbb{H}^n(\bar{X}_{et},\tau_{\leq2}R\gamma_*(\varphi_!\mathbb{Z}))/\textsl{tors}$.
\end{itemize}
\end{cor}

\begin{proof}
By \cite{Fund-group-I} Theorem 6.3, this follows from Theorem \ref{thm-existence-proof}.
\end{proof}

\end{document}